%% file: pr_paper_new.tex
\documentclass[11pt]{article}

\usepackage{jcd}
\usepackage[numbers]{natbib}
\usepackage{hyperref}
\usepackage[top=2.54cm,left=2.54cm,right=2.54cm,bottom=2.54cm]{geometry}
\usepackage[]{algorithm2e}
\usepackage{overpic}

\newcommand{\init}{^{\rm init}}

\newcommand{\subopt}{_\star}

\renewcommand{\sign}{\text{sgn}}

\newcommand{\sphere}{\mathbb{S}}
\newcommand{\equivdist}{\stackrel{{\rm dist}}{=}}

\newcommand{\noise}{\xi}
\newcommand{\quant}{\mathsf{quant}}
\newcommand{\quantile}{\quant}
\newcommand{\median}{\mathsf{med}}

\newcommand{\cg}{^{\rm H}} 

\newcommand{\proj}{P}
\newcommand{\subgauss}{\sigma} 

\newcommand{\gap}{\mathsf{gap}}

\newcommand{\stabfunc}{\kappa_{\rm st}}

\newcommand{\probbig}{p_0}
\newcommand{\pfail}{p_{\rm fail}}
\newcommand{\stabconst}{\lambda}
\newcommand{\Lipconst}{L}
\newcommand{\lipconst}{\Lipconst}

\newcommand{\Real}{{\rm Re}}
\newcommand{\Imag}{{\rm Im}}

\newcommand{\stepsize}{\alpha}
\newcommand{\event}{\mc{E}}
\newcommand{\errorterm}{\nu}

\newtheorem{condition}{Condition}
\newtheorem{model}{Model}

\newcommand{\indset}{\mc{I}}
\newcommand{\outliers}{\mc{I}^{\rm out}}
\newcommand{\inliers}{\mc{I}^{\rm in}}
\newcommand{\inliersin}{\mc{I}^{\rm in}_{\rm sel}}

\newcommand{\outliersin}{\mc{I}^{\rm out}_{\rm sel}}

\newcommand{\qfail}{q_{\rm fail}}

\newcommand{\selected}{\mc{I}_{\rm sel}}
\newcommand{\yesselected}{\mc{I}_{-\epsilon}}
\newcommand{\maybeselected}{\mc{I}_{+\epsilon}}

\newcommand{\imagunit}{\iota}

\newcommand{\includelong}[1]{}

\begin{document}

\title{Solving (most) of a set of quadratic equalities: \\
  Composite optimization for robust phase retrieval}
\author{John C.\ Duchi ~~~~ Feng Ruan \\
Stanford University}

\maketitle

\begin{abstract}
  We develop procedures, based on minimization of the composition $f(x) =
  h(c(x))$ of a convex function $h$ and smooth function $c$, for solving
  random collections of quadratic equalities, applying our methodology to
  phase retrieval problems.  We show that the prox-linear
  algorithm we develop can solve phase retrieval problems---even
  with adversarially faulty measurements---with high probability
  as soon as the number of
  measurements $m$ is a constant factor larger than the dimension $n$ of the
  signal to be recovered. The algorithm requires essentially no tuning---it
  consists of solving a sequence of convex problems---and it is
  implementable without any particular assumptions on the measurements
  taken. We provide substantial experiments investigating our methods,
  indicating the practical effectiveness of the procedures and showing that
  they succeed with high probability as soon as $m / n \ge 2$ when
  the signal is real-valued.
\end{abstract}


\input{introduction-complex}


\input{composite-convergence-complex}


\input{no-noise-retrieval-complex}


\input{noise-retrieval-short}


\input{optimization}


\input{experiments}


\paragraph{Acknowledgments}
We thank Dima Drusvyatskiy and Courtney Paquette for a number of
inspirational and motivating conversations. JCD and FR were partially
supported by NSF-CAREER Award 1553086 and
the Toyota Research Institute.  FR was additionally supported by a
Stanford Graduate Fellowship.

\appendix

\input{appendix-no-noise-complex}

\input{appendix-noise}
\input{appendix-generic}

\bibliographystyle{abbrvnat}
\bibliography{bib}

\end{document}

%% file: introduction-complex.tex

\section{Introduction}

We wish to solve the following problem: we have a set of $m$ vectors $a_i
\in \C^n$ and nonnegative scalars $b_i \in \R_+$, $i = 1, \ldots, m$, and
wish to find a vector $x \in \C^n$ such that
\begin{equation}
  \label{eqn:feasibility}
  b_i = |\<a_i, x\>|^2
  ~~ \mbox{for~most~~} i \in \{1, \ldots, m\}.
\end{equation}
As stated, this is a combinatorial problem that is, in the worst case,
NP-hard~\cite{Matthew14}. Yet it naturally arises in a number of
real-world situations, including phase
retrieval~\cite{Fienup78,Fienup82,GerchbergSa72}, in which one receives
measurements of the form
\begin{equation*}
  b_i = |\<a_i, x\subopt\>|^2
\end{equation*}
for known measurement vectors $a_i \in \C^n$, while $x\subopt \in \C^n$ is
unknown.  The problem in phase retrieval arises due to limitations of
optical sensors, where one illuminates an object $x\subopt$, which yields
diffraction patter $A x\subopt$, but sensors may measure only the amplitudes $b
= |A x\subopt|^2$, where $|\cdot|^2$ denotes the elementwise
squared-magnitude~\cite{SchechtmanElCoChMiSe15}.
In the case in which some measurements may
be corrupted, the problem is even more challenging.

An alternative objective for the problem~\eqref{eqn:feasibility} is an exact
penalty formulation~\cite{HiriartUrrutyLe93}, which replaces the equality
constraint $b_i = |\<a_i, x\>|^2$ with a non-differentiable cost measuring
the error $b_i - |\<a_i, x\>|^2$, yielding the formulation
\begin{equation}
  \label{eqn:optimization-problem}
  \minimize_x ~ f(x) \defeq \frac{1}{m} \sum_{i = 1}^m
  \left||\<a_i, x\>|^2 - b_i\right|
  = \frac{1}{m} \lone{|Ax|^2 - b}.
\end{equation}
This objective is a natural replacement of the equality constrained
problem~\eqref{eqn:feasibility}, and the $\ell_1$-loss handles gross errors
in the measurements $b_i$ in a relatively benign way (as is well-known in
the statistics and optimization literature on $\ell_1$-based losses and
median estimators). Moreover, in the case when $b_i = |\<a_i, x\subopt\>|^2$
for all $i$, it is clear that taking $\imagunit = \sqrt{-1}$ to be the
imaginary unit, then the set $\{e^{\imagunit\theta} x\subopt \mid \theta\in
\openright{0}{2\pi}\}$ globally minimizes $f(x)$, though it is only possible
to recover $x\subopt$ up to its phase (or sign flip in the real
case). Cand\`{e}s, Strohmer and Voroninski~\cite{CandesStVo13} and Eldar and
Mendelson~\cite{EldarMe14}, as well as results we discuss later in the
paper, show (roughly) that $f(x)$ stably identifies $x\subopt$, in that $f(x)$
grows very quickly as $\dist(x, X\subopt) = \inf\{\ltwo{x-x\subopt} \mid x\subopt \in
X\subopt\}$, where $X\subopt$ denotes the global minimum of $f$, grows. The
objective is, unfortunately, non-smooth, non-convex---not even locally
convex near $x\subopt$, as is clear in the special case when $x\in \R$ and
$f(x) = |x^2 - 1|$, so that a local analysis based on convexity is
impossible---and at least $f(x)$ \emph{a priori} seems difficult to
minimize.

Nonetheless, the objective~\eqref{eqn:optimization-problem} enjoys a number
of structural properties that, as we explore below, make solving
problem~\eqref{eqn:feasibility} tractable as long as the measurement vectors
$a_i$ are sufficiently random. In particular, we can write $f$ as the
composition $f(x) = h(c(x))$ of a convex function $h$ and smooth function
$c$, a structure known in the optimization literature to be amenable to
efficient algorithms~\cite{FletcherWa80, Burke85, DrusvyatskiyIoLe16}. This
compositional structure lends itself nicely to the \emph{prox-linear}
algorithm, a variant of the Gauss-Newton procedure, which we describe
briefly here for the real case.  The composite optimization problem, which
Fletcher and Watson \cite{FletcherWa80} originally develop (see
also~\cite{Bertsekas77, Polyak79}) and a number of
researchers~\cite{Burke85, BurkeFe95, DrusvyatskiyIoLe16, DrusvyatskiyPa16}
have studied further, is to minimize
\begin{equation}
  \label{eqn:composite-optimization}
  \minimize_x~ f(x) \defeq h(c(x))
  ~~ \subjectto x \in X
\end{equation}
where the function $h : \R^m \to \R$ is convex, $c : \R^n \to \R^m$ is
smooth, and $X$ is a convex set. Extended to the complex
case, this general form
encompasses our objective~\eqref{eqn:optimization-problem}, where
we take $h(z) = \frac{1}{m} \lone{z}$ and $c(x) =
[|\<a_i, x\>|^2 - b_i]_{i = 1}^m$.  Using the common idea of most optimization
schemes---trust region, gradient descent, Newton's method---to build a
simpler to optimize local model of the objective and repeatedly minimize
this model, we can replace $h(c(x))$ in
problem~\eqref{eqn:composite-optimization} by linearizing only $c$. This
immediately gives a convex surrogate and leads to the prox-linear algorithm
developed by Burke and Ferris~\cite{BurkeFe95,Burke85}, among
others~\cite{DrusvyatskiyIoLe16, DrusvyatskiyLe16, DrusvyatskiyPa16}.
Fixing $x \in \R^n$, for any $y \in \R^n$ we define the local
``linearization'' of $f$ at $x$ by
\begin{equation}
  \label{eqn:local-linearization}
  f_x(y) \defeq h\left(c(x) + \nabla c(x)^T  (y - x)\right),
\end{equation}
where $\nabla c(x) \in \R^{n \times m}$ denotes the Jacobian transpose of
$c$ at $x$. This function is evidently convex in $y$, and the prox-linear
algorithm proceeds iteratively $x_1, x_2, \ldots$ by minimizing the
regularized models
\begin{equation}
  \label{eqn:prox-iteration}
  x_{k + 1} = \argmin_{x \in X}
  \left\{f_{x_k}(x) + \frac{1}{2 \stepsize_k} \ltwo{x - x_k}^2
  \right\},
\end{equation}
where $\stepsize_k > 0$ is a stepsize. If $h$ is $L$-Lipschitz and $\nabla
c$ is $\beta$-Lipschitz, then choosing any stepsize $\stepsize \le
\frac{1}{\beta L}$ guarantees that the method~\eqref{eqn:prox-iteration} is
a descent method and finds approximate stationary points of the
problem~\eqref{eqn:composite-optimization}~\cite{DrusvyatskiyLe16,
  DrusvyatskiyIoLe16}. The case in which $c : \C^n \to \R^m$ requires a bit
more elaboration based on the Wirtinger calculus, which we address later.

We briefly summarize our main contribution as follows.  Broadly, this work
provides a general method for robust non-convex modeling, focusing
carefully on the problem~\eqref{eqn:optimization-problem}; our work carries
on a line of work identifying statistical scenarios that nominally
yield non-convex optimization problems yet admit computationally efficient
estimation and optimization procedures~\cite{KeshavanMoOh10, LohWa13,
  ChenCa15, BalakrishnanWaYu17}.
In this direction, our work develops analytic and statistical tools
to analyze a collection of optimization and modeling approaches
beyond gradient-based (and Riemannian gradient-based) procedures,
using both non-smooth and non-convex models while leveraging statistical
structure. More precisely, we show how to apply
prox-linear method~\eqref{eqn:prox-iteration} to any measurement matrix $A$
with no tuning parameters except that the stepsize satisfies $\stepsize \le
(\frac{1}{m} \opnorm{A\cg A})^{-1}$.  Each iteration requires solving a QP in
$n$ variables, which is efficiently solvable using standard convex
programming approaches. We show that---with extremely high probability under
appropriate random measurement models---our prox-linear method exhibits 
local quadratic convergence to the signal as soon as the number of 
measurements $m/n$ is greater than
some numerical constant, meaning we must solve only $\log_2 \log_2
\frac{1}{\epsilon}$ such convex problems to find an estimate $\what{x}$ of
$x\subopt$ such that $\dist (x, X\subopt) \le \epsilon$. In practice, this is 5 convex quadratic programs.
Our procedure applies both in the noiseless setting and when a (constant but
random) fraction of the measurements are even adversarially corrupted.

\subsection{Related work and approaches to phase retrieval}

Our work should be viewed in the context of the recent and successful
collection of work on phase retrieval. A natural strategy for
problem~\eqref{eqn:feasibility}, when we wish to find $x$ satisfying
$|\<a_i, x\>|^2 = b_i$ for all $i \in [m]$, is to lift the problem into
a semidefinite program (SDP) by setting $X = xx\cg$, relaxing the rank one
constraint, and solving
\begin{equation*}
  \minimize_X ~ \tr(X) ~~ \subjectto
  ~ X \succeq 0,
  ~ \tr(X a_i a_i\cg ) = b_i.
\end{equation*}
This is the approach that a number of convex
approaches to phase retrieval
take~\cite{ChaiMoPa11,CandesElStVo13,CandesStVo13,CandesLi14,
  WaldspurgerDaMa15}. Th resulting SDP is computationally
challenging for large $n$, as it requires storing and manipulating an $n
\times n$ matrix variable. Moreover, computation times to achieve
$\epsilon$-accurate solutions to this problem generally scale on the order of 
$n^3 / \mathsf{poly}(\epsilon)$, where $\mathsf{poly}(\epsilon)$
denotes a polynomial in $\epsilon$.

These difficulties have led a number of researchers to consider non-convex
approaches to the phase retrieval problem that---as we do---maintain only a
vector $x \in \C^n$, rather than forming a full matrix $X \in \C^{n
  \times n}$.  We necessarily give only a partial overview,
focusing on recent work on provably convergent schemes. Early work in
computational approaches to phase retrieval is based on (non-convex)
alternating projection approaches, notably those by Gerchberg and
Saxton~\cite{GerchbergSa72} and Fineup~\cite{Fienup82}.  Motivated by the
challenges of convex approaches and the success of alternating
minimization~\cite{GerchbergSa72,Fienup82}, Netrapalli et
al.~\cite{NetrapalliJaSa13} develop an algorithm (AltMinPhase) that
alternates between minimizing $\ltwo{Ax - C b}$ in $x$ and in $C$ over
diagonal matrices of phases (signs) with modulus 1. Their algorithm is
elegant, but the analysis requires resampling a new measurement matrix $A$
and measurements $b$ in each iteration.  More recently, Cand\`{e}s et
al.~\cite{CandesLiSo15} develop Wirtinger flow, a gradient-based method that
performs a careful modification of gradient descent on the objective
\begin{equation*}
  F(x) \defeq \frac{1}{2m} \sum_{i = 1}^m
  (|\langle a_i, x\rangle|^2 - b_i)^2,
\end{equation*}
where $x \in \C^n$ may be complex. Wang et al.~\cite{WangGiEl16} build on
this work by attacking a modification of this objective, showing how to
perform a generalized descent method on
\begin{equation*}
  F(x) \defeq \frac{1}{2m} \sum_{i = 1}^m
  \left(|\langle a_i, x \rangle| - \sqrt{b_i}\right)^2,
\end{equation*}
and providing arguments for the convergence of their method. Wang et
al.\ achieve striking empirical results when the measurements and signals
are real-valued, achieving better than 50\% perfect signal recovery when the
measurement ratio $m/n = 2$, which is essentially at the threshold for
injectivity of the real-valued measurements $b = (Ax\subopt)^2$. Zhang et
al.~\cite{ZhangChLi16} also study a variant of Wirtinger flow based on
median estimates that handles some outliers.  Unfortunately, these
procedures rely fairly strongly on Gaussianity assumptions, and their
gradient descent approaches require subsampling schemes (to select ``good''
terms in the sum); these procedures have parameters chosen carefully to
reflect Gaussianity in the measurement matrices $A$, and it is not always
clear how to extend them to non-Gaussian measurements.

\subsection{Our contributions and outline}

In this paper, we focus on prox-linear methods, the
iterations~\eqref{eqn:prox-iteration} for the non-smooth non-convex
problem~\eqref{eqn:optimization-problem}. In addition to being (to us at
least) aesthetically pleasing, as we minimize the natural
objective~\eqref{eqn:optimization-problem}, our approach yields a number
of theoretical and practical benefits.

In the literature on signal recovery from phaseless measurements,
\emph{stability} of the reconstruction of a signal is of paramount
importance. To solve the phase retrieval problem at all, one requires
intectivity of the measurements $b = |Ax|^2$, which for real $A \in \R^{m
  \times n}$ in general position necessitates $m \ge 2n - 1$ and for complex
$A \in \C^{m \times n}$ in general position necessitates $m \ge 4n - 2$
(cf.~\cite{BalanCaEd06}).  Stability makes this injectivity more robust.
Consider the real-valued case first. Eldar and Mendelson~\cite{EldarMe14}
say that a measurement matrix $A \in \R^{m \times n}$ is $\stabconst \ge 0$
stable if
\begin{subequations}
  \label{eqn:stability}
  \begin{equation}
    \label{eqn:stability-real}
    \lone{(Ax)^2 - (Ay)^2} \ge
    \lambda \ltwo{x - y} \ltwo{x + y}
    ~~ \mbox{for~all~} x, y \in \R^n.
  \end{equation}
  Such conditions, which hold with high probability for suitable designs $A$,
  are also common in semidefinite relaxation approaches to phase retrieval;
  cf.\ Cand\`{e}s et al.~\cite[Lemma 3.2]{CandesStVo13}.  (See also the
  paper~\cite{BandeiraCaMiNe14}.) This condition means that distant signals
  $x, x' \in \R^n$ cannot be confused in the measurement domain $\{(Ay)^2 \mid
  y \in \R^n\} \subset \R^m_+$ because $A$ does a good job of separating them;
  the more stable a measurement matrix, the ``easier'' the recovery problem
  should be.  In the case in which $x$ is complex, the stability
  condition~\eqref{eqn:stability-real} becomes
  \begin{equation}
    \label{eqn:stability-complex}
    \lone{|Ax|^2 - |Ay|^2} \ge \lambda \inf_\theta
    \ltwos{x - e^{\imagunit \theta} y}
    \cdot \sup_\theta \ltwos{x - e^{\imagunit \theta} y}
    ~~ \mbox{for~all~} x, y \in \C^n,
  \end{equation}
\end{subequations}
a slightly stronger condition.  We provide stability guarantees for both
situations for general classes of random matrices by adapting
Mendelson's ``small ball'' techniques~\cite{Mendelson14}
(Sec.~\ref{sec:no-noise-stability}). Most literature on non-convex
approaches to phase retrieval requires such a stability condition---and
usually more because of the quadratic objectives often used---to guarantee
signal recovery.  In contrast, our procedure requires essentially
\emph{only} the stability condition~\eqref{eqn:stability}, a mild bound on
the operator norm $\frac{1}{m} \opnorm{A}^2$, and an initialization within
some constant factor of $\ltwo{x\subopt}$ to guarantee both fast convergence
and exact signal recovery.

With this in mind, in Section~\ref{sec:composite} we develop purely
optimization-based deterministic results, which build off of
classical results on composite optimization, which rely on the stability
condition~\eqref{eqn:stability}.  By identifying the conditions required for
fast convergence and recovery, we can then spend the remainder of the paper
showing how various measurement models guarantee sufficient conditions for
our convergence results. In particular, in Section~\ref{sec:no-noise}, we
show how a number of sensing matrices $A$ suffice to guarantee convergence
and signal recovery in the noiseless setting, that is, when $b = |A
x\subopt|^2$. In Section~\ref{sec:noise}, we extend these results to the case
when a constant fraction of the measurements $b_i$ may be arbitrarily
corrupted, showing that stability and a somewhat stronger condition on
$\opnorm{A}$ are still sufficient to guarantee signal recovery; again, these
results hold for our basic algorithm with no tuning parameters.

In the final sections of the paper, we provide a substantial empirical
evaluation of our proposed algorithms. While our method in principle
requires no tuning---it solves a sequence of explicit \emph{convex}
problems---there is some art in developing efficient methods
for the solution of the sequence of convex optimization problems we
solve. In Section~\ref{sec:optimization-methods}, we describe these
implementation details, and in Section~\ref{sec:experiments} we provide
experimental evidence of the success of our proposed approach.  In
reasonably high-dimensional settings ($n \ge 1000$), with real-valued random 
Gaussian measurements our method achieves perfect signal recovery in 
about 80-90\% of cases even when $m/n = 2$. The method also
handles outlying measurements well, substantially improving state-of-the-art
performance, and we give applications with measurement matrices that
demonstrably fail all of our conditions but for which the method is still
straightforward to implement and empirically successful.

\paragraph{Notation}
We collect our common notation here.  We let $\norm{\cdot}$ and
$\ltwo{\cdot}$ denote the usual vector $\ell_2$-norm, and for a matrix $A
\in \C^{m \times n}$, $\opnorm{A}$ denotes its $\ell_2$-operator norm. The
notation $A\cg$ means the Hermitian conjugate (conjugate transpose) of
$A \in \C^{m \times n}$. For $a \in \C$, $\Real(a)$ denotes its real
part and $\Imag(a)$ its imaginary part.
We take $\<\cdot, \cdot\>$ to be the standard inner product on whatever
space it applies; for $u, v \in \C^n$, this is
$\<u, v\> = u\cg v$, while for $A, B \in \C^{m \times n}$, this is
$\<A, B\> = \tr(A\cg B)$.
Let
$\quant_\alpha(\{c_i\})$ denote the $\alpha$-quantile of a vector $c \in
\R^m$, that is, if $c_{(1)} \le c_{(2)} \le \cdots \le c_{(m)}$, the
$\alpha$th quantile linearly interpolates $c_{(\floor{m \alpha})}$ and
$c_{(\ceil{m \alpha})}$. For a random variable $X$, $\quant_\alpha(X)$
denotes its $\alpha$th quantile.

%% file: composite-convergence-complex.tex

\section{Composite Optimization, Algorithm and Convergence Analysis}
\label{sec:composite}

\begin{algorithm}[t!]
  \caption{\label{alg:prox-linear}
    Prox-linear algorithm for problem~\eqref{eqn:composite-optimization}}
  \KwData{Initializer $x_0$, stepsize $\stepsize$, tolerance
    $\epsilon > 0$, convex
    $h : \R^m \to \R$, suitably smooth $c : \C^n \to \R^m$}
  \textbf{set} $k = 0$ \\
  \Repeat{$\ltwo{x_{k-1} - x_k} \le \stepsize \epsilon$}{
    $x_{k+1} = \argmin_{x \in X}
    \{f_{x_k}(x) + \frac{1}{2\stepsize} \ltwo{x-x_k}^2\}$
    where $f_x(y) = h(c(x) + \Real(\nabla c(x)\cg (y - x)))$\\
    $k = k + 1$
  }
  \Return $x_k$
\end{algorithm}

We begin our development by providing convergence guarantees---under
appropriate conditions---for the prox-linear algorithm
(the iteration~\eqref{eqn:prox-iteration}) applied to the
composite optimization problem~\eqref{eqn:composite-optimization},
which we recall is
to
\begin{equation*}
  \minimize_x~ f(x) \defeq h(c(x))
  ~~ \subjectto x \in X,
\end{equation*}
where $h : \R^m \to \R$ is convex and $c : \C^n \to \R^m$ is appropriately
smooth. In our context, as $c$ is a real-valued complex function, it cannot
have an ordinary complex derivative, so some care is required; we use the
Wirtinger Calculus, also known as the $\C\R$-calculus, referring the
interested reader to the survey of \citet{KreutzDelgado09} for more (see
especially~\cite[Eq.~(31)]{KreutzDelgado09}). In brief, however, because $c$
is real-valued, we let $\nabla c(x)$ denote the Hermitian conjugate of
the Jacobian of $c$, which allows treatment $c$ as a
mapping from $\R^{2n}$ to $\R$, so that $\nabla c(x) \in \C^{n \times m}$
satisfies
\begin{equation*}
  c(y) = c(x) + \Real(\nabla c(x)\cg (y - x)) + O(\norm{y - x}^2)
\end{equation*}
as $y \to x$.  We summarize the procedure in Alg.~\ref{alg:prox-linear} for
further reference. In our application to quadratic constraints and phase
retrieval, $h(z) = \frac{1}{m} \lone{z}$ and $c(x) = |Ax|^2 - b$, so that
the iteration~\eqref{eqn:prox-iteration} is the solution of a quadratic
problem. (We describe this in more detail in the next section.)

A number of researchers have studied convergence and stopping conditions for
Algorithm~\ref{alg:prox-linear}, showing that it
converges to stationary points~\cite{Burke85}, as well as demonstrating that
the stopping condition $\ltwo{x_k - x_{k-1}} \le \stepsize \epsilon$ holds
after $O(\epsilon^{-2})$ iterations and guarantees
approximate stationarity~\cite{DrusvyatskiyLe16, DrusvyatskiyIoLe16}.
Algorithm~\ref{alg:prox-linear} and the iteration~\eqref{eqn:prox-iteration}
sometimes enjoy fast (local) convergence rates as well. To describe this
phenomenon, define the distance function $\dist(x, S) \defeq \inf_y
\{\ltwo{x-y} \mid y\in S\}$.  We say that $h$ has \emph{weak sharp minima}
if it grows linearly away from its minima, meaning $h(z) \ge \inf_z h(z) +
\lambda \dist(z, \argmin h)$ for some $\lambda > 0$. Under this condition
(with an additional transversality condition between $c$ and $\argmin h$),
Burke and Ferris~\cite{BurkeFe95} show that convergence of the prox-linear
algorithm near points in $X\subopt \defeq \{x : c(x) \in \argmin h\}$ is
quadratic, because the model~\eqref{eqn:local-linearization} of $f$ is
quadratically good, but $h(c(x))$ grows linearly away from $X\subopt$. These
assumptions can be weakened to growth of $h \circ c$ along its minimizing
set~\cite[Thm.~7.2]{DrusvyatskiyLe16}. We build from this elegant
development---though our problems do not precisely satisfy the weak sharp
minima conditions because of outliers---to show how the prox-linear
algorithm provides an effective, easy-to-implement, and elegant approach to
problems involving solution of quadratic equalities, specifically focusing
on phase retrieval.

\subsection{Quadratic convergence and the prox-linear
  method for phase retrieval}

We turn now to an analysis of the prox-linear algorithm for
phase retrieval problems, providing conditions on the function $f$
sufficient for quadratic convergence.
We introduce two conditions on the function
$f(x)$ and its linearized form $f_x(y)$ that suffice for this desidaratum;
as we show in the sequel,
these conditions hold with extremely high probability for a number of random
measurement models. These conditions are the keystones of our
analysis of the (robust) phase retrieval problem. 

As motivation for our first condition, consider the phase retrieval
problem. If the measurement matrix $A$ satisfies
conditions~\eqref{eqn:stability} and the measurements $b_i$ are noiseless
with $b = |Ax\subopt|^2$, then for $f(x) = \frac{1}{m} \lones{|Ax|^2 - b}$ we
have $f(x) - f(x\subopt) \ge \stabconst \dist(x, X\subopt)\ltwo{x\subopt}$ for the
set of signals $X\subopt = \{ e^{\imagunit \theta} x\subopt \mid \theta \in\R\}$.
When the measurements have noise or outliers, this may still hold, prompting
us to define the following
\begin{condition}
  \label{condition:stability}
  There exists $\stabconst > 0$ such that for all $x \in \R^n$ (or $x\in
  \C^n$ in the complex case)
  \begin{equation*}
    f(x) - f(x\subopt) \geq \stabconst \dist(0, X\subopt) \dist(x, X\subopt), 
  \end{equation*}
  where $X\subopt$ denotes the set of global minima of
  $f$.
\end{condition}
\noindent
This condition is a close cousin of Burke and Ferris's sharp minima
condition~\cite{BurkeFe95}, though it does not require that $c(x\subopt) \in
\argmin_z h(z)$; based on their work, it is intuitive that it should prove
useful in establishing fast convergence of the prox-linear algorithm. The
second condition, which is essentially automatically satisfied for the
linear approximation~\eqref{eqn:local-linearization}, is a requirement that
the linearized function $f_x(y)$ is quadratically close to $f(y)$.
\begin{condition}
  \label{condition:approximation}
  There exists $\Lipconst < \infty$ such that for all $x, y\in \R^n$ (or $x, y \in \C^n$ in 
  the complex case)
  \begin{equation*}
    |f(y) - f_x(y)| \leq \frac{\Lipconst}{2} \ltwo{x-y}^2.
  \end{equation*}
\end{condition}

Locally, Condition~\ref{condition:approximation} holds for any composition
$f(x) = h(c(x))$ of a convex $h$ with smooth $c$, but the phase retrieval
objective~\eqref{eqn:optimization-problem} satisfies the bound globally.
Indeed, for $a, x, y \in \C^n$ we have
\begin{equation*}
  |\<a, y\>|^2
  = |\<a, x\>|^2 + 2 \Real(\<x, a\> \<a, y - x\>) +
  |\<a, y - x\>|^2,
\end{equation*}
so the linearization~\eqref{eqn:local-linearization} of
$f$ around $x \in \C^n$ is
\begin{equation}
  \label{eqn:linear-form-pr}
  f_x(y) = \frac{1}{m} \sum_{i=1}^m \left||\<a_i, x\>|^2 - b_i + 
  2 \Real (\<x, a_i\>\<a_i, y-x\>)\right|.
\end{equation}
Letting $A = [a_1 ~ \cdots ~ a_m]\cg$ denote the
measurement matrix, then using the preceding expansion of $|\<a, y\>|^2$, we
have immediately by the triangle inequality that
\begin{align*}
  -(x - y)\cg \left(\frac{1}{m} A\cg A\right) (x - y)
  + f_x(y)
  \le \frac{1}{m} \sum_{i = 1}^m \,& ||\<a_i, y\>|^2 - b_i| = f(y) \\
  & \le
  f_x(y) + (x - y)\cg \left(\frac{1}{m} A\cg A\right) (x - y).
\end{align*}
That is, Condition~\ref{condition:approximation} holds with
$\lipconst = 2 \opnorms{\frac{1}{m} A\cg A}$:
\begin{equation}
  \label{eqn:approximation-opnorm}
  |f(y) - f_x(y)|
  \le \frac{1}{m}\opnorm{A\cg A}
  \ltwo{x - y}^2.
\end{equation}

Given Conditions~\ref{condition:stability}
and~\ref{condition:approximation}, we turn to convergence guarantees for
the prox-linear Algorithm~\ref{alg:prox-linear}, which
in our case requires solving a sequence of convex quadratic
programs. An
implementation of
Alg.~\ref{alg:prox-linear} that
solves iteration~\eqref{eqn:prox-iteration} exactly may be computationally
challenging. Thus, we allow inaccuracy in the solutions, assuming there exists
a sequence of additive accuracy parameters $\epsilon_k \ge 0$ such that the
iterates $x_k$ satisfy
\begin{equation}
  \label{eqn:approximate-minimization}
  f_{x_k}(x_{k + 1})
  + \frac{\Lipconst}{2} \ltwo{x_{k + 1} - x_k}^2
  \le \inf_x \left\{
  f_{x_k}(x)
  + \frac{\Lipconst}{2} \ltwo{x - x_k}^2
  \right\} + \epsilon_k.
\end{equation}
We have the following theorem, 
whose proof we provide in Section~\ref{sec:proof-quadratic-convergence-errors}.

\begin{theorem}
  \label{theorem:quadratic-convergence-errors}
  Let Conditions~\ref{condition:stability} and~\ref{condition:approximation}
  hold.  Assume that in each step of
  Algorithm~\ref{alg:prox-linear}, we solve the
  intermediate optimization problem to accuracy $\epsilon_k$, and define the
  relative error measures $\beta_k = \frac{2 \epsilon_k}{\stabconst
   \dist(0, X\subopt)^2}$. Then
  \begin{equation*}
    \frac{\dist(x_k, X\subopt)}{\dist(0, X\subopt)}
    \le
    \frac{\stabconst}{2 \Lipconst}
    \max\left\{
    \left(\frac{2 \Lipconst}{\stabconst}
    \cdot \frac{\dist(x_0, x\subopt)}{\dist(0, X\subopt)}
    \right)^{2^k},
    \max_{0 \le j < k}
    \left(\frac{4 \Lipconst}{\stabconst}
    \cdot \beta_j\right)^{2^{k - j - 1}}\right\}.
  \end{equation*}
  If $\epsilon_k = 0$ in the solution quality
  inequality~\eqref{eqn:approximate-minimization}, then
  \begin{equation*}
    \frac{\dist(x_k, X\subopt)}{\dist(0, X\subopt)}
    \le \frac{\stabconst}{\Lipconst}
    \left(\frac{\Lipconst}{\stabconst} \cdot
    \frac{\dist(x_0, x\subopt)}{\dist(0, X\subopt)}\right)^{2^k}.
  \end{equation*}
\end{theorem}

Theorem~\ref{theorem:quadratic-convergence-errors} motivates our approach
for the remainder of the paper: we can
guarantee exact, accurate, and fast solutions to the phase retrieval problem
under the three conditions
\begin{enumerate}[1.]
\item Stability (Condition~\ref{condition:stability}),
\item Quadratic approximation (Condition~\ref{condition:approximation}), 
  via an upper bound on $\opnorm{A\cg A}$ and application of
  inequality~\eqref{eqn:approximation-opnorm} and
\item An initializer
  $x_0$ of the iterations that is good enough, meaning that it satisfies the
  constant relative error guarantee $\dist(x_0, X\subopt) \le \dist(0, X\subopt)
  \frac{\stabconst}{\Lipconst}$.
\end{enumerate}
In the coming sections, we show that each of these three conditions holds in
both noiseless measurement models (Section~\ref{sec:no-noise}) and with
adversarially perturbed measurements $b_i$ (Section~\ref{sec:noise}).

Before continuing, we provide a few additional remarks regarding
Theorem~\ref{theorem:quadratic-convergence-errors}. First, if $\epsilon_k$
are very small because we solve the intermediate
steps~\eqref{eqn:prox-iteration} to (near) machine precision, then for all
intents and purposes about five iterations suffice for machine precision
accurate solutions.  Quadratic convergence is also achievable with errors in
inequality~\eqref{eqn:approximate-minimization}; if the minimization
accuracies decrease quickly enough that $\epsilon_k \le 2^{-2^k}$, then we
certainly still have quadratic convergence.  More broadly,
Theorem~\ref{theorem:quadratic-convergence-errors} shows that the accuracy
of solution in iteration $j$ need not be very high to guarantee high
accuracy reconstruction of the signal $x\subopt$; only in the last few
iterations is moderate to high accuracy necessary.  If it is computationally
cheap, it is thus advantageous---as we explore in our experimental work---to
solve early iterations of the prox-linear method inaccurately.

\subsection{Proof of Theorem~\ref{theorem:quadratic-convergence-errors}}
\label{sec:proof-quadratic-convergence-errors}

We prove the result in two steps: we first provide a per-iteration progress
guarantee, and then we use this guarantee to show quadratic convergence. 

Let $x\subopt \in X\subopt$ be any global optimum of the objective $f(x)$. 
The function $x \mapsto f_{x_k}(x) + \frac{\Lipconst}{2} \ltwo{x - x_k}^2$
is $\Lipconst$-strongly convex in $x$. If we define $x_{\subopt,k + 1}$ to
be the exact minimizer of $f_{x_k}(x) + \frac{\Lipconst}{2} \ltwo{x -
  x_k}^2$, the standard optimality conditions for strongly convex
minimization imply
\begin{align*}
  f_{x_k}(x_{k+1}) + \frac{\Lipconst}{2}
  \ltwo{x_{k + 1} - x_k}^2
  & \le
  f_{x_k}(x_{\subopt,k+1}) + \frac{\Lipconst}{2}
  \ltwo{x_{\subopt,k+1} - x_k}^2
  + \epsilon_k \nonumber \\
  & \le 
  f_{x_k}(x\subopt) + \frac{\Lipconst}{2}
  \ltwo{x\subopt - x_k}^2
  - \frac{\Lipconst}{2}
  \ltwo{x\subopt - x_{\subopt,k+1}}^2
  + \epsilon_k.
\end{align*}
Using the approximation Condition~\ref{condition:approximation},
so that $\frac{\Lipconst}{2} \ltwo{x_k - x_{k+1}}^2 + f_{x_k}(x_{k + 1})
\ge f(x_{k + 1})$ and $f_{x_k}(x\subopt) \le f(x\subopt) + \frac{\Lipconst}{2}
\ltwo{x_k - x\subopt}^2$, we have by substituting in the preceding inequality
that
\begin{align*}
  f(x_{k + 1})
  & \le f_{x_k}(x\subopt) + \frac{\Lipconst}{2} \ltwo{x\subopt - x_k}^2
  - \frac{\Lipconst}{2} \ltwo{x\subopt - x_{\subopt,k+1}}^2
  + \epsilon_k \\
  & \le f(x\subopt) + \Lipconst \ltwo{x_k - x\subopt}^2
  - \frac{\Lipconst}{2} \ltwo{x\subopt - x_{\subopt,k+1}}^2
  + \epsilon_k.
\end{align*}
Summarizing, we get for all $x\subopt \in X\subopt$ and $k \in \N$, 
\begin{equation}
  f(x_{k + 1}) - f(x\subopt)
 + \frac{\Lipconst}{2} \ltwo{x\subopt - x_{\subopt,k+1}}^2
   \le \Lipconst \ltwo{x_k - x\subopt}^2 + \epsilon_k 
\end{equation}
By applying the stability Condition~\ref{condition:stability},
we immediately obtain the progress guarantee
\begin{equation}
  \label{eqn:innacurate-step-amount}
  \begin{split}
    \lambda \ltwo{x\subopt} \dist\left(x_{k+1}, X\subopt\right)
    + \frac{\Lipconst}{2}
    \ltwo{x_{\subopt,k+1} - x\subopt}^2
    & \le \Lipconst \ltwo{x_k - x\subopt}^2 + \epsilon_k 
  \end{split}
\end{equation}

We now transform the guarantee~\eqref{eqn:innacurate-step-amount} into one
involving only $x_k$, $x_{k + 1}$, and $x\subopt$, rather than $x_{\subopt,k+1}$,
by bounding the difference between $x_{\subopt,k+1}$ and $x_{k+1}$. The
$\Lipconst$-strong convexity of $f_{x_k}(\cdot) + \frac{\Lipconst}{2}
\ltwo{\cdot - x_k}^2$ implies that
\begin{align*}
  \epsilon_k +
  f_{x_k}(x_{\subopt,k+1})
  + \frac{\Lipconst}{2} \ltwo{x_{\subopt,k+1} - x_k}^2
  & \ge f_{x_k}(x_{k+1})
  + \frac{\Lipconst}{2} \ltwo{x_{k+1} - x_k}^2 \\
  & \ge f_{x_k}(x_{\subopt,k+1}) + \frac{\Lipconst}{2}
  \ltwo{x_{\subopt,k+1} - x_k}^2
  + \frac{\Lipconst}{2} \ltwo{x_{k+1} - x_{\subopt,k+1}}^2,
\end{align*}
whence we obtain $\frac{\Lipconst}{2} \ltwo{x_{\subopt,k+1} - x_{k + 1}}^2 \le
\epsilon_k$.  Using the standard quadratic inequality $\ltwo{x_{k+1} -
  x\subopt}^2 \le 2 \ltwo{x_{k+1} - x_{\subopt,k+1}}^2 + 2\ltwo{x_{\subopt,k+1} -
  x\subopt}^2$, we thus have by
expression~\eqref{eqn:innacurate-step-amount} that, for all $x\subopt \in X\subopt$, 
\begin{equation}
  \begin{split}
     \lambda \ltwo{x\subopt} \dist\left(x_{k+1}, X\subopt\right)
    + \frac{\Lipconst}{4}
    \ltwo{x_{k+1} - x\subopt}^2
    & \le \Lipconst \ltwo{x_k - x\subopt}^2 + 2\epsilon_k  
    \end{split}
  \label{eqn:quadratic-inaccurate}
\end{equation}
Now, taking infimum over $x\subopt \in X\subopt$ over both sides for 
Eq~\eqref{eqn:quadratic-inaccurate}, we find that, 
\begin{equation*}
  \lambda \dist(0, X\subopt)
  \dist(x_{k+1}, X\subopt)
  \le L \dist(x_k, X\subopt)^2 + 2 \epsilon_k.
\end{equation*}
Dividing each side by $\stabconst  \dist(0, X\subopt)^2$ yields
\begin{equation}
  \frac{\dist(x_{k+1}, X\subopt)}{\dist(0, X\subopt)}
  \le
  \frac{\Lipconst}{\stabconst} \frac{\dist(x_k, X\subopt)^2}{\dist(0, X\subopt)^2}
  + \frac{2 \epsilon_k}{\stabconst \dist(0, X\subopt)^2}.
  \label{eqn:error-descending-claim}
\end{equation}

\newcommand{\condnum}{\kappa}

Inductively applying inequality~\eqref{eqn:error-descending-claim} when
$\epsilon_k = 0$ yields the second statement of the theorem.  When
$\epsilon_k > 0$, a brief technical lemma shows the
convergence rate:
\begin{lemma}
  \label{lemma:recursive-error-sequence}
  Let the sequences $a_k \ge 0$, $\epsilon_k \ge 0$ satisfy 
  $a_k \le \condnum a_{k-1}^2 + \epsilon_{k-1}$.
  Then
  \begin{equation*}
    a_k
    \le \max\left\{
    (2 \condnum)^{2^k - 1} a_0^{2^k},
    \max_{0 \le j < k}
    (2 \condnum)^{2^{k - j - 1} - 1}
    (2 \epsilon_j)^{2^{k - j - 1}}
    \right\}.
  \end{equation*}
\end{lemma}
\begin{proof}
  The proof is by induction. For $a_1$, we certainly
  have $a_1 \le 2 \condnum \vee 2 \epsilon_0$ because both
  sequences are non-negative.
  For the general case, assume the result holds for $a_k$, where
  $k$ is arbitrary. Then
  \begin{align*}
    a_{k + 1}
    & \le \condnum a_k^2 + \epsilon_k
    \le \max\left\{2 \condnum a_k^2, 2 \epsilon_k\right\} \\
    & \le 2 \condnum (2 \condnum)^{2^{k + 1} - 2}
    a_0^{2^{k + 1}}
    \vee \max_{0 \le j < k}
    \left\{2 \condnum (2 \condnum)^{2^{k - j} - 2}
    (2 \epsilon_j)^{2^{k - j}}\right\}
    \vee 2 \epsilon_k \\
    & = \max \left\{
    (2 \condnum)^{2^{k + 1} - 1} a_0^{2^{k + 1}},
    \max_{0 \le j < k + 1}
    (2 \condnum)^{2^{k - j} - 1}
    (2 \epsilon_j)^{2^{k - j}}\right\}
  \end{align*}
  as desired.
\end{proof}

Applying Lemma~\ref{lemma:recursive-error-sequence}
in inequality~\eqref{eqn:error-descending-claim} with
$\condnum = \frac{\Lipconst}{\lambda}$ and
$a_k = \frac{\dist(x_k, X\subopt)}{\dist(0, X\subopt)}$ yields the theorem.

%% file: no-noise-retrieval-complex.tex

\section{Noiseless Phase Retrieval Problem}
\label{sec:no-noise}

We begin our discussion of the phase retrieval problem by considering the
noiseless case, that is, when the observations $b_i = |\<a_i,
x\subopt\>|^2$.  Throughout this section and the remainder of the paper, the
vectors $a_i$ are assumed to be independent and identically distributed
copies of a random vector $a \in \C^n$.  Based on our discussion after
Theorem~\ref{theorem:quadratic-convergence-errors}, we show that (i) the
objective~\eqref{eqn:optimization-problem} is
stable~\ref{condition:stability}, (ii) is quadratically
approximable~\ref{condition:approximation}, and (iii) that we have a good
initializer $x_0$.
%
%
In the coming three sections, we address each of these in turn, providing
progressively stronger assumptions that are sufficient for each condition
to hold with high probability as soon as the number of measurements
$m/n > c$, where $c$ is a numerical constant. In
Section~\ref{sec:no-noise-summary} we provide a summary theorem
that encapsulates our results.
For readability, we defer proofs to Sec.~\ref{sec:no-noise-proofs}.

\subsection{Stability}
\label{sec:no-noise-stability}

Our first step is to provide conditions under which stability holds.
We divide our discussion of stability conditions into
the real and complex cases, as the real case is considerably easier.

\subsubsection{Stability for real-valued vectors}

With the stability condition in mind, we make the following assumption
on the random measurement vectors $a\in \R^n$: 
\begin{assumption}
  \label{assumption:stability}
  There exist constants $\stabfunc > 0$ and
  $\probbig>0$ such that for all $u, v \in \R^n$,
  $\ltwo{u} = \ltwo{v} = 1$, we
  have
  \begin{equation*}
    \P\left(|\<a, u\>| \wedge |\<a,v\>| \geq \stabfunc \right)
    \geq \probbig.
  \end{equation*}
\end{assumption}
\noindent
Intuitively, Assumption~\ref{assumption:stability} says that the measurement
vectors $a \in \R^n$ have sufficient support in all
directions $u, v \in \R^n$.
One simple sufficient condition for this is a type of small-ball
condition~\cite{Mendelson14}, as follows,
which makes it clear that
Assumption~\ref{assumption:stability} requires no light
tails: just that the probability of one of $|\<a, u\>|$ and $|\<a, v\>|$
being small is small.

\begin{example}[Stability by the small-ball method]
  \label{example:small-ball}
  Assume that we may choose positive but small enough
  $\varepsilon > 0$ that
  \begin{equation*}
    \sup_{\ltwo{u} = 1} \P\left(|\<a, u\>| < \varepsilon\right) < \half.
  \end{equation*}
  Then by the union bound, the choice
  $p_0 = 1 - 2 \sup_{\ltwo{u}} \P(|\<a, u\>| < \varepsilon) > 0$ and
  $\stabfunc = \varepsilon$ immediately yields
  \begin{align*}
    \P\left(|\<a, u\>| \wedge |\<a,v\>| \ge \stabfunc\right)
    & = 1 - \P\left(|\<a, u\>| < \varepsilon
    ~ \mbox{or} ~
    \wedge |\<a,v\>| < \varepsilon \right) \\
    & \ge 1 - \P\left(|\<a, u\>| < \varepsilon\right)
    - \P\left(|\<a, u\>| < \varepsilon\right)
    \ge \probbig.
  \end{align*}
  As a further specialization, if
  $a \sim \normal(0, I_n)$ is an isotropic Gaussian, then
  the choice $\stabfunc = .31$ yields
  $\P(|\<a, u\>| \le \stabfunc) \le \frac{1}{4}$, so
  we may take $\stabfunc = .31$ and $\probbig = \half$.
\end{example}

As we note in the discussion preceding Condition~\ref{condition:stability},
for the objective~\eqref{eqn:optimization-problem} it is immediate that in
the noiseless case that,
\begin{equation*}
  f(x) - f(\pm x\subopt) = f(x) = 
  \frac{1}{m} \lone{(Ax)^2 - (A x\subopt)^2}
  = \frac{1}{m} \sum_{i=1}^m |\<a_i, x-x\subopt\>\<a_i, x+x\subopt\>|.
\end{equation*}
Thus, if we can show the stability
condition~\eqref{eqn:stability-real}---equivalently, that $\sum_{i = 1}^m
|\< a_i, u \> \< a_i, v \>| \ge \stabconst m$ for $u, v
\in \sphere^{n-1}$---then Condition~\ref{condition:stability} holds for 
$X\subopt = \{\pm x\subopt\}$.
To that end, we provide the following guarantee, which we prove in
Sec.~\ref{sec:proof-small-ball}.
\begin{proposition}
  \label{proposition:small-ball}
  Let Assumption~\ref{assumption:stability} hold. There exists a
  numerical constant $c < \infty$ such that for any $t \ge 0$,
  \begin{equation*}
    \P\left(\frac{1}{m} \sum_{i=1}^m
    |\<a_i, u\>\<a_i, v\>| \ge \stabfunc^2 \cdot 
    	\left(\probbig - c \sqrt{\frac{n}{m}}-\sqrt{2t}\right)
     ~\mbox{for all}~ u, v \in \sphere^{n-1}\right)
    \ge 1 - 2 e^{-mt}.
  \end{equation*}
\end{proposition}

Proposition~\ref{proposition:small-ball}
immediately yields the following corollary, which shows that
the stability condition~\ref{condition:stability} holds with 
high probability for $m/n \gtrsim 1$.
\begin{corollary}
  \label{corollary:no-noise-stability-of-objective}
  Let Assumption~\ref{assumption:stability} hold. Then there exists a
  numerical constant $c < \infty$ such that if $m \probbig^2
  \ge c n$, then
  \begin{equation*}
    \P\left(f(x) - f(x\subopt) \geq \half \stabfunc^2 
    \probbig \ltwo{x-x\subopt}\ltwo{x+x\subopt}~\mbox{for all $x \in \R^n$} \right)
    \ge
    1 - 2 \exp\left(-\frac{m\probbig^2}{32}\right).
  \end{equation*}
\end{corollary}

\citet{EldarMe14} establish the stability
Condition~\ref{condition:stability} under more restrictive assumptions on
the distribution of the measurement vectors $\{a_i\}_{i=1}^m$. Concretely,
they require that the distribution of $a$ is subgaussian (see
Definition~\ref{def:sub-gaussian-vector} to come) and isotropic (meaning
that $\E [aa^T] = I_n$). As Example~\ref{example:small-ball} makes
clear, our result only
requires weaker small ball assumptions without any restrictions on
tails or the covariance structure of the random vector $a$.

\subsubsection{Complex Case}

\newcommand{\smallfunc}{h}  
\newcommand{\subg}{\sigma}  
\newcommand{\absWconst}{\tau}  

We now investigate conditions sufficient for stability of $A$ in the
measurement vectors $a_i$ are complex-valued. In this case, the argument is
not quite so simple, as stability for complex vectors requires more uniform
notions of function growth.  Accordingly, we make the following two
assumptions on the random vectors $a_i \in \C^n$. The first is a small-ball
type assumption, while the second requires that the vectors $a_i$ are
appropriately uniform in direction.  To fully state our assumptions, we
require an additional definition on the sub-Gaussianity of random
vectors. We define this in terms of Orlicz norms (following
\cite[Ch.~2.2]{Vershynin12,VanDerVaartWe96}).
\begin{definition}
  \label{def:sub-gaussian-vector}
  The random vector $a \in \C^n$ is \emph{$\subgauss^2$-sub-Gaussian} 
  if for all $v \in \C^n$, $\ltwo{v} = 1$,
  \begin{equation*}
    \E \left[\exp\bigg(\frac{|\<a, v\>|^2}{\subgauss^2}
      \bigg)\right] \leq e. 
  \end{equation*}
\end{definition}

With this definition, we now provide assumptions on the random measurement 
vector $a \in \C^n$.
\begin{assumption}
  \label{assumption:complex-small-ball}
  There exists a non-increasing function $\smallfunc :
  \R_+ \to \R_+$ with $\smallfunc(0) = 0$ such that for $0 \le \epsilon$,
  \begin{equation*}
    \P(\ltwo{a} \le \epsilon \sqrt{n})
    \le \smallfunc(\epsilon).
  \end{equation*}
\end{assumption}
\noindent
We also require that the $a$, when normalized, are sufficiently
uniform.
\begin{assumption}
  \label{assumption:rank-2-growth}
  Let $w = \sqrt{n} a / \ltwo{a}$. The random vector $w$ is
  $\subg^2$-sub-Gaussian (Definition~\ref{def:sub-gaussian-vector}).
  In addition,
  for any matrix $X \in \C^{n \times n}$ with rank at most $2$,
  \begin{equation*}
    \E\left[\left|w\cg X w\right|\right]
    \ge \absWconst^2 \lfro{X}.
  \end{equation*}
\end{assumption}
\noindent
Assumption~\ref{assumption:rank-2-growth} may seem somewhat challenging
to verify, but it holds for any rotationally symmetric distribution, and
moreover, in this case we have that
$\absWconst \ge c\subg$ for a numerical constant $c > 0$.

\begin{example}[Rotationally symmetric measurements]
  \label{example:rotation-invariance}
  Let the measurement vectors $a_i$ be rotationally symmetric, so that for
  unitary $U \in \C^{n \times n}$, the distribution of $U a$ is identical to
  $a$. We show that Assumption~\ref{assumption:rank-2-growth} holds.  In
  this case, $w = \sqrt{n} a / \ltwo{a}$ is uniform on the radius-$\sqrt{n}$
  sphere in $\C^n$, and standard results in convex geometry~\cite{Ball97}
  show $w$ is $O(1)$-sub-Gaussian
  (Definition~\ref{def:sub-gaussian-vector}). As $a$ is rotationally
  symmetric, $w$ is also equal in distribution to $\sqrt{n} z / \ltwo{z}$
  where $z$ is standard complex normal; thus, we have for any rank $2$ or
  less Hermitian $X$ that
  \begin{equation*}
    \E[|w\cg X w|]
    = n \E[|z\cg X z|] \E[1 / \ltwo{z}^2]
    \stackrel{(i)}{\ge} \E[|z \cg X z|]
    \stackrel{(ii)}{\ge} \frac{2 \sqrt{2}}{\pi} \lfro{X}
  \end{equation*}
  where inequality~$(i)$ is a consequence of
  $\E[1 / \ltwo{z}^2] \ge \frac{1}{n}$ and
  inequality~$(ii)$ is a calculation
  (see Lemma~\ref{lemma:funny-rank-2-gaussian}
  in Appendix~\ref{sec:properties-gaussians})
  as $X$ is rank $2$.
\end{example}

With these assumptions in place, we have the following stability guarantee
for the random matrix $A$. We defer the proof to
Appendix~\ref{sec:proof-complex-growth}.
\begin{proposition}
  \label{proposition:complex-growth}
  Let Assumptions~\ref{assumption:complex-small-ball}
  and~\ref{assumption:rank-2-growth} hold.
  Let $c > 0$ be chosen such that
  $\smallfunc(c) \le \frac{1}{2(1 + e)} \frac{\absWconst^4}{\subg^4}$.
  There
  exist numerical constants $c_0 > 0$ and $c_1 < \infty$
  such that
  with probability at least
  \begin{equation*}
    1 - \exp\left(-c_0 m \frac{\absWconst^4}{\subg^4}
      + c_1 n \log \frac{\subg^2}{\absWconst^2}\right),
  \end{equation*}
  we have
  \begin{equation*}
    \frac{1}{m} \lone{|Ax|^2 - |Ay|^2}
    \ge \frac{c^2 \absWconst^2}{4}
    \cdot \inf_\theta \ltwos{x - e^{\imagunit\theta} y}
    \cdot \sup_\theta \ltwos{x - e^{\imagunit\theta} y}
  \end{equation*}
  simultaneously for all $x, y \in \C^n$.
\end{proposition}
\noindent
The proposition shows that if the ratio between the growth constant
$\absWconst$ and the sub-Gaussian constant $\subg$ in
Assumption~\ref{assumption:rank-2-growth} is bounded, as it is in the case
of rotationally invariant vectors $a$
(Example~\ref{example:rotation-invariance}), then we have the
complex stability guarantee~\eqref{eqn:stability-complex} as soon as
$m \gtrsim n$ with extremely high probability.

\subsection{Quadratic Approximation}
\label{sec:no-noise-qa}

With the stability condition~\ref{condition:stability} in place, we turn to
a discussion of the approximation condition~\ref{condition:approximation}.
As implied by the estimate in inequality~\eqref{eqn:approximation-opnorm},
the quadratic approximatio ncondition is satisfied with parameter $\lipconst
= 2\opnorms{\frac{1}{m} A\cg A}$.  To control this quantity, we require that
the rows of the matrix $A \in \C^{m \times n}$ be sufficiently light-tailed.
\begin{assumption}
  \label{assumption:sub-gaussian-vector}
 The random vector $a \in \C^n$ is $\subgauss^2$-sub-Gaussian.
\end{assumption}
\noindent
It is of course possible to bound $\opnorm{A}$ when the rows 
have heavier tails using appropriate symmetrization
techniques and matrix Khintchine inequalities
(cf.~\cite[Sec.~2.6]{Vershynin12});
the extension is clear, so we do not address such issues.

Certainly, not all measurement vectors $a_i$ satisfy
Assumption~\ref{assumption:sub-gaussian-vector}. Using that $\E[e^{\lambda
    Z^2}] =\hinge{1 - 2 \lambda}^{-\half}$ for $Z \sim \normal(0, 1)$, it
holds for standard real normal vectors $a_i \simiid \normal(0, I_n)$ or
complex normal vectors $a_i \simiid \frac{1}{\sqrt{2}}\left(\normal(0,
I_n) + \imagunit \normal(0, I_n)\right)$
with $\subgauss^2 = \frac{2 e^2}{e^2 - 1}
\approx 2.313$. Similarly, it also holds for $a_i$ uniform on
$\sphere^{n-1}$ with $\subgauss^2 = O(1) \cdot \frac{1}{n}$.  In practice it
may be useful to apply our algorithm to the transformed data $\{a_i /
\ltwo{a_i}\}_{i = 1}^m$ and $\{b_i / \ltwo{a_i}^2\}_{i=1}^m$, which (in the
noiseless case or case with infrequent but arbitrary corruptions of $b_i$)
is likely to make the problem better conditioned.  There are two heuristic
motivations for this: first, those measurement vectors with larger
magnitudes $\norm{a_i}$ tend to place a higher weight in the optimization
problem~\eqref{eqn:optimization-problem}, and thus normalization can make
the observations ``comparable'' to each other; second, normalization
guarantees the measurement vectors $\{a_i\}_{i=1}^m$ satisfy
Assumption~\ref{assumption:sub-gaussian-vector}, yielding easier
verification of Condition~\ref{condition:approximation}.  (It may be more
challenging to verify Assumption~\ref{assumption:stability}, but if the
$a_i$ are sufficiently isotropic this presents no special difficulties.)

Standard results guarantee that the random matrices $A$ have well-behaved
singular vectors whenever Assumption~\ref{assumption:sub-gaussian-vector}
holds; we provide one  such result due to Vershynin~\cite[Thm.~39,
  Eq.~(25)]{Vershynin12} with constants that are achievable by
tracing his proof.
\begin{lemma}
  \label{lemma:matrix-concentration}
  Let Assumption~\ref{assumption:sub-gaussian-vector} hold
  and $\Sigma = \E[aa\cg]$. Then
  for all $t \ge 0$,
  \begin{equation*}
    \P\left(\opnorm{\frac{1}{m} A\cg A - \Sigma}
    \ge 11 \subgauss^2 \max\left\{
    \sqrt{\frac{4n}{m} + t},
    \frac{4n}{m} + t \right\}
    \right) \le \exp(-m t).
  \end{equation*}
  Moreover, $\opnorm{\Sigma} \le \subgauss^2$ and
  $\E[|\langle a, v\rangle|^k] \le \Gamma(\frac{k}{2} + 1) e \subgauss^k$
  for all $k \ge 0$.
\end{lemma}
\noindent
Thus, we have the following 
corollary of Lemma~\ref{lemma:matrix-concentration}, which
guarantees that Condition~\ref{condition:approximation} holds with
high probability for $m/n \gtrsim 1$.
\begin{corollary}
  \label{corollary:no-noise-quadratic-approximation}
  Let Assumption~\ref{assumption:sub-gaussian-vector} hold.
  Then there exists a numerical constant $c < \infty$ such that whenever
  $m \ge c n$
  \begin{equation*}
    \P\left(\left|f_x(y) - f(y)\right| \leq 2 \subgauss^2 \ltwo{x-y}^2 
    ~\mbox{holds uniformly for}~ x, y\in \R^n\right)
    \ge 1- \exp\left(-\frac{m}{c}\right).
  \end{equation*}
\end{corollary}
\begin{proof}
  Assume $m \geq cn$ for $c$ large enough, and
  choose $t$ small enough in Lemma~\ref{lemma:matrix-concentration}
  that $11 \subgauss^2 \sqrt{4n / m + t} \le \subgauss^2$.
  Applying the triangle inequality to
  $\opnorms{\frac{1}{m} A\cg A - \Sigma} \le
  \opnorms{\frac{1}{m} A\cg A} + \subgauss^2$ yields the result
  by equation~\eqref{eqn:approximation-opnorm}.
\end{proof}

\newcommand{\smallprob}{\kappa}
\newcommand{\perploss}{\phi}
\newcommand{\residual}{\Delta}

\subsection{Initialization}
\label{sec:no-noise-init}

The last ingredient in achieving strong convergence guarantees for our
prox-linear procedure for phase retrieval is to provide a good
initialization. There are a number of initialization strategies in the
literature~\cite{CandesLiSo15, WangGiEl16, ZhangChLi16} based on spectral
techniques, which work as follows. We decompose the initialization into two
steps: we (i) find an estimate of the direction direction $d\subopt \defeq
x\subopt/\ltwo{x\subopt}$, and (ii) estimate the magnitude $r\subopt \defeq
\ltwo{x\subopt}$. The latter is easy: assuming that $\E[aa\cg] = I_n$, one
simply uses $\what{r}^2 = m^{-1}\sum_{i = 1}^m b_i^2$, which is unbiased and
tightly concentrated. The former, the direction estimate, is somewhat
trickier.

Wang, Giannakis, and Eldar~\cite{WangGiEl16} provide an empirically excellent initialization whose
heuristic justification is as follows. First, for random vectors $a_i$ in
high dimensions, we expect $a_i$ to usually be orthogonal to the direction
$d\subopt$. Thus, by extracting the smallest magnitude $b_i = |\<a_i,
x\subopt\>|^2$, we have the vectors $a_i$ that are ``most'' orthogonal to the
direction $d\subopt$; letting $\selected$ be these small indices, the
eigenvector corresponding to the smallest eigenvalue (for simplicity, we
simply call this the smallest eigenvector) of $\sum_{i \in \selected} a_i
a_i\cg$ should be close to the direction $d\subopt$. A variant of this procedure
is to note that $\frac{1}{m} \sum_{i = 1}^m a_i a_i\cg \approx I_n$ when the
$a_i$ are isotropic, so that the largest eigenvector of $\sum_{i \not \in
  \selected} a_i a_i\cg$ should also be close to $d\subopt$. This initalization
strategy has the added benefit that---unlike the original spectral
initialization schemes developed by Cand\`{e}s et al.~\cite{CandesLiSo15}, which rely on
eigenvectors of $\sum_{i = 1}^m b_i a_i a_i\cg$ that may not concentrate at
sub-Gaussian rates (as the sum involves fourth moments of random
vectors)---the sums $\sum_i a_i a_i\cg$ are tightly concentrated.

Unfortunately, we believe Wang et al.'s proof that this
initialization works contains a mistake (note that they consider only the 
case where $\{a_i\}_{i=1}^m$ are real):  letting $U \in \R^{n \times (n-1)}$
be an orthogonal matrix whose columns are all orthogonal to $d\subopt$, in the
proof of Lemma 2 (Eqs.~(68)--(70) in~\cite{WangGiEl16}) they assert that
$|\selected^c|^{-1} \sum_{i \in \selected^c} U\cg a_i a_i U \to I_{n-1}$ when
the $a_i$ are uniform on $\sqrt{n} \sphere^{n-1}$. This is not true
(nor does appropriate normalization by $n$ or $n-1$ make it true), as it
ignores the subtle effects of conditioning in the construction of
$\selected$. In spite of this issue, the initialization
they propose works remarkably well, and as we show
presently, it provably provides a good estimate $\what{d}$ of $d\subopt$.
We include the initialization procedure in
Algorithm~\ref{alg:non-noisy-init}.

\begin{algorithm}[t]
\label{algorithm:non-noisy-initialization}
  \caption{\label{alg:non-noisy-init} Initialization procedure for
    non-noisy data}
  \KwData{Measurement matrix $A \in \C^{m \times n}$ and signals
    $b = |Ax\subopt|^2$}
\Begin{
    Set $\what{r}^2 = \frac{1}{m} \sum_{i = 1}^m b_i$ \\
    Select indices
    $\selected \defeq \{i \in [m] : b_i \le \half \what{r}^2
    \}$ \\
    Construct directional estimates $\what{d}$ by
    \begin{equation*}
      X\init \defeq \sum_{i \in \selected} a_i a_i\cg,
      ~~~
      \what{d} = \argmin_{d \in \sphere^{n-1}} d\cg X\init d
    \end{equation*} \\
    \KwRet{$(\what{r}, \what{d})$}
  }
\end{algorithm}


\newcommand{\probdir}{\mathsf{p}_0(d\subopt)}

With the previous discussion in mind, we provide a general
assumption that is sufficient for
Alg.~\ref{alg:non-noisy-init} to return a direction and magnitude estimate
sufficiently accurate for phase retrieval.
\label{sec:no-noise-initialization}
\begin{assumption}
  \label{assumption:well-spread}
  For some
  $\epsilon_0 \in \openleft{0}{1}$ and
  $\probdir > 0$, the following hold.
  \begin{enumerate}[(i)]
  \item \label{item:continuity-direction}
    For all $\epsilon \in \openleft{0}{\epsilon_0}$,
    the following continuity and directional likelihood
    conditions hold:
    \begin{equation*}
      \P\left(|\<a, d\subopt\>|^2 \in \left[\frac{1-\epsilon}{2}, \frac{1+\epsilon}{2}\right]\right) 
      \leq \smallprob \epsilon	
      ~~\text{and}~~\P\left(|\<a, d\subopt\>|^2 \leq \frac{1-\epsilon}{2}\right) 
      \geq \probdir >0.
  \end{equation*}
  \item \label{item:conditional-id} There exist
    functions 
    $\perploss : [0, \epsilon_0] \to \R_+$ and
    $\residual : [0, \epsilon_0] \to \C^{n \times n}$
    such that
    \begin{equation*}
      \E \left[aa\cg \mid |\<a, d\subopt\>|^2 \leq \frac{1-\epsilon}{2} \right]
      = I_n - \perploss(\epsilon) d\subopt d\subopt\cg + \residual(\epsilon)
      ~~ \mbox{for} ~ \epsilon \in [0, \epsilon_0].
    \end{equation*}
  \item \label{item:isotropic} $\E[aa\cg] = I_n$.
  \end{enumerate}
\end{assumption}
\noindent
Assumption~\ref{assumption:well-spread} on its face seems complex, but each
of its components is not too
stringent. Part~\eqref{item:continuity-direction} says that $|\<a,
d\subopt\>|^2$ has no point mass at $|\<a, d\subopt\>|^2 = \half$ and that
$|\<a, d\subopt\>|^2$ has reasonable probability of being smaller than
$\half$. Part~\eqref{item:isotropic} simply states that in expectation, $a$
is isotropic (and a rescaling of $a$ can guarantee this).
Part~\eqref{item:conditional-id} is the most subtle and essential for our
derivation; it says that $a \in \R^n/\C^n$ is reasonably isotropic, even if
we condition on $|\<a, d\subopt\>|$ being near zero for some direction
$d\subopt$, so that most mass of $aa\cg$ is distributed uniformly in the
orthogonal directions $I_n - d\subopt d\subopt\cg$.  The error terms
$\perploss$ and $\residual$ allow non-trivial latitude in this condition, so
that Assumption~\ref{assumption:well-spread} holds for more than just
Gaussian vectors. That said, for concreteness we provide the following
example.

\begin{example}[Gaussian vectors and conditional directions]
  We consider the standard cases that $a_i \simiid \normal(0, I_n)$, or the
  complex counterpart $a_i \simiid \frac{1}{\sqrt{2}} \left(\normal(0, I_n)
  + \imagunit \normal(0, I_n)\right)$, showing that such $a_i$ satisfy
  Assumption~\ref{assumption:well-spread} for any $\epsilon_0 \in (0, 1)$
  with residual error $\residual \equiv 0$.  Clearly
  Part~\eqref{item:isotropic} holds.  For
  Part~\eqref{item:continuity-direction}, note that $|\<a, d\subopt\>|^2$ is
  $\chi_1^2$-distributed with density $f(t) = e^{-t/2} (2 \pi t)^{-\half}$.
  Integrating the density using its upper bound, we may set $\smallprob \le
  f(\frac{1 - \epsilon_0}{2}) = \exp(\frac{\epsilon_0 - 1}{4}) / \sqrt{\pi
    (1 - \epsilon_0)}$ and $\probdir = \P(\chi_1^2 \leq
  \frac{1-\epsilon_0}{2}) \ge \half \sqrt{1 - \epsilon_0}$.
  Part~\eqref{item:conditional-id} is all that remains. By the rotational
  invariance of $a \sim \normal(0, I_n)$, we see for any $t \in \R_+$ that
  \begin{equation*}
    \E[aa\cg \mid |\<a, d\subopt\>|^2 \leq t]
    = I_n - d\subopt d\subopt\cg
    + \E [|\<a, d\subopt\>|^2 \mid |\<a, d\subopt\>|^2 \leq t]
    d\subopt d\subopt\cg.
  \end{equation*}
  We claim the following lemma, whose proof we provide in 
  Appendix~\ref{sec:proof-ez-square}.
  \begin{lemma}
    \label{lemma:ez-square}
    Let $Z$ be a continuous random variable with a density symmetric about zero 
    and decreasing on $\R_+$. Then for any $c \in \R$
    we have
    $\E[Z^2 \mid Z^2 \le c^2] \le \frac{c^2}{3}$.
  \end{lemma}
  \noindent
  Thus, setting $\perploss(\epsilon) = 1 - \E[|\<a, d\subopt\>|^2 \mid 
  |\<a,  d\subopt\>|^2 \le \frac{1 - \epsilon}{2}] \ge 1 - \frac{(1 - \epsilon)^2}{6}
  \ge \frac{5}{6}$, we see that part~\eqref{item:conditional-id} of
  Assumption~\ref{assumption:well-spread}
  holds with $\perploss(\epsilon) \ge \frac{5}{6}$
  and $\residual(\epsilon) \equiv 0$.
\end{example}

We now state our main proposition of this section.
\begin{proposition}
  \label{proposition:non-noisy-initialization}
  Let Assumptions~\ref{assumption:sub-gaussian-vector}
  and~\ref{assumption:well-spread} hold and let $\epsilon_0$ and $\probdir$ be
  as in Assumption~\ref{assumption:well-spread}.
  Define the error measure
  \begin{equation*}
    \errorterm(\epsilon)
    \defeq \left[1
      + \frac{4 e (1 + \log \frac{1}{1 \wedge \smallprob \epsilon})
        \subgauss^2 \smallprob}{\probdir}
      + \frac{2 \smallprob (1+\epsilon)}{{\probdir}^2}
    \right] \epsilon.
  \end{equation*}
  There exists a numerical constant
  $c > 0$ such that the following
  holds. Let $(\what{r}, \what{d})$ be the estimated magnitude
  and direction of Alg.~\ref{alg:non-noisy-init} and
  define $x_0 = \what{r} \what{d}$. For any $\epsilon \in [0, \epsilon_0]$,
  if
  \begin{equation*}
    c \cdot \frac{m}{n}
    \ge \frac{\subgauss^4 \log^2 \probdir}{
      \probdir \epsilon^2}
    \vee
    \frac{1}{(\smallprob \epsilon)^2}
  \end{equation*}
  then
  \begin{equation}
    \label{eqn:no-noise-dist}
    \frac{\dist(x_0, X\subopt)}{\ltwo{x\subopt}}
    \le \sqrt{2(1 + \epsilon)}\frac{
      \opnorm{\residual(\epsilon)} + \errorterm(\epsilon)}{
      \hinge{\perploss(\epsilon)
        - \opnorm{\residual(\epsilon)}
        - \errorterm(\epsilon)}}
    + \epsilon
  \end{equation}
  with probability at least
  \begin{equation*}
    1
    - \exp\left(-\frac{c m \epsilon^2}{\subgauss^4}\right)
    - 2\exp\left(-c m \epsilon^2 \smallprob^2\right)
    - \exp\left(-\frac{m \probdir^2}{2}\right)
    - \exp\left(-\frac{cm \smallprob \epsilon^2}{
      \subgauss^4 \log^2 \probdir}\right).
  \end{equation*}
\end{proposition}
\noindent
We prove Proposition~\ref{proposition:non-noisy-initialization} in two
parts. In the first part
(Sec.~\ref{sec:proof-non-noisy-initialization-cond}), we define a number of
events and proceed conditionally, showing that if each of the events occurs
then the conclusion~\eqref{eqn:no-noise-dist} holds.  In the second part
(Sec.~\ref{sec:proof-non-noisy-initialization-prob}) we show that the events
occur with high probability.

We provide a few remarks to make the result clearer. Let us make the
simplifying assumptions that the constants in
Assumption~\ref{assumption:well-spread} are absolute (which
is true for Gaussian measurements), that is, that that
$\subgauss^2 = O(1)$ and $\probdir = \Omega(1)$
(it is no loss of generality to assume $\smallprob \ge 1$).
Then for numerical constants $c > 0, C < \infty$ we have for any
$\epsilon \in [0, \epsilon_0]$ that
the error measure $\errorterm(\epsilon)
\le C \epsilon \log \frac{1}{\epsilon}$, and
with probability at least $1 - 5 \exp(-c m \epsilon^2)$ that
\begin{equation*}
  \frac{m}{n} \ge \frac{C}{\epsilon^2}
  ~~ \mbox{implies} ~~
  \frac{\dist(x_0, X\subopt)}{\ltwo{x\subopt}}
  \le C\frac{\opnorm{\residual(\epsilon)} + \errorterm(\epsilon)}{
      \hinge{\perploss(\epsilon) - \opnorm{\residual(\epsilon)}
        - \errorterm(\epsilon)}}
  + \epsilon.
\end{equation*}
Here, we see three competing terms. The first two, the separation
$\perploss(\epsilon)$ and error $\residual(\epsilon)$, arise from the
conditional expectation of
Assumption~\ref{assumption:well-spread}\eqref{item:conditional-id}, where
$\E[aa\cg \mid \langle a, d\subopt\rangle^2 \le \frac{1 - \epsilon}{2}] =
I_n - \perploss(\epsilon) d\subopt d\subopt\cg + \residual(\epsilon)$.
This is intuitive: the larger the separation $\perploss(\epsilon)$ from
uniformity in the conditional expectation of $aa\cg$, the easier it is for
spectral initialization to succeed; larger error $\residual(\epsilon)$ will
hide the directional signal $d\subopt$.  The last term is the error
$\errorterm(\epsilon) \lesssim \epsilon \log \frac{1}{\epsilon}$, which
approaches 0 nearly as quickly as $\epsilon \to 0$.

In the case that the error $\residual(\epsilon) = 0$ and gap
$\perploss(\epsilon)$ is bounded away from zero, which holds for elliptical
distributions with identity covariance---the Gaussian distribution and
uniform distribution on the sphere being the primary
examples---we thus see that as soon as $\frac{m}{n} \gtrsim \epsilon^{-2}$
we have relative error
\begin{equation}
  \frac{\dist(x_0, X\subopt)}{\ltwo{x\subopt}}
  \lesssim \epsilon \log \frac{1}{\epsilon}
  ~~ \mbox{with~probability~}
  \ge 1 - 5 \exp(-c m \epsilon^2).
  \label{eqn:good-close-initialization}
\end{equation}
That is, we can construct an arbitrarily good initialization with
large enough sample size. (This proves that the initialization scheme
of Wang et al.~\cite{WangGiEl16} also succeeds with high probability.)
On the other hand, when $\residual(\epsilon) \neq 0$ for all $\epsilon \in
[0, \epsilon_0]$, then
Proposition~\ref{proposition:non-noisy-initialization} cannot guarantee
arbitrarily good initialization: the error
term $\opnorms{\residual(\epsilon)}$ is never zero.
However, if it is small enough, we still achieve initializers that
are within \emph{constant} relative distance of $x\subopt$, which
is good enough for Theorem~\ref{theorem:quadratic-convergence-errors}.


\subsection{Summary and success guarantees}
\label{sec:no-noise-summary}
\newcommand{\finalstabconst}{\nu}

We now have guarantees of stability, quadratic approximation, and
good initialization for appropriate measurement matrices $A \in \R^{m \times
  n}$ when the observations $b = |Ax\subopt|^2$.  We
provide a summary theorem showing that our composite optimization procedure
works as soon as the sample size is large enough.  In stating the
theorem, we assume that each of Assumptions~\ref{assumption:stability},
\ref{assumption:sub-gaussian-vector}, and~\ref{assumption:well-spread}
holds with all of their constants actually numerical constants.
The one somewhat technical assumption we
require is that relating the sub-Gaussian parameter
$\subgauss^2$, the stability parameters
$\stabfunc$ and $\probbig$ (and in the complex case 
$\absWconst$ and the function $\smallfunc(\cdot)$), and
the error $\residual(\epsilon)$ and directional separation
constants $\perploss(\epsilon)$. Denote the stability constant 
$\finalstabconst = \stabfunc^2 \probbig$ in the real case and 
$\finalstabconst = \left(\smallfunc^{-1}\left(\frac{\absWconst^4}{2(1+e)\subgauss^4}\right)\absWconst\right)^2$. 
Then, in particular, if we assume that for a suitably small numerical constant 
$c > 0$, we have
\begin{equation*}
  \opnorm{\residual(\epsilon)}
  \le c \frac{\nu}{\subgauss^2},
  ~~~
  \perploss(\epsilon) \ge \half \opnorm{\residual(\epsilon)},
  ~~ \mbox{and} ~~
  \perploss(\epsilon) \ge \perploss\subopt > 0
\end{equation*}
for all $\epsilon \in [0, \epsilon_0]$.  We then have the following theorem,
which follows by combining our convergence
Theorem~\ref{theorem:quadratic-convergence-errors}
with Proposition~\ref{proposition:small-ball},
Corollary~\ref{corollary:no-noise-quadratic-approximation},
and Proposition~\ref{proposition:non-noisy-initialization}.
\begin{theorem}
  \label{thm:no-noise-retrieval}
  Let the conditions of the preceding paragraph hold.  There exists
  a numerical
  constant $c > 0$ such that if $\frac{n}{m} < c$, then (i) the
  initializer $x_0$ returned by Alg.~\ref{alg:non-noisy-init} satisfies
  $\dist(x_0, X\subopt) \le \frac{\nu }{8 \subgauss^2} \ltwo{x\subopt}$ and
  (ii) assuming no error in the minimization steps of the prox-linear
  method (Alg.~\ref{alg:prox-linear}),
  \begin{equation*}
    \dist(x_k, X\subopt) \le \ltwo{x\subopt} 2^{-2^k}
  \end{equation*}
  for all $k \in \N$ with probability at least
  $1 - e^{-c m}$.
\end{theorem}

We make two brief summarizing remarks. First, it is necessary to have $m
\gtrsim n$ to achieve exact recovery of the signal, as the parameter
$x\subopt \in \C^n$ has $2n$ unknowns and we have $m$ equations (indeed, $m
\ge 4n - 2$ is necessary for inectivity of the measurements in the complex
case~\cite{BalanCaEd06}). Thus, the sample complexity
Theorem~\ref{thm:no-noise-retrieval} specifies is optimal to within
numerical constants. Second, Theorem~\ref{thm:no-noise-retrieval} shows that
the prox-linear algorithm exhibits local quadratic convergence to the signal
$x\subopt$, which is in contrast to the local linear convergence of other
non-convex methods based on gradients and generalized
gradients~\cite{CandesLiSo15,ChenCa15,WangGiEl16}.  In contrast, however,
each iteration of our algorithm requires solving a structured convex
quadratic program, which is somewhat more expensive than the typical
gradient iterations; as we demonstrate in our experiments
(Section~\ref{sec:experiments}), this means our methods are about four
times
slower in overall run time than the best gradient-based methods, though
their recovery properties are better.

%% file: noise-retrieval-short.tex

\newcommand{\projperp}{P_\perp}
\providecommand{\sgn}{\mathop{\rm sgn}}

\section{Phase retrieval with outliers}
\label{sec:noise}

The objective~\eqref{eqn:optimization-problem} is the analogue of the
least-absolute deviation estimator---the median in $\R$---so in analogy with
our understanding of robustness~\cite{HuberRo09}, it is natural to expect it
should be robust to outliers.  We show that this is indeed the case, and the
prox-linear method we develop is effective. For simplicity in our
development, we assume for this section that all measurements and signals
are real-valued, and we consider the following corruption model: let
$\{\noise_i\} \subset \R$ be an arbitrary sequence, and given the $m$
measurement vectors $a_i$, we observe
\begin{equation*}
  b_i = \begin{cases}
    \<a_i, x\subopt\>^2 & \mbox{if $i\in \inliers$}\\
    \noise_i & \mbox{if $i\in \outliers$},
  \end{cases}
\end{equation*}
where $\outliers \subset [m]$ and $\inliers \subset [m]$ denote the outliers
and inliers, respectively. We assume there is a pre-specified measurement
failure probability $\pfail \in \openright{0}{\half}$, and $|\outliers| =
\pfail m$, and the indices $i \in \outliers$ are chosen
randomly. That is, measurement failures are random, though the noise
sequence $\noise_i$ may depend on $a_i$ (even adversarially), as we specify
presently. We assume no prior knowledge of which indices $i \in [m]$
actually satisfy $i \in \outliers$, or even of $\pfail$.



We consider the following two models for errors:
\begin{model}
  \label{model:full-independence}
  The measurement vectors $\{a_i\}_{i=1}^m$ are independent of the
  all the values $\{\noise_i\}_{i=1}^m$.
\end{model}
\begin{model}
  \label{model:partial-independence}
  The inlying measurement vectors $\{a_i\}_{i\in \inliers}$ are independent
  of the values $\{\noise_i\}_{i \in \outliers}$ of the corrupted
  observations.
\end{model}
\noindent
Model~\ref{model:full-independence} requires independence between the
noise and measurements: the adversary may only corrupt $\noise_i$ without
observing $a_i$.  Model~\ref{model:partial-independence} relaxes this,
allowing arbitrary dependence between the corrupted
data and the measurement vectors $a_i$ for $i \in \outliers$.
This is natural in scenarios in which the corruption may depend on
the measurement $a_i$.

The arbitrary corruption causes some technical challenges, but we may still
follow the outline in our analysis of phase retrieval without
noise in Sec.~\ref{sec:no-noise}. As we show in
Section~\ref{sec:noise-stability}, the objective $f(x)$ is still stable
(Condition~\ref{condition:stability}) as long as the measurement vectors are
light-tailed, though Gaussianity is unnecessary.  The quadratic
approximation conditions (Condition~\ref{condition:approximation}) are
completely identical to those in Sec.~\ref{sec:no-noise-qa}, so we ignore
them. Thus, as long as $\opnorm{A}$ is not too large (meaning $f_x(y)
\approx f(y)$) and we can find a good initializer, the prox-linear
iterations~\eqref{eqn:prox-iteration} will converge quadratically to
$x\subopt$. Finding a good initializer $x_0$ is somewhat trickier, but in
Section~\ref{sec:noise-initialization} we provide a spectral method,
inspired by Wang et al.~\cite{WangGiEl16}, that works with high probability as soon as
$m/n \ge C$ for some numerical constant $C$. We defer our arguments
to Appendix~\ref{sec:proofs-noisy}.

\subsection{Stability}
\label{sec:noise-stability}

The outlying indices, even when corruptions are chosen adversarially, have
limited effect on the growth and identification behavior of $f(x) =
\frac{1}{m} \lone{(Ax)^2 - b}$.  In particular, for $\pfail$ smaller than a
numerical constant, which we can often specify, the stability
condition~\ref{condition:stability} holds with high probability whenever
$m/n$ is large. More precisely, we have the following proposition, which
applies to independent $\subgauss^2$-sub-Gaussian measurements.
(See Sec.~\ref{sec:proof-noise-stability} for a proof.)

\begin{proposition}
  \label{proposition:noise-stability}
  Let Assumption~\ref{assumption:sub-gaussian-vector} hold
  and $\stabfunc = \inf_{u, v \in \sphere^{n-1}}
  \E[|\<a, u\> \<a, v\>|]$.
  Then under either of the models~\ref{model:full-independence}
  or~\ref{model:partial-independence}, there are
  numerical constants $c > 0$ and $C < \infty$ such that
  \begin{equation*}
    f(x) - f(x\subopt)
    \ge \left(\stabfunc -
    2 \pfail -
    C \subgauss^2 \sqrt[3]{\frac{n}{m}}
    - C \subgauss^2 t
    \right) \ltwo{x - x\subopt}
    \ltwo{x + x\subopt} ~~ \mbox{for~all~} x \in \R^n
  \end{equation*}
  with probability at least $1 - 2 e^{-c m} - 2 e^{-m t^2}$.
\end{proposition}

We continue our running example of Gaussian random
variables to motivate the proposition.
\begin{example}[Gaussian vectors]
  We claim that for $a \sim \normal(0, I_n)$ we have
  \begin{equation*}
    \stabfunc \defeq \inf_{u, v \in \sphere^{n-1}}
    \E[|\langle a, u \rangle \langle a, v \rangle |] =
    \frac{2}{\pi}.
  \end{equation*}
  Let $Z_u = \<a, u\>$, $Z_v = \<a, v\>$,
  and let $X, Y$ be independent $\normal(0, 1)$. Then
  \begin{align*}
    \inf_{u, v \in \sphere^{n-1}}
    \E[|Z_u Z_v|]
    & = \inf_{\rho \in [0, 1]}
    \left\{f(\rho) \defeq
    \E\left[\left|\rho X^2 - (1- \rho) Y^2\right|\right]\right\}.
  \end{align*}
  The function $f(\cdot)$ is convex and symmetric around $\half$.
  Thus $f(\rho) \geq f(1/2) = 2/\pi$.
\end{example}

\noindent
In the Gaussian measurement case, whenever
$\pfail < \frac{1}{\pi} \approx .318$, there
is a numerical constant $\lambda > 0$ such that
we have the stability
$f(x) - f(x\subopt) \ge \lambda \norm{x - x\subopt} \norm{x + x\subopt}$ for
as long
as $m/n$ is larger than some numerical constant.


\subsection{Initialization}
\label{sec:noise-initialization}

\begin{algorithm}[t]
  \caption{\label{alg:noisy-initialization}
    Initialization procedure with outliers}
  \KwData{Measurement matrix $A \in \R^{m \times n}$ and corrupted
    signals $b$}
\Begin{ 
    Define indices $\selected \defeq \left\{i \in [m] : b_i \leq
    \median\left(\{b_i\}_{i=1}^m\right)\right\}$.\\

    Construct directional and norm estimates $\what{d}$ and
    $\what{r}$ by
    \begin{align*}
      X\init & \defeq \frac{1}{m} \sum_{i=1}^m a_i a_i^T
      \indic{i \in \selected},
      ~~~
      \what{d} = \argmin_{d \in \sphere^{n-1}} d^T X\init d \\
      \what{r}^2 & \defeq \argmin_r
      G(r) \defeq
      \frac{1}{m} \sum_{i=1}^m \left|b_i -
      r \big\langle a_i, \what{d}\big\rangle^2
      \right|.
    \end{align*}
    \KwRet{$(\what{r}, \what{d})$}
  }
\end{algorithm}

The last ingredient for achieving strong convergence guarantees for the
prox-linear algorithm for phase retrieval is to provide a good
initialization $x_0 \approx x\subopt$. The strategies in the noiseless setting
in Section~\ref{sec:no-noise} will fail because of corruptions.  With this
in mind, we present Algorithm~\ref{alg:noisy-initialization}, which
provides an initializer in corrupted problems.

Before turning to the analysis, we provide some intuition for the
algorithm. We must construct two estimates: an estimate $\what{d}$ of the
direction $d\subopt = x\subopt / \ltwo{x\subopt}$ and an estimate $\what{r}$ of the
radius, or magnitude, of the signal $r\subopt = \ltwo{x\subopt}$. For the former,
a variant of the spectral initialization
(Alg.~\ref{alg:non-noisy-init})
suffices. If we take the $\selected \subset [m]$ to be the set of indices
$\selected$ corresponding to the smallest (say) $b_i$, in either
model~\ref{model:full-independence} or~\ref{model:partial-independence} the
indices $i \in \outliers$ are independent of the measurements $a_i$, so we
expect as in Section~\ref{sec:no-noise-init}
that $X\init = |\selected|^{-1} \sum_{i \in \selected} a_i a_i^T = z
I_n - z' d\subopt {d\subopt}^T + \Delta$, where $z, z'$ are random positive
constants and $\Delta$ is an error
matrix coming from both randomness in the $a_i$ and the corruptions. As long
as the error $\Delta$ is small, the minimum eigenvector of $X\init$ should
be approximately $d\subopt$. Once we have a good initializer $\what{d} \approx
d\subopt$, a natural idea to estimate $r\subopt$ is to pretend that $\what{d}$
\emph{is} the direction of the signal, substitute the variable $x = \sqrt{r}
\what{d}$ into the objective~\eqref{eqn:optimization-problem}, and solve for
$r$ to get a robust estimate of the signal strength $\norm{x\subopt}$. As
we show presently, this procedure succeeds with high probability
(and the estimate $\what{r}$ is good even when the data are non-Gaussian).



Let us make these ideas precise. First, we show that our estimate $\what{r}$ of
$\ltwo{x\subopt}$ is accurate
(See Sec.~\ref{sec:proof-good-direction-to-good-radius} for a proof).
\begin{proposition}
  \label{proposition:good-direction-to-good-radius}
  Let Assumption~\ref{assumption:sub-gaussian-vector} hold
  and $\E[aa^T] = I_n$.
  Let $\delta \in [0, 1]$ and $\pfail \in [0, \half]$. There exist
  numerical constants $0 < c$, $C < \infty$ such that if
  $\what{d}$ is an estimate of $d\subopt$ for which
  \begin{equation}
    \label{eqn:technical-condition-numerical}
    \delta \defeq
    \frac{C \subgauss^2}{1 - 2 \pfail} \dist(\what{d}, \{\pm d\subopt\})
    \le 1,
  \end{equation}
  then with probability at least $1 - 2 e^{-c m (1 - 2 \pfail)^2 /
    \subgauss^4}$
  all minimizers $\what{r}^2$ of
  $G(r) = \frac{1}{m} \sum_{i = 1}^m|b_i - r \langle a_i, \what{d}
  \rangle^2|$, defined in Alg.~\ref{alg:noisy-initialization},
  satisfy $\what{r}^2 \in [1 \pm \delta] \ltwo{x\subopt}^2$.
\end{proposition}

Given Proposition~\ref{proposition:good-direction-to-good-radius}, finding a
good initialization of $x\subopt$ reduces to finding a good estimate $\what{d}$
of the direction $d\subopt = x\subopt / \ltwo{x\subopt}$.  To make this precise, let
$r\subopt \defeq \norm{x\subopt}$ and assume that $\delta = \frac{C \subgauss^2}{1
  - 2\pfail} \dist(\what{d}, \{\pm d\subopt\}) \le 1$ as in
Proposition~\ref{proposition:good-direction-to-good-radius}; assume also the
relative error guarantee $|\what{r} - r\subopt| \le \delta r\subopt$.  Using the
triangle inequality and
Proposition~\ref{proposition:good-direction-to-good-radius}, for $x_0 =
\what{r} \what{d}$ we have
\begin{equation*}
  \dist(x_0, X\subopt) \leq \what{r} \dist(\what{d}, \{\pm d\subopt\}) + 
  |r\subopt - \what{r}|
  \leq r\subopt\left[(1+\delta)\dist(\what{d}, \{\pm d\subopt\}) + \delta\right] 
  \leq 2\delta r\subopt
  = \frac{C \subgauss^2}{1 - 2\pfail} \dist(\what{d}, \{\pm d\subopt\}) r\subopt,
\end{equation*}
as claimed.
We turn to the directional estimate; to make the analysis
cleaner we make the normality
\begin{assumption}
  \label{assumption:normal}
  The measurement vectors $a_i \simiid \normal(0, I_n)$.
\end{assumption}
\noindent
To state our guarantee on $\what{d}$, we require additional
notation for quantiles of Gaussian and $\chi^2$-random variables.  Let
$W \sim \normal(0, 1)$ and define the constant $\qfail$ and its associated
$\chi^2$-quantile $w_q^2$ by
\begin{equation}
  \label{eqn:q-fail}
  \P(W^2 \le w_q^2) =
  \qfail \defeq \frac{1}{2(1-\pfail )} + \frac{1-2\pfail }{4(1-\pfail)}
  = \frac{3 - 2 \pfail}{4(1 - \pfail)} < 1.
\end{equation}
Second, define the constant $\delta_q = 1 - \E[W^2 \mid W^2 \le w_q^2]$.
We have the following guarantee.
\begin{proposition}
  \label{proposition:noisy-good-initial-direction}
  Let $\what{d}$ be the smallest eigenvector of $X\init \defeq \frac{1}{m}
  \sum_{i=1}^m a_i a_i^T \indic{i\in \selected}$,
  as in Alg.~\ref{alg:noisy-initialization}. Let the constant $M = 0$ if
  Model~\ref{model:full-independence} holds and $M = 1$ if
  Model~\ref{model:partial-independence} holds.  There are numerical
  constants $0 < c, C < \infty$ such that for $t \ge 0$, with probability at
  least
  \begin{equation*}
    1 - \exp(-mt) - \exp(-c(1 - 2\pfail) m) -
    \exp\left(-c \frac{m \delta_q^2}{w_q^2}\right)
  \end{equation*}
  we have
  \begin{equation*}
    \dist(\what{d}, \{\pm d\subopt\})
    \le \frac{C \sqrt{\frac{n}{m} + t}}{
      \hinge{(1 - 2\pfail) \delta_q - C \sqrt{\frac{n}{m} + t}
        - M \pfail}}.
  \end{equation*}
\end{proposition}

For intuition, we provide a few simplifications of
Proposition~\ref{proposition:noisy-good-initial-direction} by bounding the
quantities $w_q$ and $\delta_q$ defined in Eq.~\eqref{eqn:q-fail}. Using the
conditional expectation bound in Lemma~\ref{lemma:ez-square} and a more
careful calculation for Gaussian random variables (see
Lemma~\ref{lemma:conditional-expectation-truncated-gaussian} in the
appendices) we have
\begin{equation*}
  \E \left[a_{i, 1}^2 \mid a_{i, 1}^2 \leq w_q^2 \right]
  \leq \min\left\{\frac{w_q^2}{3},
  1 - \half w_q^2 \P(a_{i,1}^2 \ge w_q^2)\right\} < 1,
\end{equation*}
so $\delta_q \ge \max\{1 - \frac{w_q^2}{3}, \half w_q^2 \P(W^2 \ge
w_q^2)\}$.  For $\pfail \le \frac{3}{8}$,
we may take $w_q^2 \le 2.71$
and $\delta_q > \frac{1}{11}$.  More generally, a standard Gaussian
calculation that $\Phi^{-1}(p) \ge \sqrt{|\log(1 - p)|}$ as $p \to 1$ shows
that $w_q^2 \ge \log \frac{8(1 - \pfail)}{1 - 2\pfail}$
as $\pfail \to \half$, so $\delta_q \ge \frac{1 -
  2\pfail}{8(1 - \pfail)} \log \frac{8(1 - \pfail)}{1 - 2\pfail}$.
Under Model~\ref{model:full-independence},
then, as long as the sample is large enough and $\pfail < \half$, we can
achieve constant accuracy in the directional estimate $\dist(\what{d},
\{\pm d\subopt\})$ with probability of failure $e^{-cm}$.
Under the more adversarial noise model~\ref{model:partial-independence}, we
require a bit more; more precisely, we must have $(1 - 2\pfail) \delta_q -
\pfail > 0$ to achieve accurate estimates.  A numerical calculation shows
that if $\pfail < \frac{1}{4}$, then this condition holds, so that under
Model~\ref{model:partial-independence} we can achieve constant accuracy
in the directional estimate $\dist(\what{d}, \{\pm d\subopt\})$ with high probability.



\subsection{Summary and success guarantees}

With the guarantees of stability, quadratic approximation, and good
initialization for suitably random matrices $A \in \R^{m \times n}$, we can
provide a theorem showing that the prox-linear approach to the composite
optimization phase retrieval objective succeeds with high probability.
Roughly, once $m/n$ is larger than a numerical constant,
the prox-linear method with noisy initialization succeeds with exponentially
high probability, even with outliers.  Combining the convergence
Theorem~\ref{theorem:quadratic-convergence-errors} with
Propositions~\ref{proposition:noise-stability},
\ref{proposition:good-direction-to-good-radius},
\ref{proposition:noisy-good-initial-direction}, and
Corollary~\ref{corollary:no-noise-quadratic-approximation},
we have the following theorem.
\begin{theorem}
  Let Assumptions~\ref{assumption:normal} hold. There exist numerical
  constants $c > 0$ and $C < \infty$ such that the following hold for any $t
  \ge 0$. Let the independent outliers Model~\ref{model:full-independence}
  hold and $\pfail < \frac{1}{\pi}$ or the adversarial outliers
  Model~\ref{model:partial-independence} hold and $\pfail < \frac{1}{4}$.
  Let $x_0$ be the initializer returned by
  Alg.~\ref{alg:noisy-initialization}, and assume the iterates $x_k$ of
  Alg.~\ref{alg:prox-linear} are generated without
  error. Then
  \begin{equation*}
    \dist(x_0, X\subopt) \le \frac{C}{(1 - 2\pfail)^2}
    \sqrt{\frac{n}{m} + t} \cdot \ltwo{x\subopt}
    ~~~ \mbox{and} ~~~
    \dist(x_k, X\subopt) \le
    \ltwo{x\subopt} 2^{-2^k} ~~\mbox{for all}~ k \in \N
  \end{equation*}
  with probability at least $1 - 4 e^{-mt} - e^{-c(1 - 2 \pfail)^2
    m}$.
\end{theorem}
\noindent
Thus, we see that the method succeeds with high probability as long as the
sample size is large enough, though there is non-trivial degradation 
when $\pfail$ is large.

%% file: optimization.tex

\section{Optimization methods}
\label{sec:optimization-methods}

In practice, we require some care to solve the
sub-problems~\eqref{eqn:prox-iteration},
to minimize
$f_x(y) + \frac{\lipconst}{2} \ltwo{x - y}^2$, at
large scale. In this section, we describe the three schemes we
use to solve the sub-problems (one is simply using
the industrial \texttt{Mosek} solver). We evaluate these more carefully in
the experimental section to come.

To fix notation, let $\phi(\cdot)$ be the elementwise square operator.
Recalling that $f(x) = \frac{1}{m} \lone{\phi(A x) - b}$,
we perform a few simplifications to more easily describe the
schemes we use to solve the prox-linear sub-problems.
Assuming we begin an iteration at point $x_0$,
defining the diagonal matrix $D = 2 \diag(A x_0) \in \R^{m \times
  m}$ and setting $c = b + \phi(A x_0)$, we recall
inequality~\eqref{eqn:approximation-opnorm} and have
\begin{align*}
  f(y)
  & \le f_{x_0}(y)
  + \frac{1}{m} (y - x_0)^T A^T A (y - x_0)
  \le \frac{1}{m} \lone{D A y - c}
  + \opnorm{m^{-1} A^T A} \ltwo{y - x_0}^2.
\end{align*}
Rewriting this with appropriately rescaled diagonal matrices $D$ and vector
$c$, implementing Algorithm~\ref{alg:prox-linear}
becomes equivalent to solving a sequence of optimization problems of the
form
\begin{equation}
  \label{eqn:optimization-we-solve}
  \minimize_x ~ \lone{D A x - c} + 
  \half \ltwo{x}^2.
\end{equation}

In small scale scenarios, the problem~\eqref{eqn:optimization-we-solve} is
straightforward to solve via standard interior point method software; we use
\texttt{Mosek} via the \texttt{Convex.jl} toolbox~\cite{UdellMoZeHoDiBo14}.
We do not describe this further.  In larger-scale scenarios, we use more
specialized methods, which we now describe.


\subsection{Graph splitting methods for the prox-linear sub-problem}
\label{sec:pogs}

\newcommand{\mult}{\rho}
\newcommand{\residp}{r^{\rm pri}}
\newcommand{\residd}{r^{\rm dual}}

When the matrices are large, we use a variant of the Alternating Directions
Method of Multipliers (ADMM) procedure known as the proximal operator graph
splitting (POGS) method of Parikh and Boyd~\cite{ParikhBo13,ParikhBo14},
which minimizes objectives of the form $f(x) + g(y)$ subject to a linear
constraint $Bx = y$.  We experimented with a number of specialized
first-order and interior-point methods for solving
problem~\eqref{eqn:optimization-we-solve}, but
in our experience, POGS offers the
empirically best performance. Let us describe the POGS method.
Let the matrix $B = DA$ for shorthand; evidently,
problem~\eqref{eqn:optimization-we-solve} has precisely this form and is
equivalent to
\begin{equation*}
  \minimize_{x \in \R^n, y \in \R^m}
  ~ \lone{y - c} + \half \ltwo{x}^2
  ~~ \subjectto B x = y.
\end{equation*}
The POGS method iterates to solve problem~\eqref{eqn:optimization-we-solve}
as follows.  Introduce dual variables $\lambda_k \in \R^n$ and $\nu_k
\in \R^m$ associated to $x$ and $y$, and consider the iterations
\begin{equation}
  \begin{split}
    x_{k + \half} & =
    \argmin_x \left\{\half \ltwo{x}^2
    + \frac{\mult}{2} \ltwo{x - (x_k - \lambda_k)}^2 \right\} \\
    y_{k + \half} & =
    \argmin_y \left\{\lone{y - c}
    + \frac{\mult}{2} \ltwo{y - (y_k - \nu_k)}^2 \right\} \\
    \left[\begin{matrix} x_{k + 1} \\
        y_{k + 1} \end{matrix}\right]
    & = \left[\begin{matrix} I_n & B^T \\ B & -I_m \end{matrix}\right]^{-1}
    \left[\begin{matrix} I_n & B^T \\ 0 & 0 \end{matrix}\right]
    \left[\begin{matrix} x_{k + \half} + \lambda_k \\
        y_{k + \half} + \nu_k \end{matrix} \right] \\
    \lambda_{k+1} & = \lambda_k + (x_{k + \half} - x_{k + 1}),
    ~~ \nu_{k+1} = \nu_k + (y_{k + \half} - y_{k + 1}).
  \end{split}
  \label{eqn:pogs-update}
\end{equation}
Each of the steps of the method~\eqref{eqn:pogs-update}
is trivial except for the matrix inversion, or the ``graph projection''
step, which projects the pair
$(x_{k + \half} + \lambda_k, y_{k + \half} + \nu_k)$ to the
set $\{x, y : Bx = y\}$.
The first two updates amount to
\begin{equation*}
  x_{k + \half} = \frac{\mult}{1 + \mult}
  (x_k - \lambda_k)
  ~~ \mbox{and} ~~
  y_{k + \half} =
  c + \sign(y_k - \nu_k - c) \odot
    \hinge{|y_k - \nu_k - c| - 1/\mult}
\end{equation*}
where $\odot$ denotes elementwise multiplication and each operation
is element-wise.
The matrix $B$ is tall, so setting $v_k
= x_{k + \half} + \lambda_k + B^T (y_{k + \half} + \nu_k) \in \R^n$, then
the solution of the system
\begin{equation}
  \left[\begin{matrix} I_n & B^T \\ 
      B & -I_n \end{matrix} \right]
  \left[\begin{matrix} x \\ y \end{matrix} \right]
  = \left[\begin{matrix} v_k \\ 0_n \end{matrix}\right]
  ~~ \mbox{is} ~~
  x_{k+1} = (I_n + B^T B)^{-1} v_k
  ~~ \mbox{and} ~~
  y_{k+1} = B x_{k+1}.
  \label{eqn:tall-matrix-update}
\end{equation}
In this iteration, it is straightforward to cache the matrix $(I_n + B^T
B)^{-1}$, or a Cholesky factorization of the matrix, so that we can
repeatedly compute the multiplication~\eqref{eqn:tall-matrix-update} in time
$n^2 + nm$.

Following Parikh and Boyd~\cite{ParikhBo14}, we define the primal and dual residuals
\begin{equation*}
  \residp_{k + 1} \defeq
  \left[ \begin{matrix} x_{k+1} - x_{k + \half} \\
      y_{k+1} - y_{k + \half} \end{matrix}\right]
  ~~ \mbox{and} ~~
  \residd_{k + 1} \defeq
  \mult \left[\begin{matrix} x_k - x_{k + 1} \\ y_k - y_{k + 1}
      \end{matrix}\right].
\end{equation*}
These residuals define a natural stopping criterion~\cite{ParikhBo14}, where
one terminates the iteration~\eqref{eqn:pogs-update} once the residuals
satisfy
\begin{equation}
  \label{eqn:residual-stopping}
  \ltwo{\residp_k} < \epsilon \left(\sqrt{n}
  + \max\{\ltwo{x_k}, \ltwo{y_k}\}\right)
  ~~ \mbox{and} ~~
  \ltwo{\residd_k} < \epsilon \left(\sqrt{n}
  + \max\{\ltwo{\lambda_k},\ltwo{\nu_k}\}\right)
\end{equation}
for some $\epsilon > 0$, which must be specified. In our case, the quadratic
convergence guarantees of Theorem~\ref{theorem:quadratic-convergence-errors}
suggest a strategy of decreasing $\epsilon$ across
iterative solutions of the sub-problem~\eqref{eqn:optimization-we-solve}.
That is, we begin with some
$\epsilon = \epsilon_0$. We perform iterations of the prox-linear
Algorithm~\ref{alg:prox-linear} using the POGS
iteration~\eqref{eqn:pogs-update} (until the residual
criterion~\eqref{eqn:residual-stopping} holds) to solve the inner
problem~\eqref{eqn:optimization-we-solve}. Periodically, in the outer
prox-linear iterations of Alg.~\ref{alg:prox-linear},
we decrease $\epsilon$ by some large multiple.

In our experiments with the POGS method for sub-problem solutions, we
perform two phases. In the first phase, we iterate the prox-linear method
(Alg.~\ref{alg:prox-linear}) until either $k = 25$ or
$\ltwo{x_k - x_{k + 1}} \le \delta / \opnorm{(1/m) A^TA}$, where $\delta =
10^{-3}$, using accuracy parameter $\epsilon = 10^{-5}$ for POGS (usually,
the iterations terminate well-before $k = 25$). We then decrease this
parameter to $\epsilon = 10^{-8}$, and begin the iterations again.

\subsection{Conjugate gradient methods for sub-problems}
\label{sec:cg-sub-prob}

In a number of scenarios, it is possible to multiply by the matrix $A$
quickly, though the direct computation $(I_n + A^T D^2 A)^{-1}$ in
expression~\eqref{eqn:tall-matrix-update} (recall $B = DA$) may be
difficult. For example, if $A$ is structured (say, a Fourier or Hadamard
transform matrix or or sparse),
computing multiplication by $I_n + A^T D^2 A$ quickly is possible. This
suggests~\cite[Part VI]{TrefethenBa97} using conjugate gradient methods.

We make this more explicit here to mesh with our experiments to come.
Let $H_n$ be an $n \times n$ orthogonal matrix for which computing the
multiplication $H_n v$ is efficient (e.g.\ a Hadamard
or discrete cosine transform matrix).  We
assume that the measurement matrix $A$ takes the form of repeated randomized
measurements under $H_n$, that is,
\begin{equation}
  A = \left[\begin{matrix}
      H_n S_1 & H_n S_2 & \cdots & H_n S_k \end{matrix}\right]^T
  \label{eqn:hadamard-sensing}
\end{equation}
where $S_l \in \R^{n \times n}$ are (random) diagonal matrices, so that $A
\in \R^{m \times n}$ with $m = kn$.
In this case, the expensive part of the
system~\eqref{eqn:tall-matrix-update} is the solution of
$(I + A^T D^2 A) x = v$.
This is a positive definite system, and it is
possible to compute the multiplication $(I + A^T D^2 A)x$ in time $O(k n + k
T_{\rm mult})$, where $T_{\rm mult}$ is the time to muliply an $n$-vector by
the matrix $H_n$ and $H_n^T$.  When $S$ is a random sign matrix
and $H_n$ is a Hadamard or FFT matrix, the matrix $A^TD^2 A$
is
well-conditioned, as the analysis of random (subsampled)
Hadamard and Fourier transforms shows~\cite{Tropp11b}. Thus, in our
experiments with structured matrices we use the conjugate gradient method
\emph{without} preconditioning~\cite{TrefethenBa97}, iterating until we have
relative error $\norm{(I + A^T D A)x - v} \le \epsilon_{\rm cg} / \norm{v}$
where $\epsilon_{\rm cg} = 10^{-6}$.



%% file: experiments.tex

\section{Experiments}
\label{sec:experiments}

We perform a number of simulations as well as experiments on real
images to evaluate our method and compare with other state-of-the-art
(non-convex) methods for real-valued phase retrieval problems. To standardize notation
and remind the reader, in each experiment, we generate data via
(variants) of the following process. We take a measurement matrix $A \in
\R^{m \times n}$, from one of a few distributions and generate a signal
$x\subopt \in \R^n$ either by drawing $x\subopt \sim \normal(0, I_n)$ or by taking
$x\subopt \in \{-1, 1\}^n$ uniformly at random. We receive observations of the
form
\begin{equation*}
  b_i = \<a_i, x\subopt\>^2,
\end{equation*}
and with probability $\pfail \in [0, \half]$, we corrupt the measurements
$b_i$.

In our experiments, to more carefully isolate the relative performance of
the iterative algorithms, rather than initialization used, we compare three
initializations.  The first two rely on Gaussianity of the
measurement matrix $A$.
\begin{enumerate}[(i)]
\item\label{item:big} \textbf{Big:} 
  The initialization of Wang et al.~\cite{WangGiEl16}. Defining
  \begin{equation*}
    \mc{I}_0 \defeq \left\{i \in [m] : b_i / \ltwo{a_i}^2
    \ge \quant_{5/6}(\{b_i / \ltwo{a_i}^2\})\right\}
    ~~ \mbox{and} ~~
    X\init \defeq \sum_{i \in \mc{I}_0} \frac{1}{\ltwo{a_i}^2} a_i a_i^T,
  \end{equation*}
  we set the direction $\what{d} = \argmax_{\ltwo{d} = 1} d^T X\init
  d$ and $x_0 = (\frac{1}{m} \sum_{i = 1}^m b_i)^\half \what{d}$.
\item \label{item:med} \textbf{Median:} The initialization of Zhang et al.~\cite{ZhangChLi16}. 
	We set $\what{r}^2 = \quant_\half(\{b_i\}) / .455$ and define
  \begin{equation*}
    \mc{I}_0 \defeq \left\{i \in [m] : |b_i| \le 9 \lambda_0\right\}
    ~~ \mbox{and} ~~
    X\init \defeq \sum_{i \in \mc{I}_0} b_i a_i a_i^T.
  \end{equation*}
  We then set the direction $\what{d} = \argmax_{\ltwo{d} = 1} d^T X\init
  d$ and $x_0 = \what{r} \what{d}$.
\item \label{item:small} \textbf{Small:} The
  outlier-aware initialization we describe in
  Sec.~\ref{sec:noise-initialization}.
\end{enumerate}

Our convergence results in Section~\ref{sec:composite} rely on four
quantities: the quality of the initialization, $\dist(x_0, X\subopt)$; the
stability parameter $\lambda$ such that $f(x)
- f(x\subopt) \ge \lambda \ltwo{x - x\subopt} \ltwo{x + x\subopt}$; the quadratic
upper bound guaranteed by our linearized models $f_x(y)$, so that $|f_x(y) -
f(y)| \le \frac{\lipconst}{2} \ltwo{x - y}^2$; and the accuracy to which we
solve the optimization problems.  The random matrix $A$ governs the middle
two quantities---stability $\lambda$ and closeness $\lipconst$---and we take
$\lipconst = \frac{2}{m}\opnorm{A}^2$. Thus, in our experiments we directly
vary the initialization scheme to generate $x_0$ and the accuracy to which
we solve the sub-problems~\eqref{eqn:prox-iteration}, that is,
\begin{equation*}
  x_{k + 1} = \argmin_x \left\{
  \frac{1}{m} \sum_{i = 1}^m \left|\<a_i, x_k\>^2
  + 2 \<a_i, x_k\> \<a_i, x - x_k\>
  - b_i \right| + \frac{\Lipconst}{2} \ltwo{x - x_k}^2\right\}.
\end{equation*}

We perform each of our experiments on a server with a 16 core 2.6 GhZ Intel
Xeon processor with 128 GB of RAM using \texttt{julia} as our language, with
\texttt{OpenBLAS} as the BLAS library. We restrict each optimization method
to use 4 of the cores (\texttt{OpenBLAS} is multi-threaded).

\subsection{Simulations with zero noise and Gaussian matrices}
\label{sec:simulations-gaussians}

For our first collection of experiments, we evaluate the performance of our
method for recovering random signals $x\subopt \in \R^n$ using Gaussian
measurement matrices $A \in \R^{m \times n}$, varying the number of
measurements $m$ over the ten values $m \in \{1.8 n, 1.9 n, \ldots, 2.7
n\}$. (Taking $m \ge 2.7n$ yielded 100\% exact recovery in all the methods
we experiment with.) We assume a noiseless measurement model, so that for $A
= [a_1 ~ \cdots ~ a_m]^T \in \R^{m \times n}$, we have $b_i = \<a_i,
x\subopt\>^2$ for each $i \in [m]$. We perform experiments with dimensions $n
\in \{400, 600, 800, 1000, 1500, 2000, 3000\}$, and within each experiment
we vary a number of problem parameters, including initialization scheme and
the algorithm we use to solve the sub-problems~\eqref{eqn:prox-iteration}.

To serve as our baseline for comparison, we use
Truncated Amplitude Flow (TAF)~\cite{WangGiEl16}, given that it outperforms
other non-convex iterative methods for phase retrieval, including Wirtinger
Flow~\cite{CandesLiSo15}, Truncated Wirtinger Flow~\cite{ChenCa15}, and
Amplitude Flow~\cite{WangGiEl16}.  Setting $\psi_i = \sqrt{b_i}$, TAF tries
to minimize the loss $\frac{1}{2 m} \sum_{i = 1}^m (\psi_i - |\<a_i,
x\>|)^2$ via a carefully designed (generalized) gradient method.  For
convenience, we replicate the method (as described by Wang et al. \cite{WangGiEl16}) in
Alg.~\ref{alg:taf}.
\begin{algorithm}[th]
  \caption{\label{alg:taf}
    Truncated amplitude flow}
  \KwData{Initializer $x_0 \in \R^n$, i.i.d.\ standard Gaussian
    matrix $A \in \R^{m \times n}$, $\psi_i = |\<a_i, x\subopt\>|$,
    and parameters $\gamma = .7$ and $\stepsize = .6$}
  \For{$k = 0, 1, \ldots, K - 1$}{
    Set $\mc{I}_k = \{i \in [m] : |\<a_i, x_k\>| \ge \frac{1}{
      1 + \gamma} \psi_i \}$ \\
    Update $x_{k+1} = x_k
    - \frac{\stepsize}{m} \sum_{i \in \mc{I}_k}
    \left(\<a_i, x_k\> - \psi_i \frac{\<a_i, x_k\>}{|\<a_i, x_k\>|}
    \right) a_i$}
  \Return $x_K$
\end{algorithm}

\subsubsection{Low-dimensional experiments with accurate prox-linear steps}

We begin by describing our experiments with dimensions $n \in \{400, 600,
800\}$. For each of these, we solve the iterative sub-problems to machine
precision, using \texttt{Mosek} and the Julia package
\texttt{Convex.jl}~\cite{UdellMoZeHoDiBo14}.  We plot representative results
for $n = 400$ and $n = 800$ in Figure~\ref{fig:smalldim-zero-noise}, which
summarize 400 independent experiments, and we generate the true $x\subopt$ by
taking $x\subopt \in \{-1, 1\}^n$ uniformly at random.  (In separate
experiments, we drew $x\subopt \sim \normal(0, I_n)$, and the results were
essentially identical.) In these figures, we use the ``big''
initialization~\eqref{item:big} (the initialization of
Wang et al., as it yields the best empirical
performance. Following Wang et al. \cite{WangGiEl16}, we perform 1000 iterations of TAF
(Alg.~\ref{alg:taf}), and within each experiment, both TAF and the
prox-linear method use identical initializer and data matrix $A$.  We run
the prox-linear method until sequential updates $x_k$ and $x_{k + 1}$
satisfy $\ltwo{x_k - x_{k + 1}} \le \epsilon = 10^{-5}$, which in
\emph{every} successful experiment we perform (with these data settings)
requires 6 or fewer iterations. We declare an experiment successful if the
output $\what{x}$ of the algorithm satisfies
\begin{equation}
  \label{eqn:def-success}
  \dist(\what{x}, X\subopt)
  = \min\left\{\ltwos{\what{x} - x\subopt},
  \ltwos{\what{x} - x\subopt}\right\}
  \le \epsilon_{\rm acc} \ltwos{x\subopt}
  ~~ \mbox{where} ~~
  \epsilon_{\rm acc} = 10^{-5}.
\end{equation}
In Fig.~\ref{fig:smalldim-zero-noise}(a), we plot the number of times that
one method succeeds (out of the 400 experiments) while the other does not as
a function of the ratio $m/n$. We see that for these relatively small
dimensional problems, the prox-linear method has mildly better performance
for $m/n$ small than does truncated amplitude flow.  In
Fig.~\ref{fig:smalldim-zero-noise}(b), we plot the fraction of successful
runs of the algorithms, again against $m/n$ for $n = 400$, and in
Fig.~\ref{fig:smalldim-zero-noise}(c) we plot the fraction of successful
runs for $n = 800$. Even when $m/n = 2$, the prox-linear method
has success rate of around .6.

\begin{figure}[t]
  \begin{center}
    \begin{tabular}{ccc}
      \hspace{-.3cm}
      \includegraphics[width=.34\columnwidth]{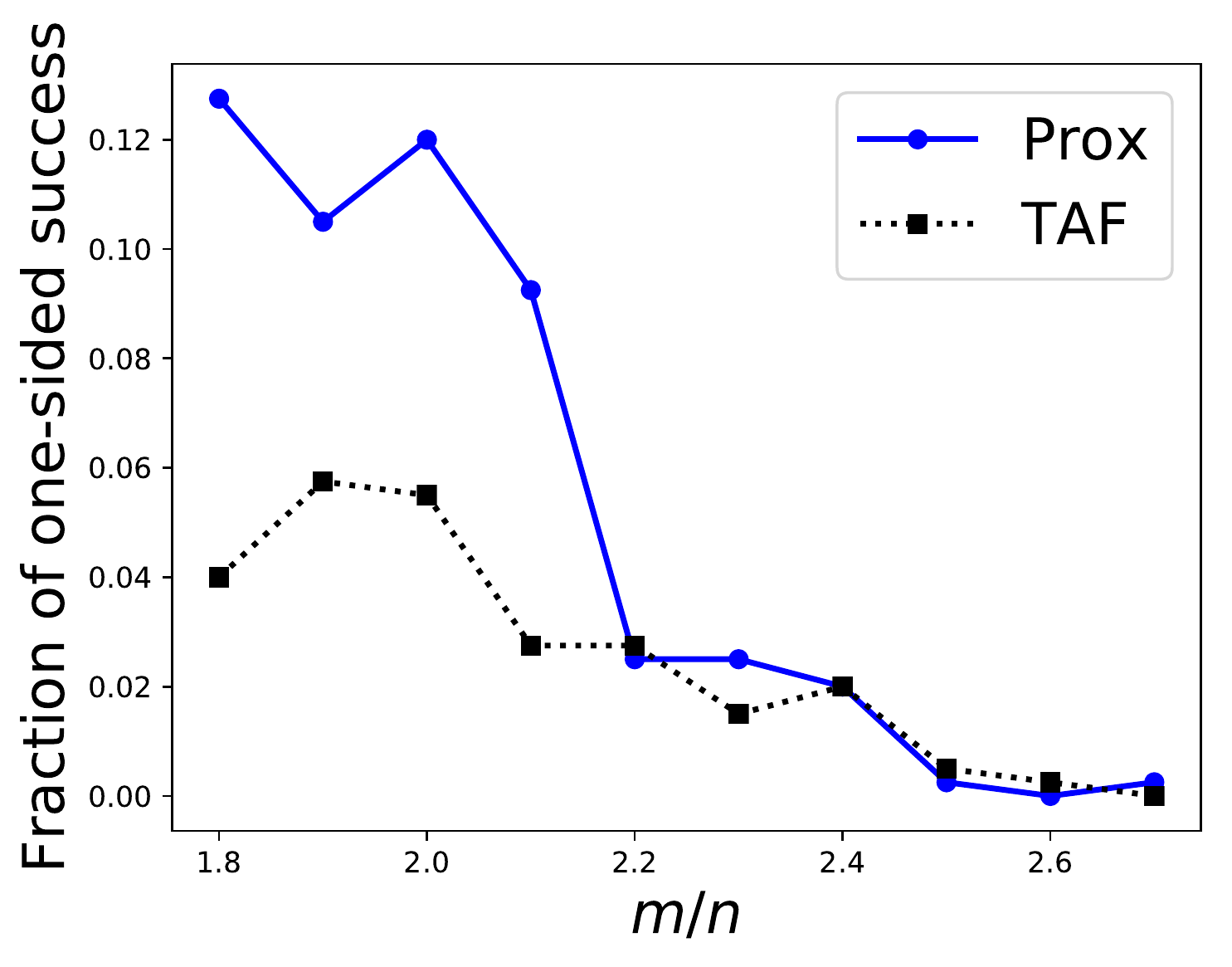} &
      \hspace{-.5cm}
      \includegraphics[width=.34\columnwidth]{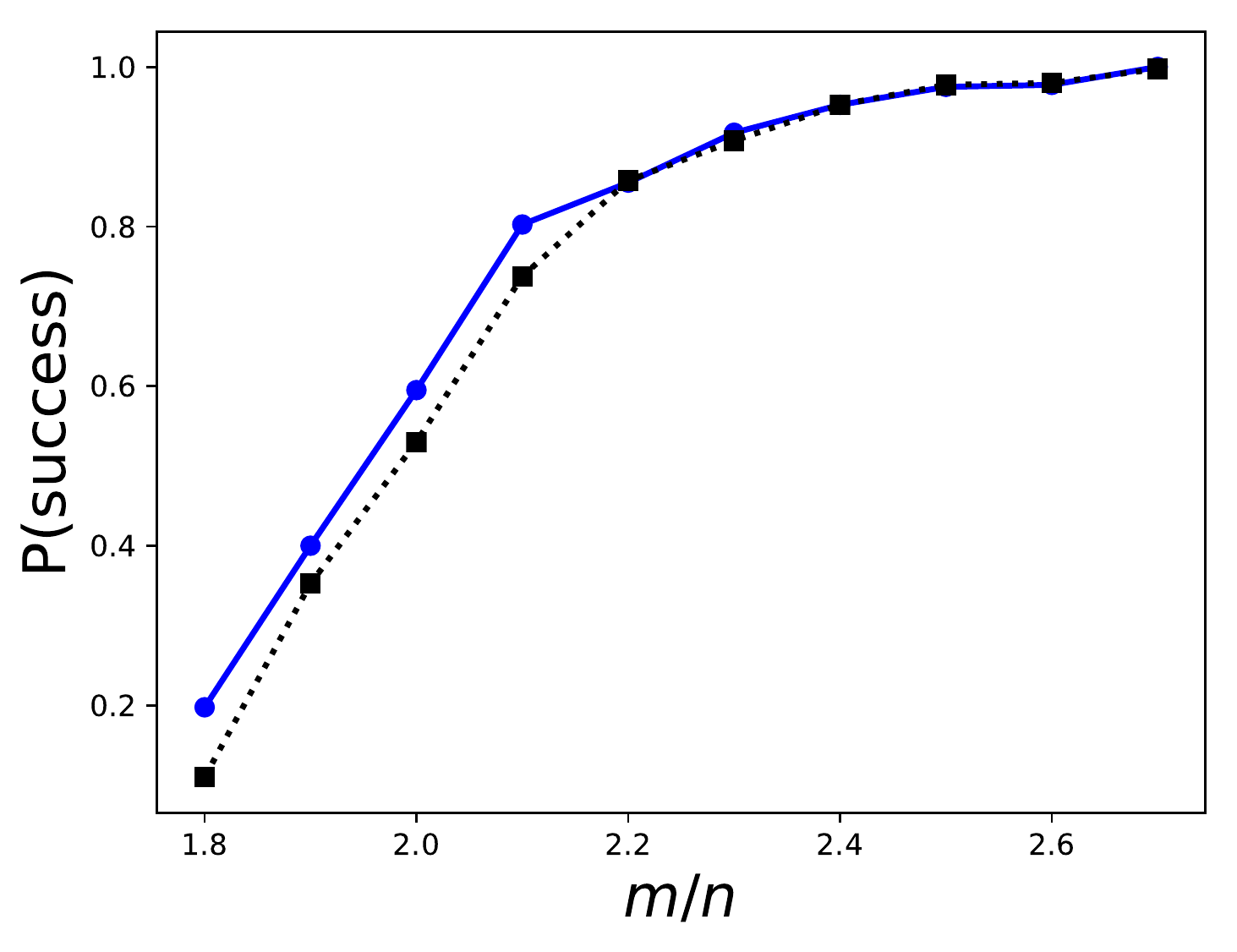} &
      \hspace{-.5cm}
      \includegraphics[width=.34\columnwidth]{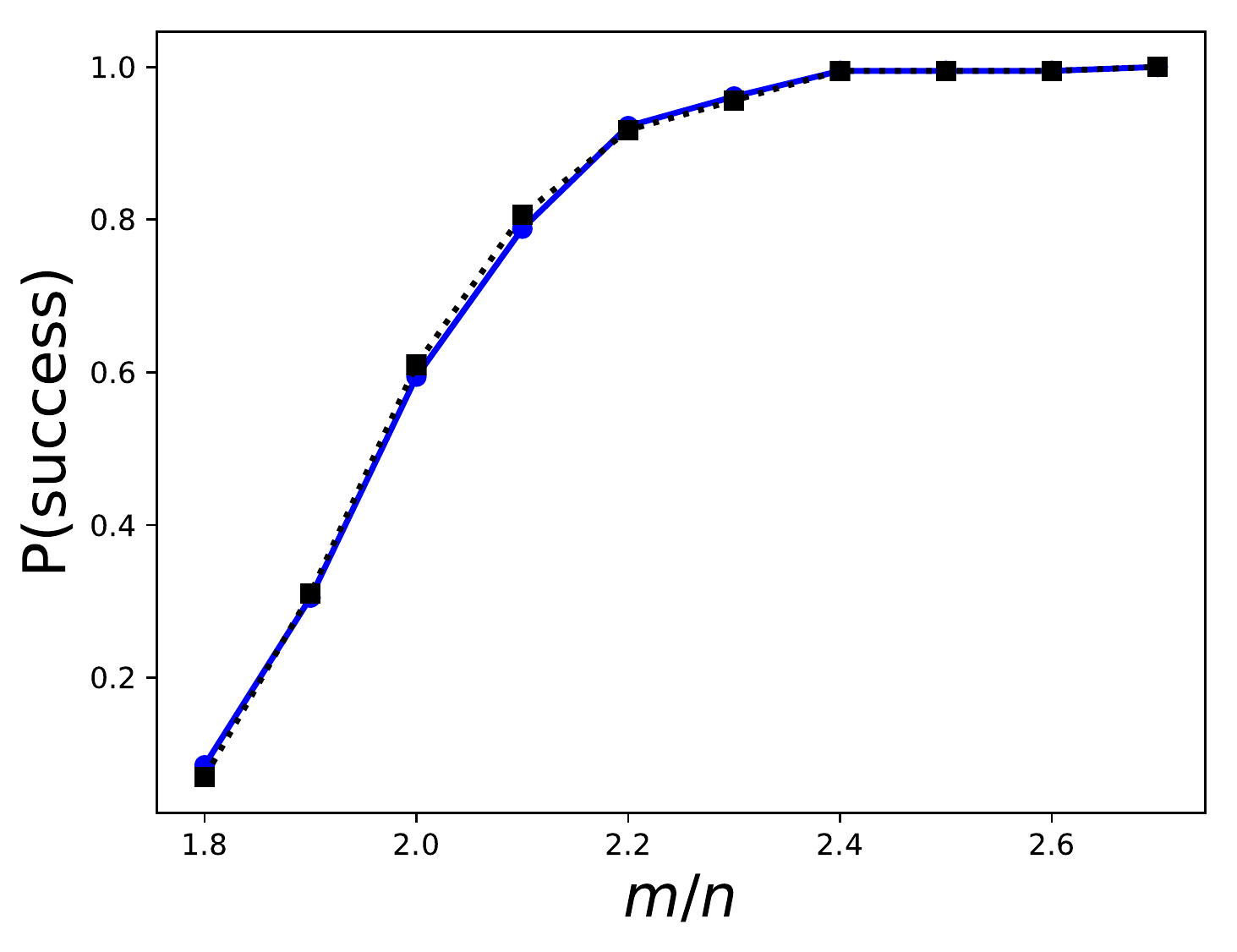} \\
      (a) & (b) & (c)
    \end{tabular}
    \caption{\label{fig:smalldim-zero-noise}
      Zero-noise experiments in dimensions $n = 400$ and $n = 800$.
      (a) Fraction of times (in 400 experiments) that
      one method succeeded while other failed for $n = 400$.
      (b) Fraction of successful recoveries for $n = 400$.
      (c) Fraction of successful recoveries for $n = 800$.
    }
  \end{center}
\end{figure}

\subsubsection{Medium dimensional experients with inaccurate
  prox-linear steps}

We now shift to a description of our experiments in dimensions $n \in
\{1000, 1500, 2000, 3000\}$, again without noise in the measurements; we
perform between 100 and 400 experiments for each dimension in this regime,
and we use the ``big'' initialization~\eqref{item:big}~\cite{WangGiEl16}. In
this case, we use Parikh and Boyd's proximal operator graph
splitting (POGS) method~\cite{ParikhBo13} for the prox-linear steps, as we
describe in Section~\ref{sec:pogs}. We
perform two phases of the prox-linear method: the first performing POGS
until the residual errors~\eqref{eqn:residual-stopping} are less than
$\epsilon = 10^{-5}$ within the prox-linear steps, the second to accuracy
$\epsilon = 10^{-8}$. We apply the prox-linear method to the matrix $A /
\sqrt{m}$ with data $b / m$, which is equivalent but numerically more stable
because $A / \sqrt{m}$ and $I_n$ are comparable (see the
recommendations~\cite{ParikhBo13}). In these accuracy
regimes, the time for solution of the prox-linear method and that required
for 1000 iterations of truncated amplitude flow are comparable; TAF requires
about 1.5 seconds while the two-phase prox-linear method requires around 4
seconds in dimension $n = 1000$, and TAF requires about 5 seconds while the
the two-phase prox-linear method requires around 20 seconds in dimension $n
= 2000$, each with $m = 2n$.

\begin{figure}[t]
  \begin{center}
    \begin{tabular}{ccc}
      \hspace{-.3cm}
      \includegraphics[width=.34\columnwidth]{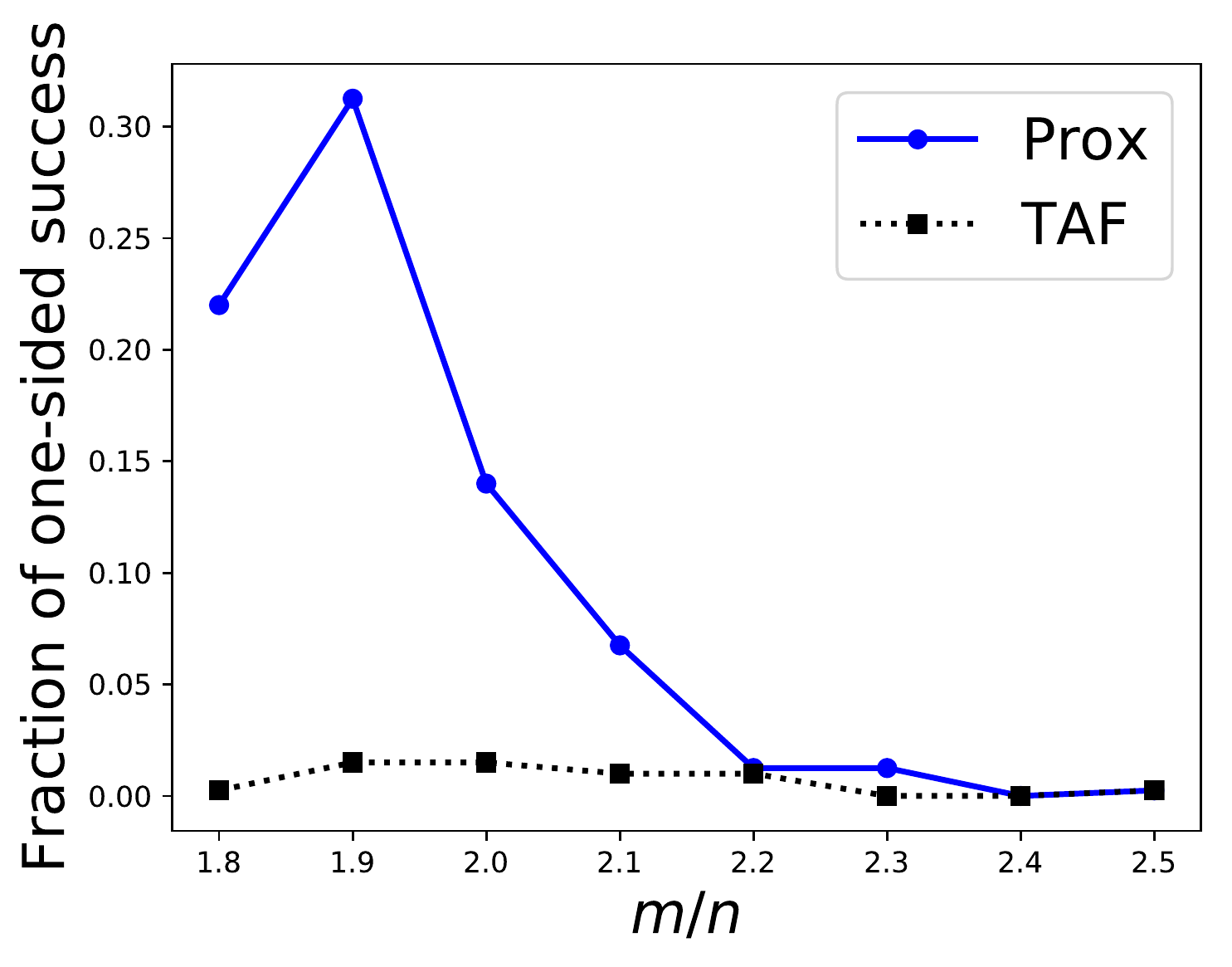} &
      \hspace{-.3cm}
      \includegraphics[width=.34\columnwidth]{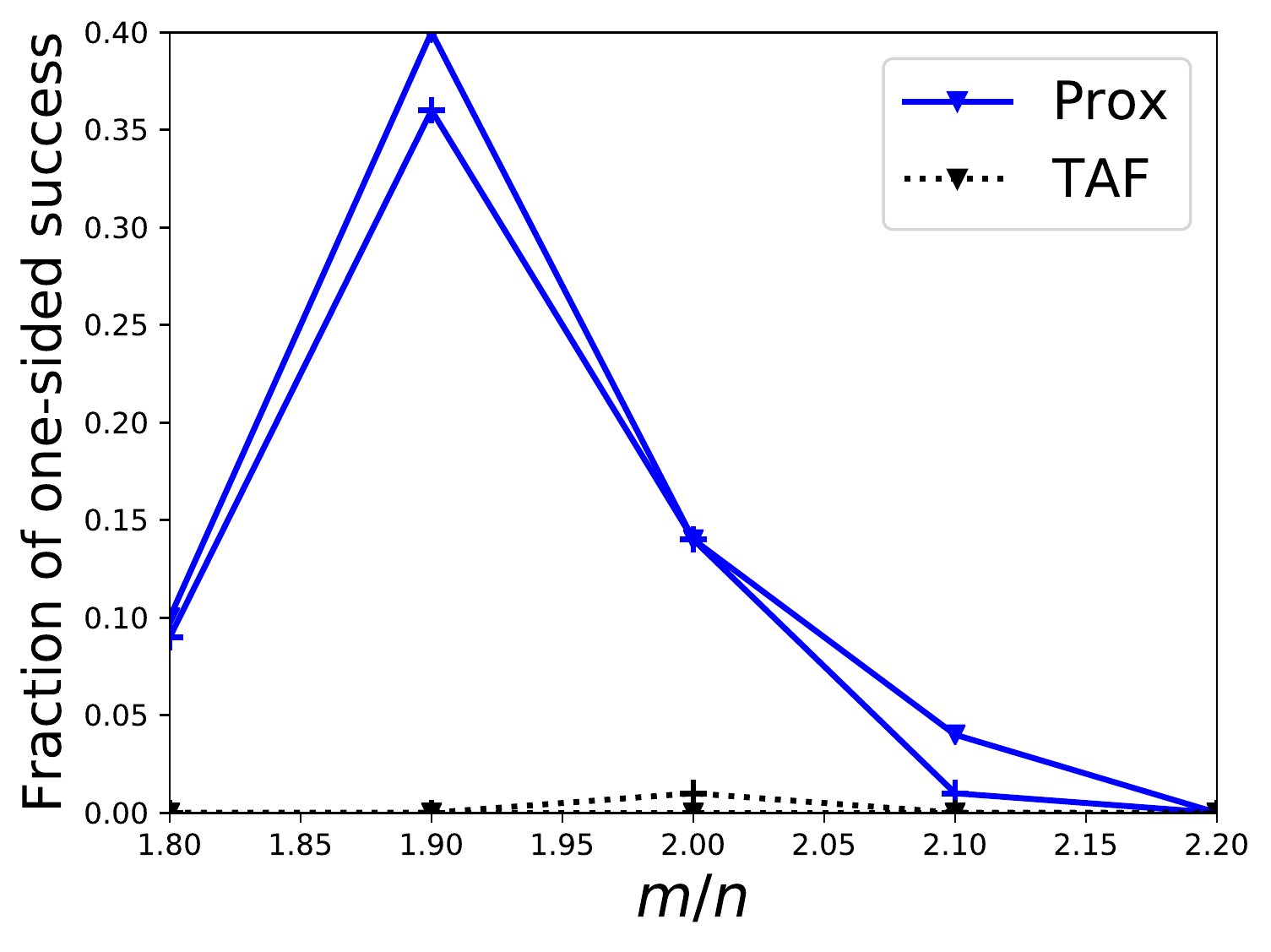} &
      \hspace{-.3cm}
      \includegraphics[width=.34\columnwidth]{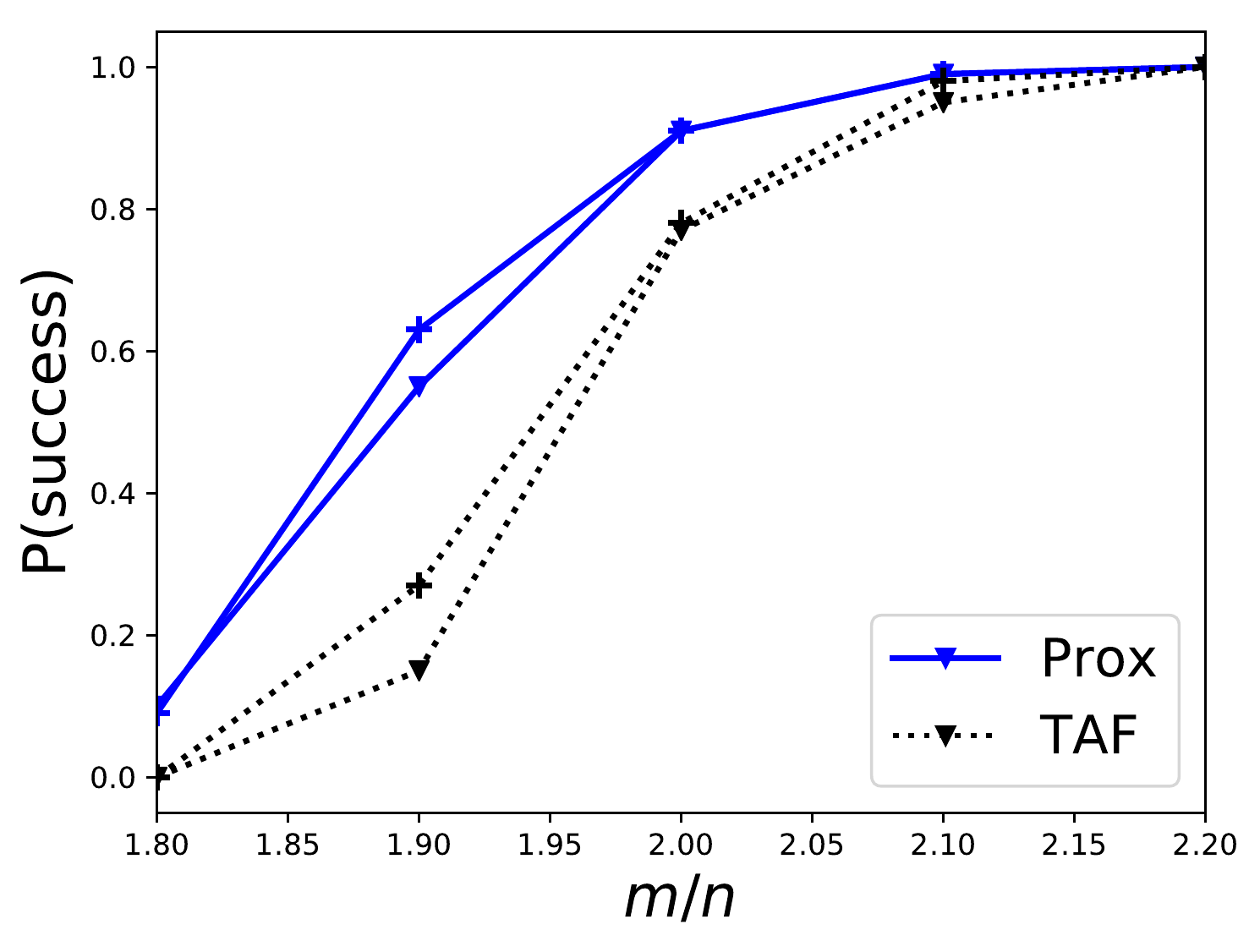} \\
      (a) & (b) & (c)
    \end{tabular}
    \caption{\label{fig:big-dim-zero-noise} Zero-noise experiments.
      (a) Fraction of trials (of 400) in which
      one method succeeds while other fails for $n = 1000$.
      (b) Fraction of trials (of 100) in which
      one method succeeds while other fails for $n = 3000$.
      (c) Fraction of successful recoveries for $n = 3000$.
      }
  \end{center}
\end{figure}

We provide two sets of plots for these results, where again we measure
success by the criterion~\eqref{eqn:def-success}, $\ltwo{x \pm x\subopt} \le
\epsilon_{\rm acc} = 10^{-5}$. In the first,
Fig.~\ref{fig:big-dim-zero-noise}, we show performance of prox-linear
against TAF specifically for dimensions $n = 1000$ and $3000$.  In
Fig.~\ref{fig:big-dim-zero-noise}(a), we plot the number of trials in which
one method succeeds (out of 400 experiments) while the other does not as a
function of the ratio $m/n$. In Fig.~\ref{fig:big-dim-zero-noise}(b) we plot
the number of trials in which the prox-linear or TAF method succeeds, while
the other does not, for $n = 3000$ with $x\subopt$ chosen either $\normal(0,
I_n)$ or uniform in $\{\pm 1\}^n$ ($\triangledown$ and $+$ markers,
respectively). We ignore ratios $m/n \ge 2.2$ as both methods succeed in all
of our trials. Out of 100 trials, there is only \emph{one} (with $m/n = 2$)
in which TAF succeeds but the prox-linear method does not.  In
Fig.~\ref{fig:big-dim-zero-noise}(c), we plot the fraction of successful
runs of the algorithms, again against $m/n$ for $n = 3000$. For these larger
problems, there is a substantial gap in recovery probability between
prox-linear method and TAF, where with $m/n = 2$ the prox-linear method
achieves recovery more than 78\% ($\pm 4$), 88\% ($\pm 6$), and 91\% ($\pm
6$) of the time, with 95\% confidence intervals, for $n = 1000, 2000$, and
$3000$, respectively.

\includelong{
In Figure~\ref{fig:alldims-zero-noise}, we plot results for all of our
experiments with $n \in \{1000, 1500, 2000\}$ simultaneously, where we
include both $x\subopt \sim \uniform(\{-1, 1\}^n)$ and $x\subopt \sim \normal(0,
I_n)$, and all experiments with the initialization. The left plot shows the
fraction of the experiments in which one method succeeds while the other
does not. The other plots the probability of success of the methods versus
the ratio $m/n$. Each line consists of the data of between 100 $(n = 2000,
3000)$ and 400 ($n = 1000, 1500$) experiments. The prox-linear method
appears to have stronger performance and more frequent recovery, even with
the same data and same initialization.

\begin{figure}[t]
  \begin{center}
    \begin{tabular}{cc}
      \hspace{-.4cm}
      \includegraphics[width=.52\columnwidth]{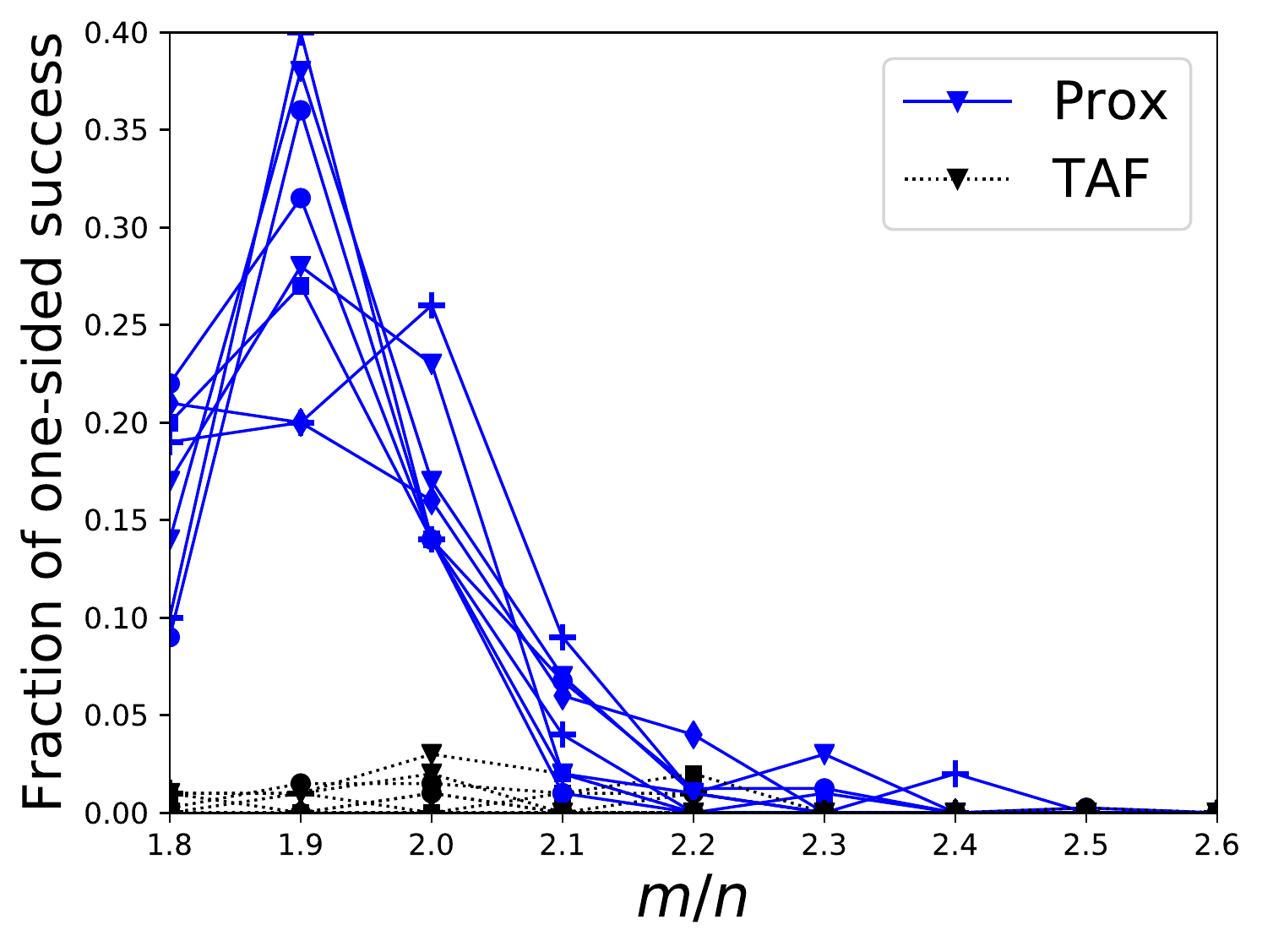} &
      \hspace{-.4cm}
      \includegraphics[width=.52\columnwidth]{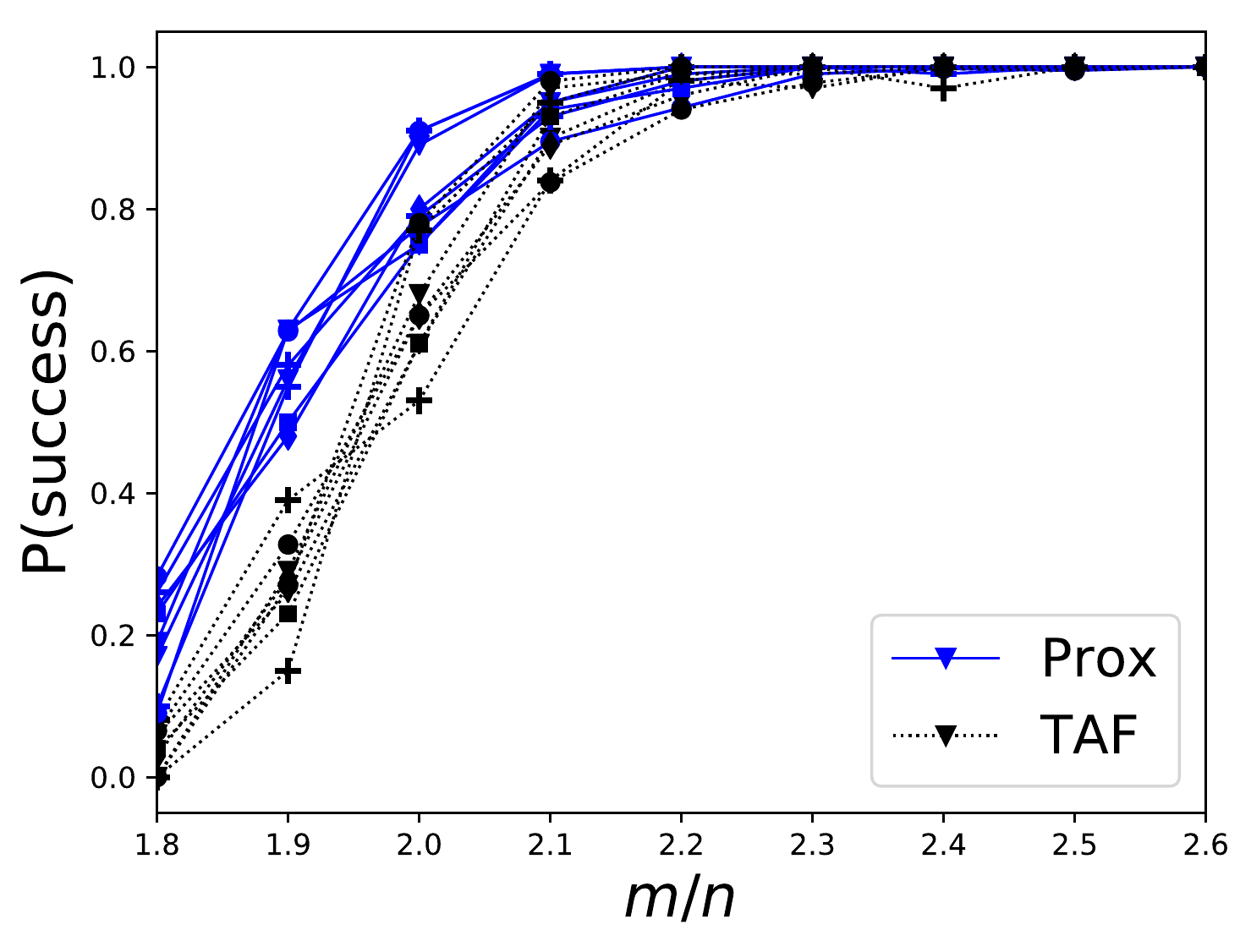}
    \end{tabular}
    \caption{\label{fig:alldims-zero-noise}
      Collected zero-noise experiments in dimensions
      $n \in \{1000, 1500, 2000\}$.}
  \end{center}
\end{figure}
}

\subsection{Phase retrieval with outlying measurements}

One of the advantages we claim for the
objective~\eqref{eqn:optimization-problem} is that it is robust to
outliers. Zhang et al.~\cite{ZhangChLi16} develop a method, which they term
median-truncated Wirtinger flow, for handling outlying measurements, which
(roughly) sets the index set $\mc{I}_k$ used for updates in
Alg.~\ref{alg:taf} to be a set of measurements near the median values of the
$b_i$. Their algorithm requires a number of parameters that strongly rely
on the assumptions that $a_i \simiid \normal(0, I_n)$;
in contrast, the
objective~\eqref{eqn:optimization-problem} and prox-linear method we
investigate are straightforward to implement without any particular
assumptions on $A$ (and require no
parameter tuning, relying on an explicit sequence of convex optimizations).

Nonetheless, to make comparisons between the algorithms as fair as possible,
we implement their procedure and perform experiments in dimension $n \in
\{100, 200\}$ with i.i.d.\ standard normal data matrices $A$.  We perform
100 experiments as follows. Within each experiment, we evaluate each $m \in
\{1.8 n, 2n, 3n, 4n, 6n, 8n\}$ and failure probability
$\pfail \in \{0, .01, .02, \ldots, .29, .3\}$. For fixed $m, n$, we draw a
data matrix $A \in \R^{m \times n}$, then choose $x\subopt \in \R^n$ either by
drawing $x\subopt \sim \uniform(\{-1, 1\}^n)$ or $x\subopt \sim \normal(0,
I_n)$. We then generate $b_i = \<a_i, x\subopt\>^2$ for $i \in [m]$. For our
experiments with $n = 100$, we simply set $b_i = 0$, which is more difficult
for our initialization strategy, as it corrupts a large fraction of the
vectors $a_i$ used to initialize $x_0$. For our experiments with $n = 200$,
we draw $b_i$ from a Cauchy distribution.  Each problem setting ($b_i$
Cauchy vs.\ zeroing and $x\subopt$ discrete or normal) yields qualitatively
similar results.

Figures~\ref{fig:outlier-recoveries}, \ref{fig:outlier-200-recoveries},
and~\ref{fig:outlier-single-performances} summarize our results. In
Figures~\ref{fig:outlier-recoveries} and~\ref{fig:outlier-200-recoveries},
we display success rates of the median truncated Wirtinger flow (the left
column in each figure) and our composite optimization-based procedure (right
column) for a number of initializations; each plot represents results of 100
experiments.  Within each plot, a white square indicates that 100\% of
trials were successful, meaning the signal is recovered to accuracy
$\dist(\what{x}, X\subopt) \le 10^{-5} \ltwo{x\subopt}$, while black squares
indicate 0\% success rates. Within each row of the figure, we present
results for the two methods using the same initialization
scheme. Figure~\ref{fig:outlier-recoveries} gives results with $n = 100$,
using precise (\texttt{Mosek}-based) solutions of the prox-linear updates,
while Figure~\ref{fig:outlier-200-recoveries} gives results with $n = 200$
using the POGS-based updates (recall step~\eqref{eqn:pogs-update}), with
identical parameters as in the previous section. It is clear from the
figures that the composite objective yields better recovery.

\begin{figure}[t]
  \begin{center}
    \begin{tabular}{cc}
      \multicolumn{2}{c}{
        \includegraphics[width=.7\columnwidth]{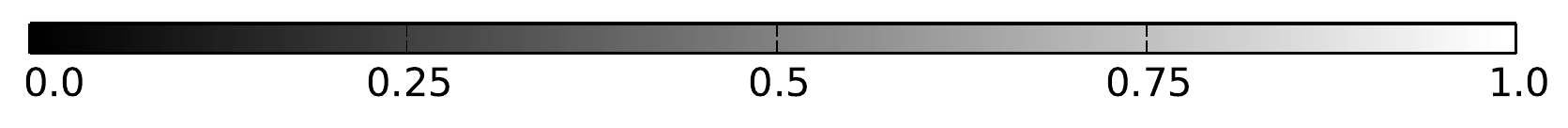}}
      \vspace{-.35cm}\\
      \multicolumn{2}{c}{
        {\small $\P({\rm success})$}} \\
      \hspace{-.4cm}
      \includegraphics[width=.5\columnwidth]{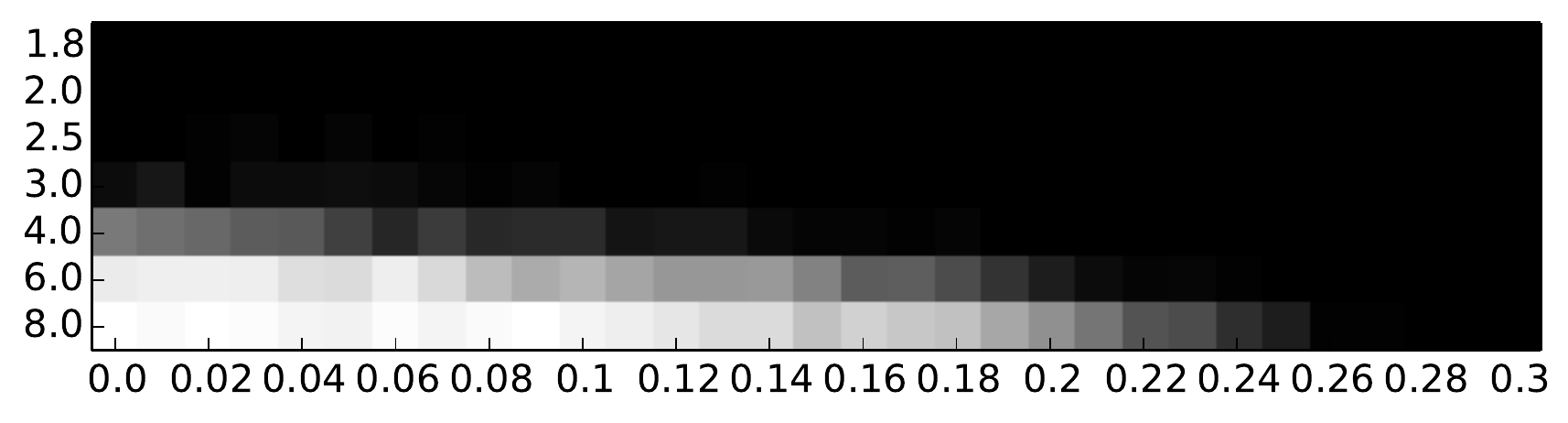} &
      \hspace{-.4cm}
      \includegraphics[width=.5\columnwidth]{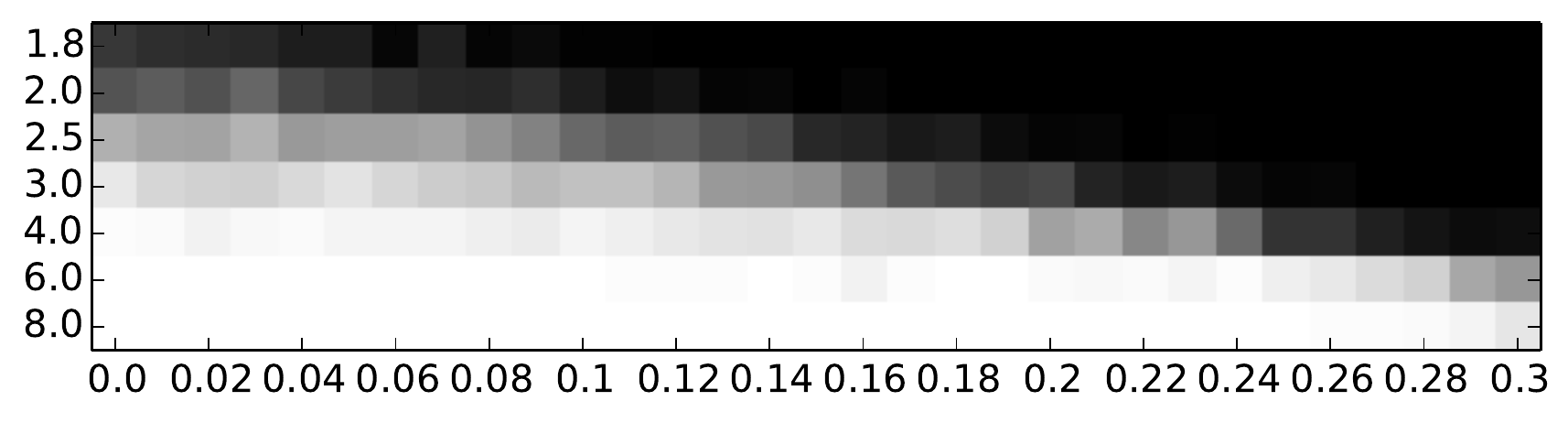} \\
      (a) & (b) \\
      \hspace{-.4cm}
      \includegraphics[width=.5\columnwidth]{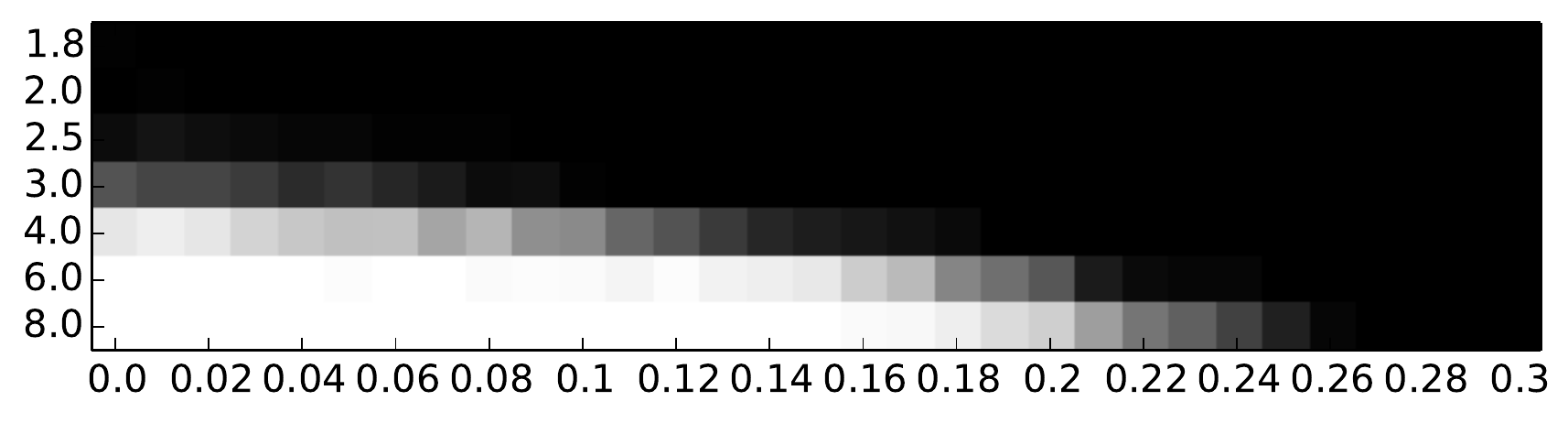} &
      \hspace{-.4cm}
      \includegraphics[width=.5\columnwidth]{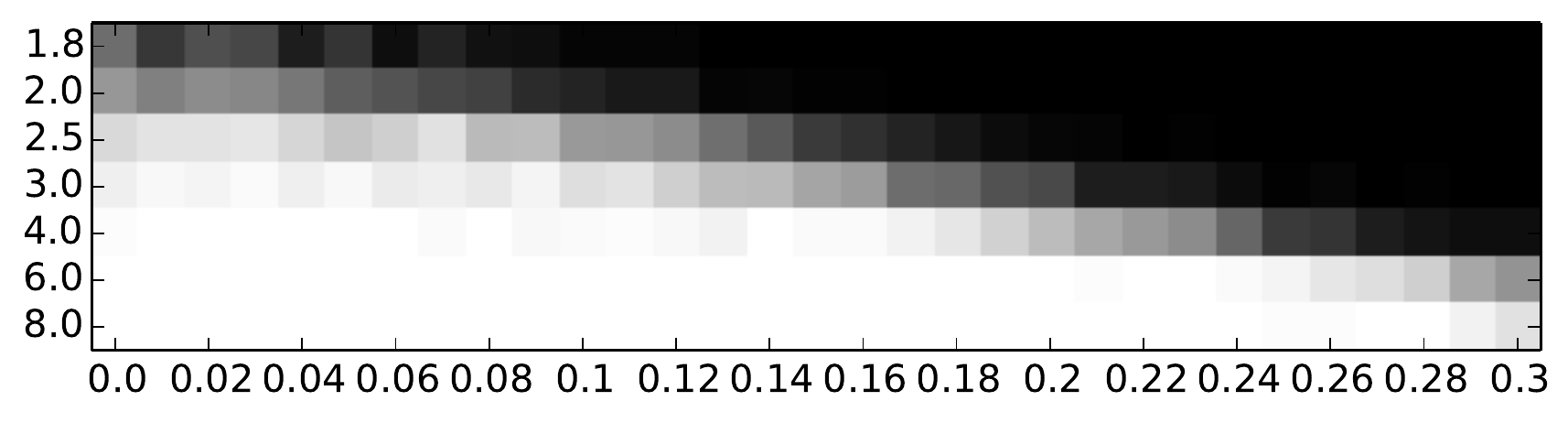} \\
      (c) & (d) \\
      \hspace{-.4cm}
      \includegraphics[width=.5\columnwidth]{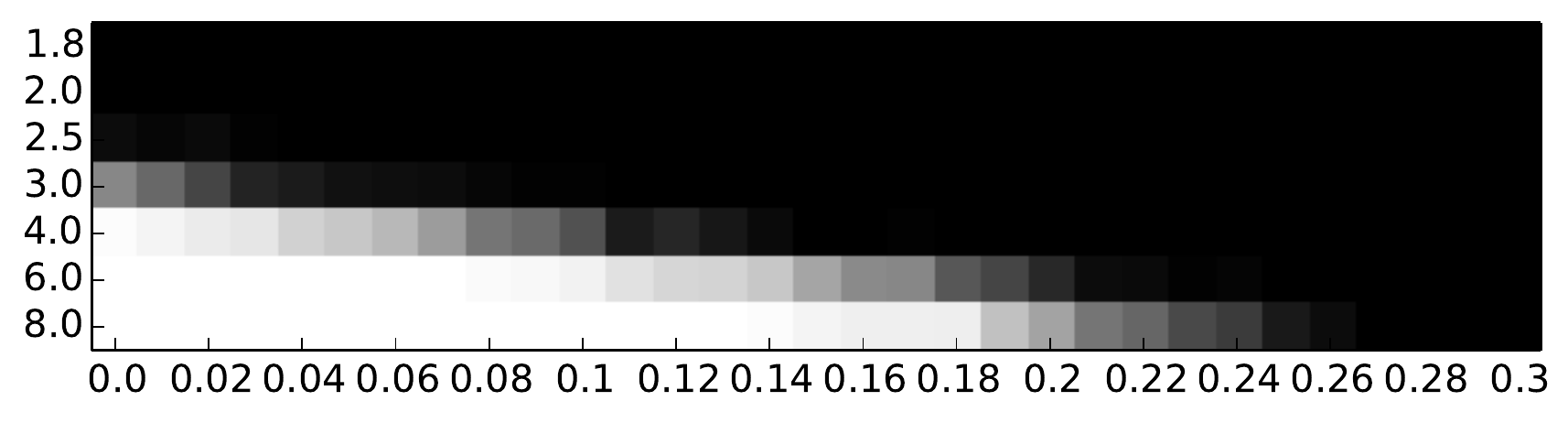} &
      \hspace{-.4cm}
      \includegraphics[width=.5\columnwidth]{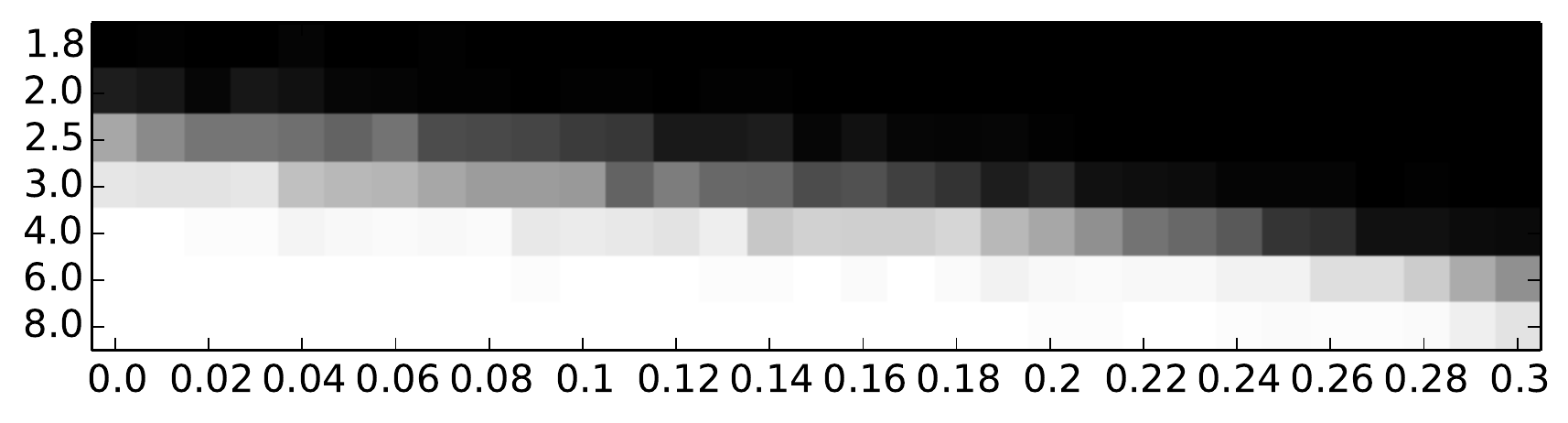}
      \vspace{-.3cm} \\ {\footnotesize $\pfail$} & {\footnotesize $\pfail$} \\
      (e) & (f)
    \end{tabular}
    \caption{\label{fig:outlier-recoveries} Probability of success
      for median-truncated Wirtinger flow (left) and
      the prox-linear method (right) for different initializations,
      where dimension $n = 100$.
      Horizontal axis indexes $\pfail$ while vertical axis
      indexes the measurement ratio $m / n$. Each pixel 
      represents the fraction of successful recoveries
      (defined as $\dist(\what{x}, X\subopt) / \ltwo{x\subopt} \le 10^{-5}$)
      over 100 experiments. (a) \& (b): ``Big'' initialization,
      with $\what{r}$ estimated by Alg.~\ref{alg:noisy-initialization}.
      (c) \& (d): ``Median'' initialization.
      (e) \& (f): ``Small'' initialization.
    }
  \end{center}
\end{figure}

\begin{figure}[t]
  \begin{center}
    \begin{tabular}{cc}
      \hspace{-.4cm}
      \includegraphics[width=.5\columnwidth]{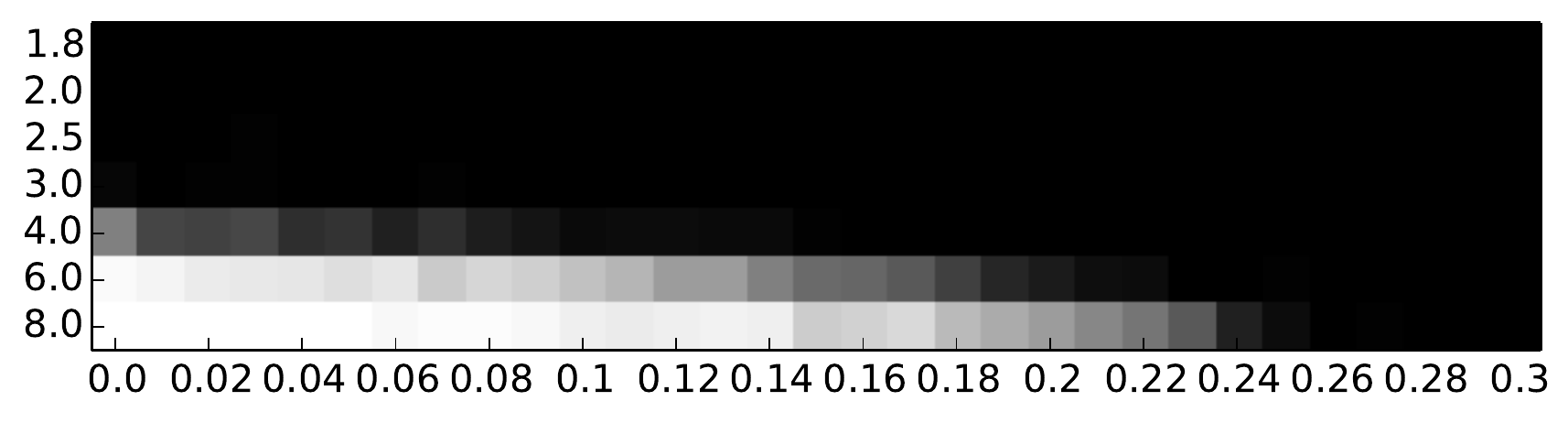} &
      \hspace{-.4cm}
      \includegraphics[width=.5\columnwidth]{%
        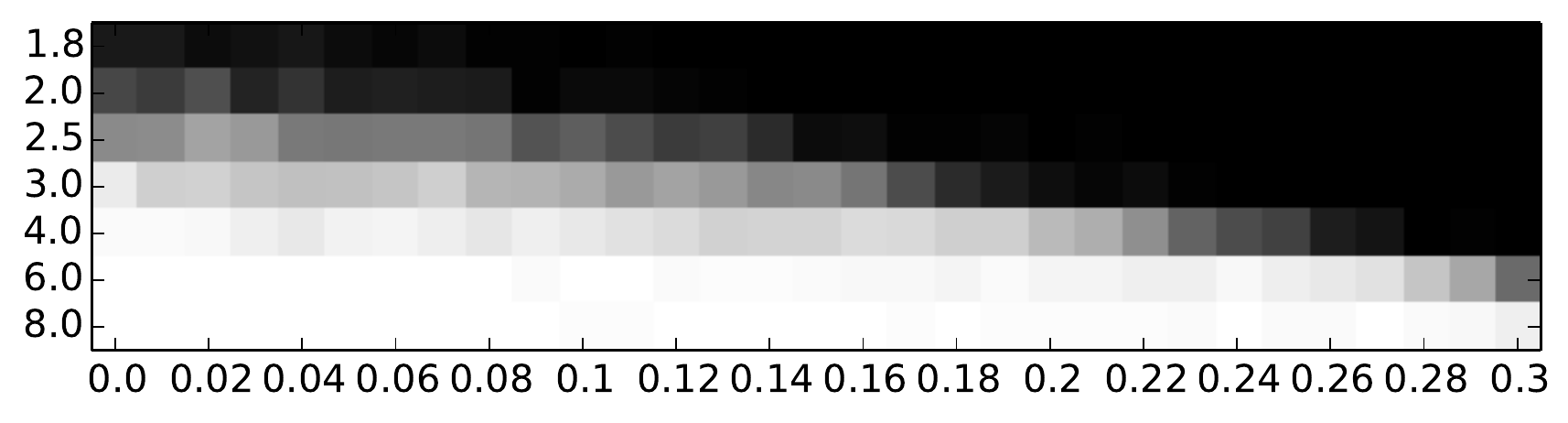} \\
      (a) & (b) \\
      \hspace{-.4cm}
      \includegraphics[width=.5\columnwidth]{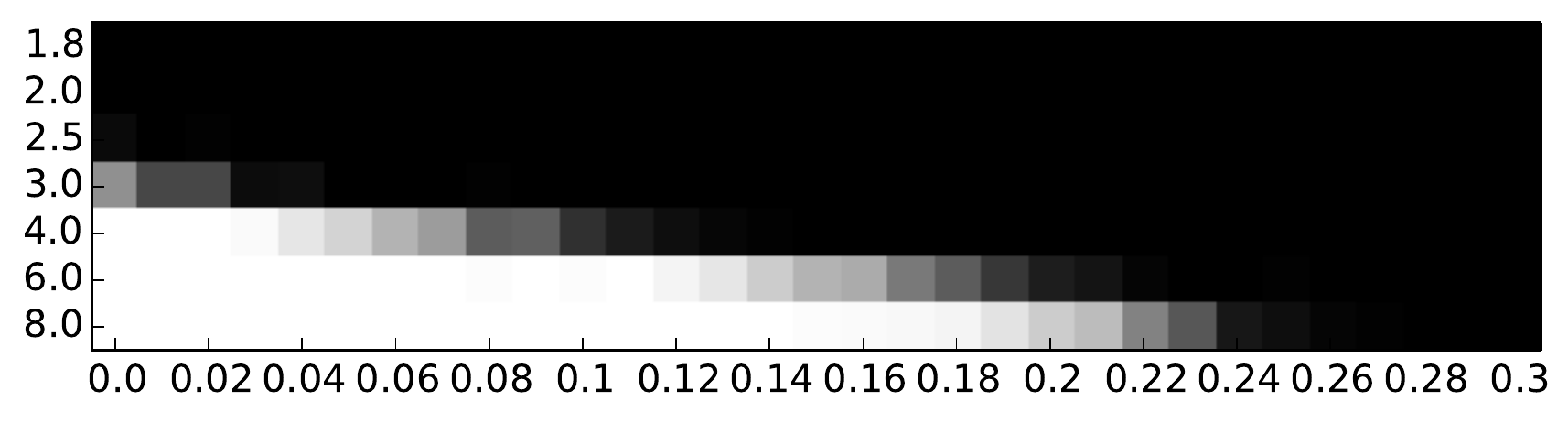} &
      \hspace{-.4cm}
      \includegraphics[width=.5\columnwidth]{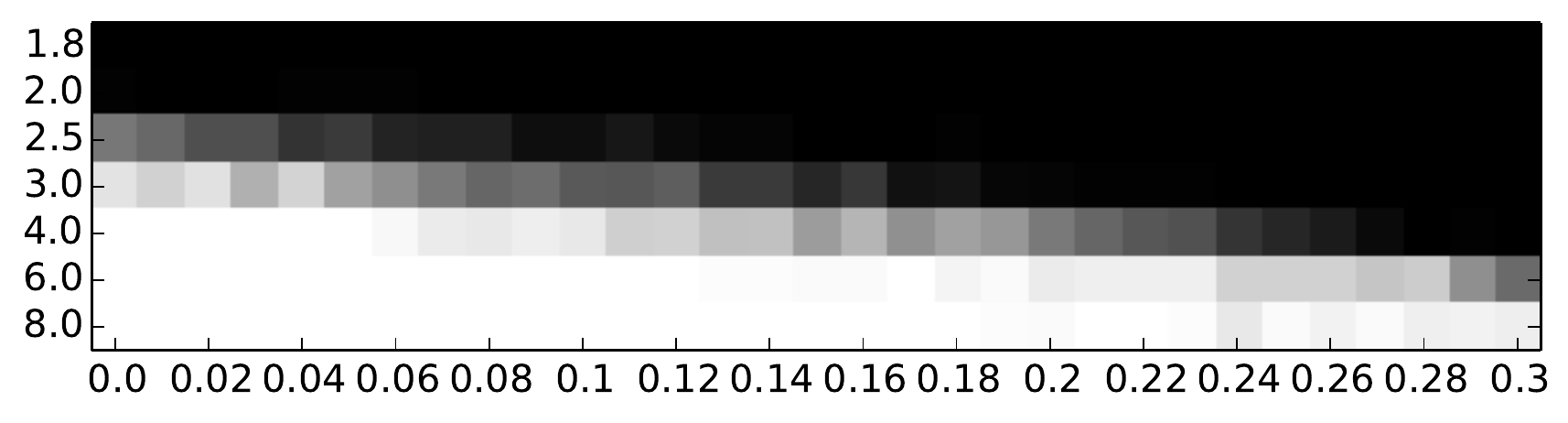} \\
      (c) & (d)
    \end{tabular}
    \caption{\label{fig:outlier-200-recoveries} Probability of success for
      median-truncated Wirtinger flow (left) and the prox-linear method
      (right), using the approximate proximal graph splitting method for
      sub-problem solves.  Dimension $n = 200$. (a) \& (b): ``Big''
      initialization, with $\what{r}$ estimated by
      Alg.~\ref{alg:noisy-initialization}. (c) \& (d): ``Small''
      initialization.}
  \end{center}
\end{figure}

In Figure~\ref{fig:outlier-single-performances}, we present a different view
into the behavior of the prox-linear and MTWF methods. In the left two
plots, we show the recovery probability (over 100 trials) for our composite
optimization method (Fig.~\ref{fig:outlier-single-performances}(a)) and MTWF
(Fig.~\ref{fig:outlier-single-performances}(c)). The success
probability for the composite method is higher. In
Fig.~\ref{fig:outlier-single-performances}(b), we plot the average number of
\emph{iterations} (along with standard error bars) the prox-linear method
performs when the dimension $n = 100$ and we use accurate sub-problem
solves. We give iteration counts only for those trials that result in
successful recovery; the iteration counts on unsuccessful trials are larger
(indeed, if the method is not converging rapidly, this serves as a proxy for
failure). In the high measurement regime, $m/n \ge 2.5$ or so, we see that
if $\pfail = 0$ no more than 7 iterations are required:
this is the quadratic convergence of the method. (Indeed, for $m/n = 8$, for
$\pfail \le .15$ each execution of the prox-linear method uses precisely 5
iterations, never more, and never fewer.)  In
Fig.~\ref{fig:outlier-single-performances}(d), we show the number of matrix
multiplications by the inverse matrix $(I_n + A^T DA)^{-1}$ the method uses
(recall the update~\eqref{eqn:pogs-update} in the proximal graph operator
splitting method). We see that for well-conditioned problems and those with
little noise, the methods require relatively few matrix multiplications,
while for more outliers and when $m/n$ shrinks, there is a non-trivial
increase.

\begin{figure}
  \begin{center}
    \begin{tabular}{cc}
      \hspace{-.5cm}
      \includegraphics[width=.52\columnwidth]{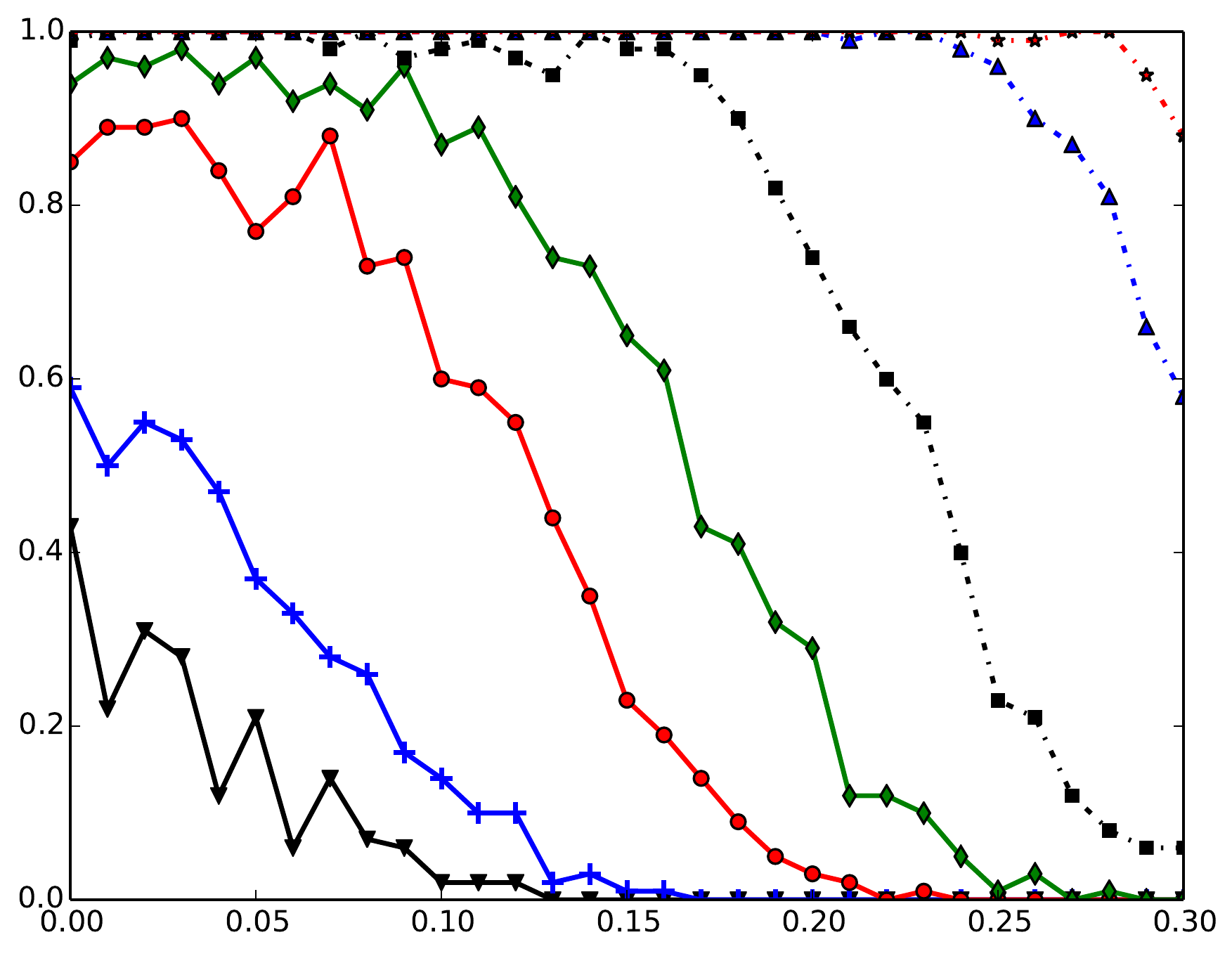} &
      \hspace{-.5cm}
      \includegraphics[width=.52\columnwidth]{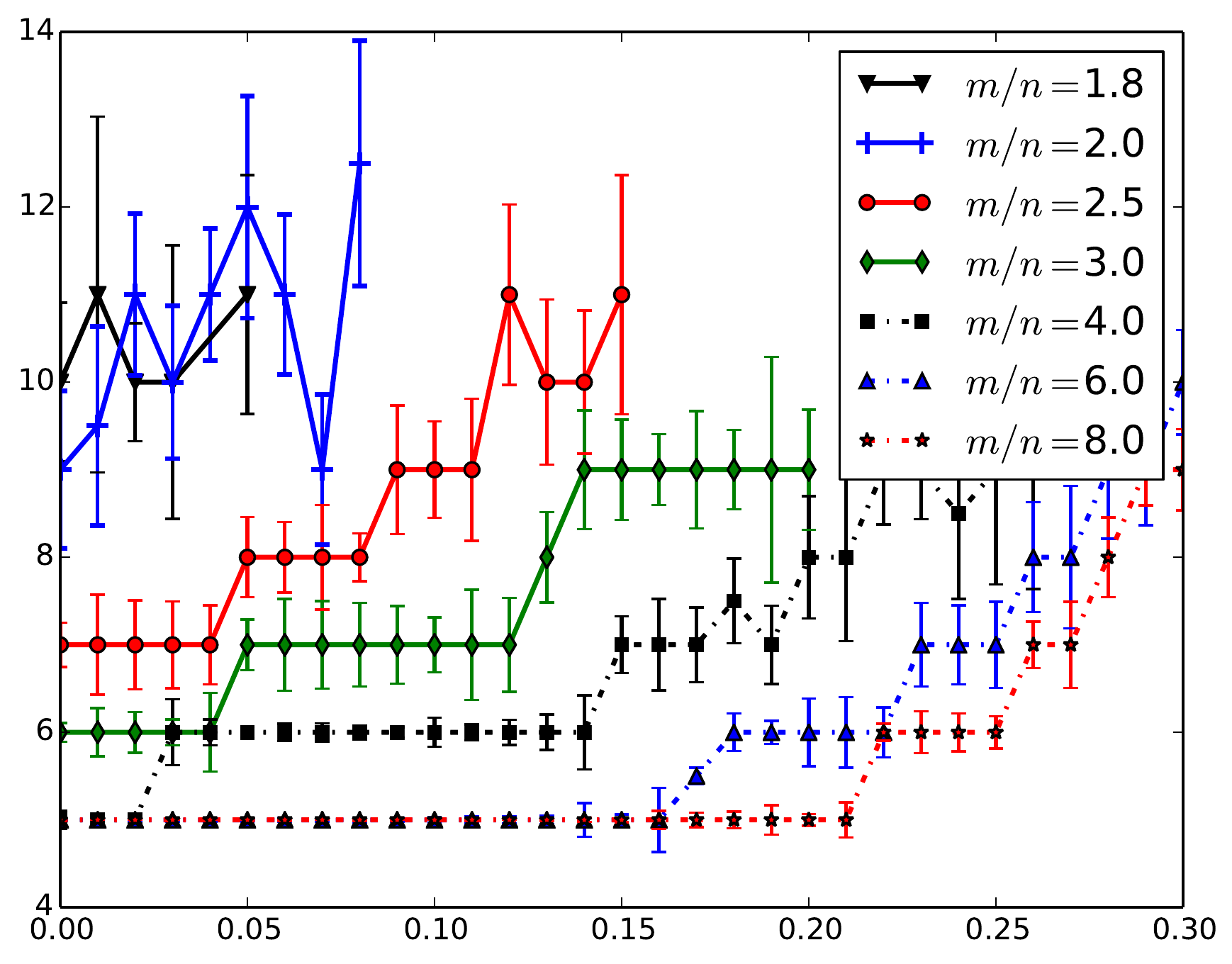} \\
      (a) & (b) \\
      \hspace{-.5cm}
      \includegraphics[width=.52\columnwidth]{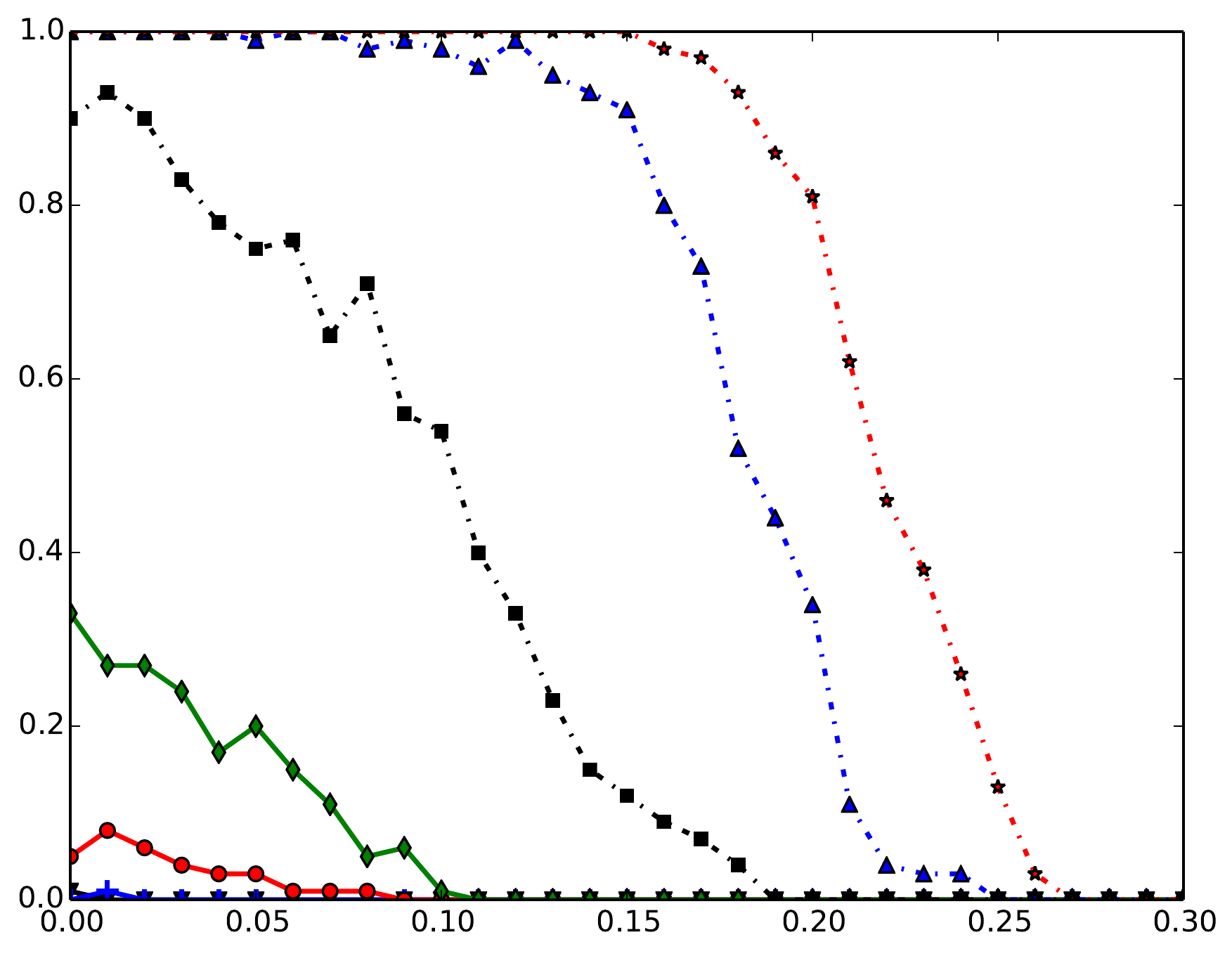} &
      \hspace{-.5cm}
      \includegraphics[width=.52\columnwidth]{%
        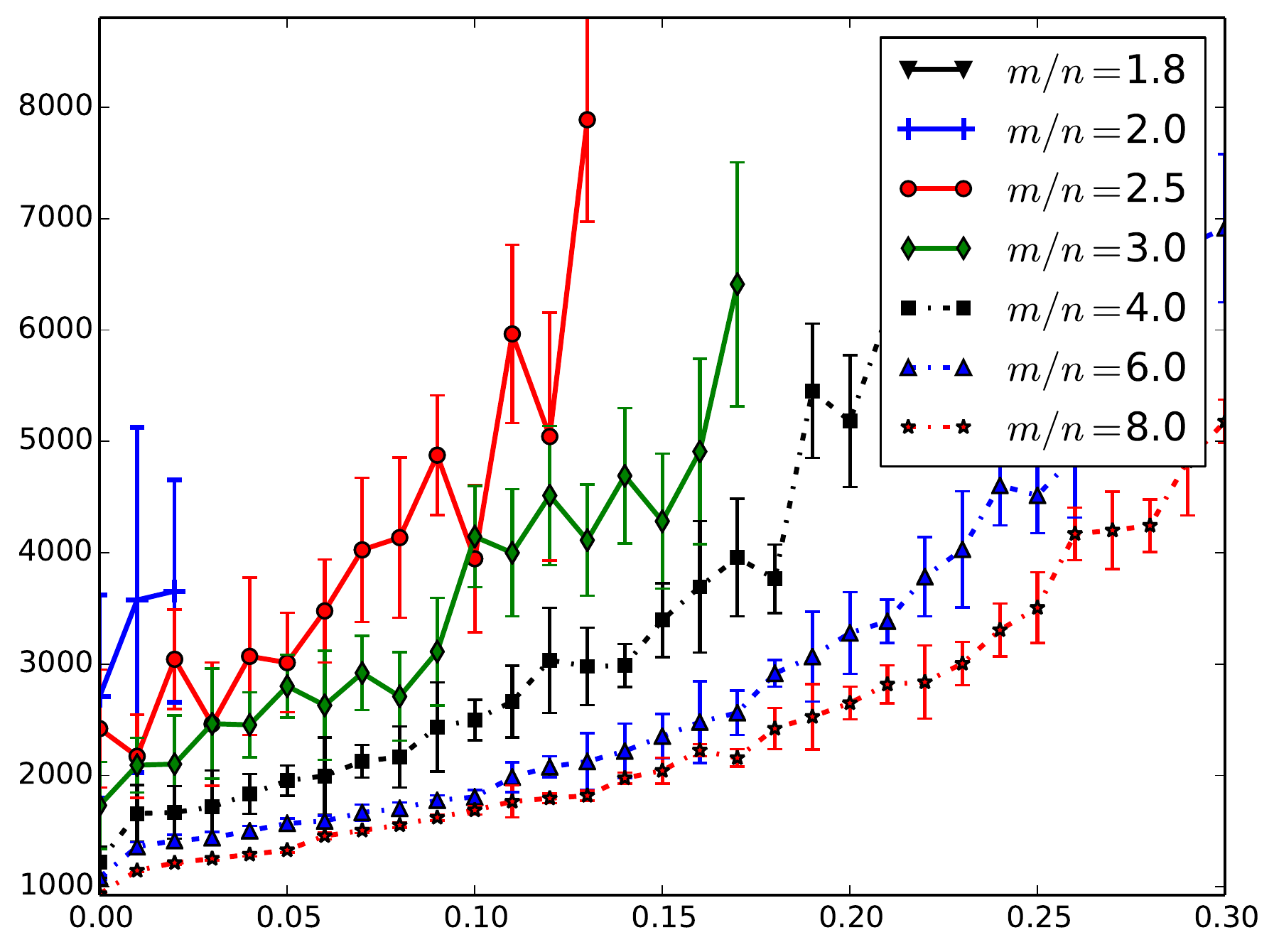} \\
      (c) & (d)
    \end{tabular}
  \caption{\label{fig:outlier-single-performances}
    Comparison of median-truncated Wirtinger flow and
    composite minimization with ``Median'' initialization.
    Horizontal axis of each plot indexes $\pfail \in [0, .3]$.
    (a) Proportion of successful solves for
    prox-linear method. (b) Number of iterations of
    prox-linear step~\eqref{eqn:prox-iteration} over \emph{successful}
    solves (with error bars). (c) Proportion of successful
    solves for median-truncated Wirtinger flow method.
    (d) Number of matrix multiplications of the form
    $x \mapsto (I + A^T D^2 A)^{-1} A x$ for the prox-linear method with
    proximal graph solves.}
  \end{center}
\end{figure}

\subsection{Recovery of real images}

\begin{figure}[t]
  \begin{center}
    \begin{tabular}{cccccc}
      \hspace{-.4cm}
      \includegraphics[width=.17\columnwidth]{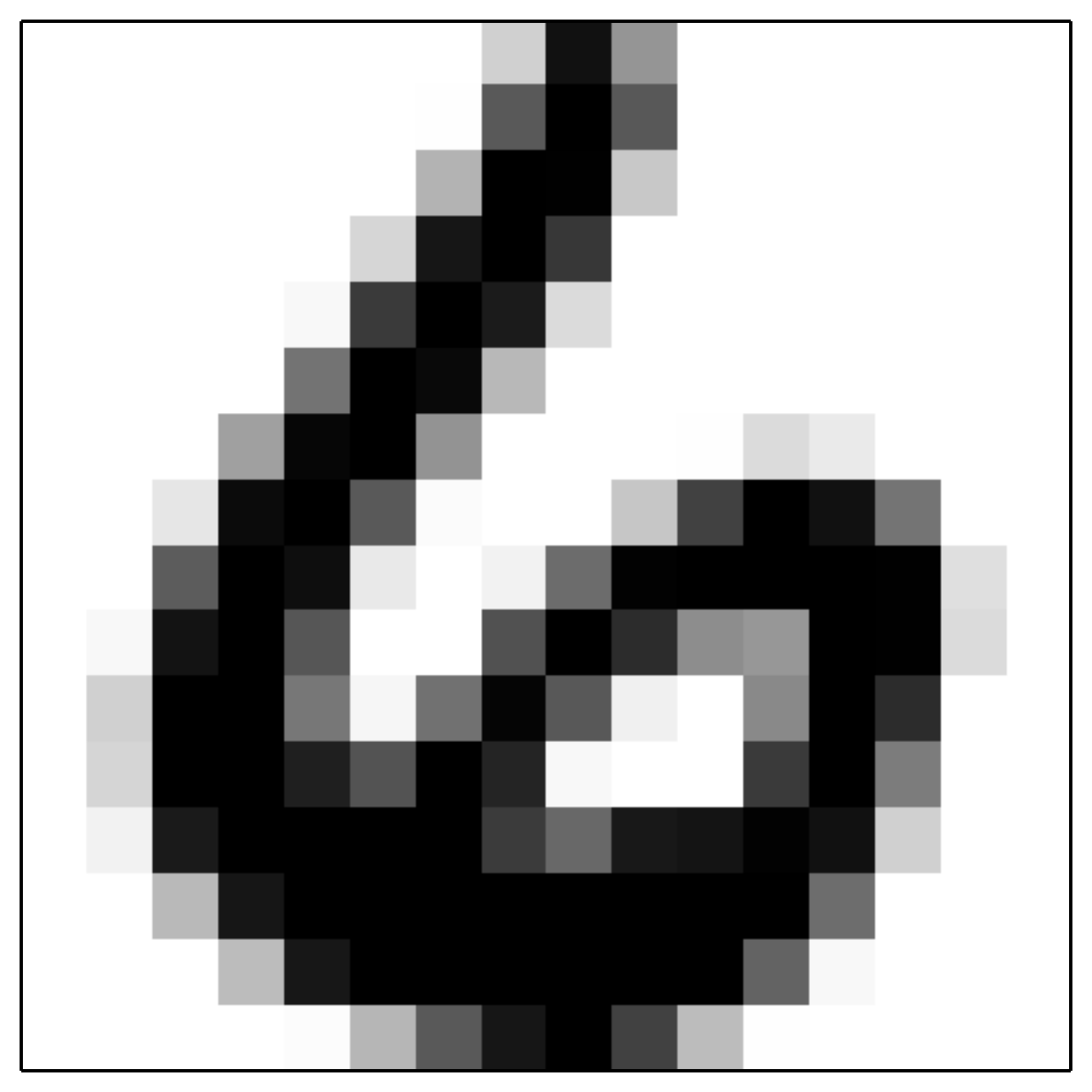} &
      \hspace{-.6cm}
      \includegraphics[width=.17\columnwidth]{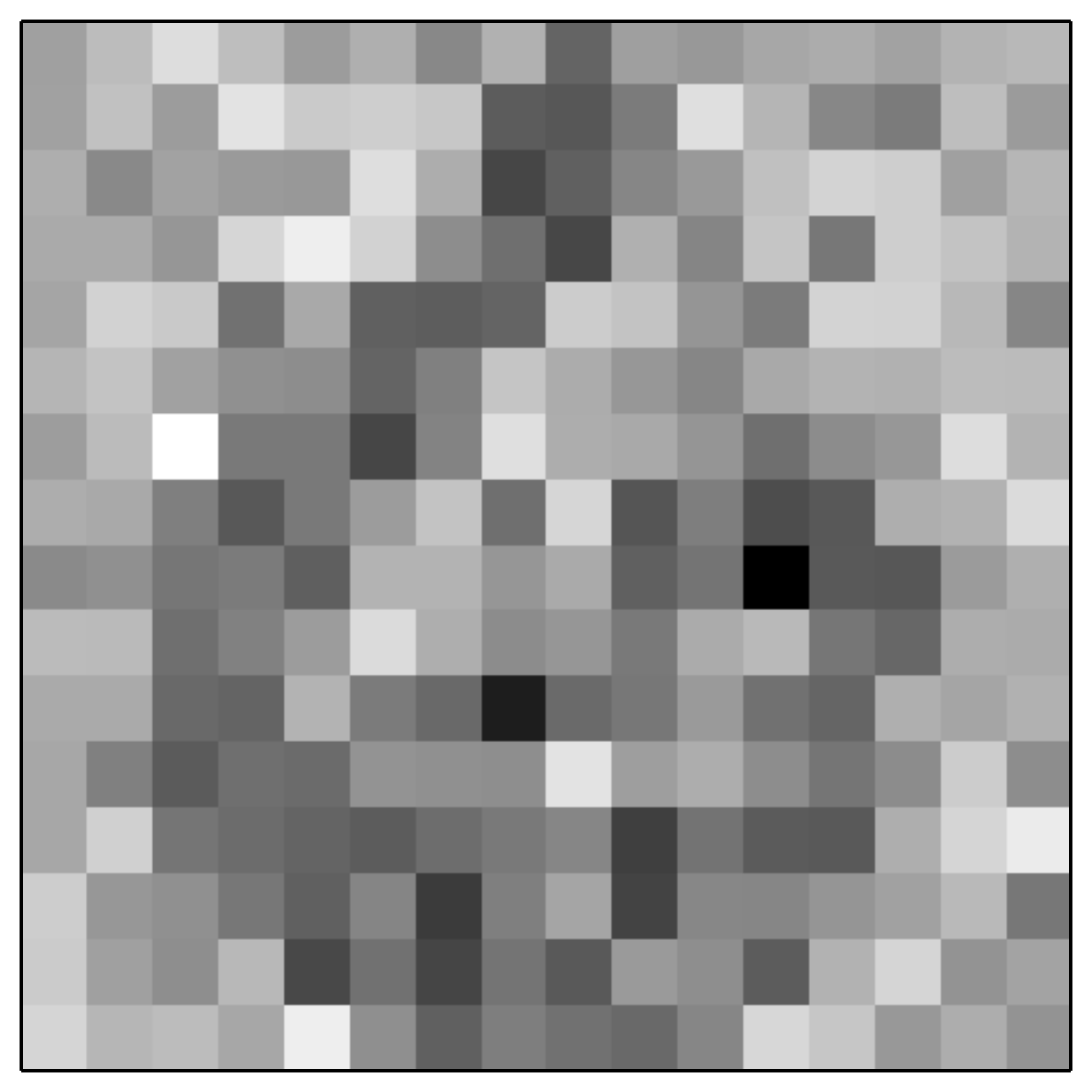} &
      \hspace{-.6cm}
      \includegraphics[width=.17\columnwidth]{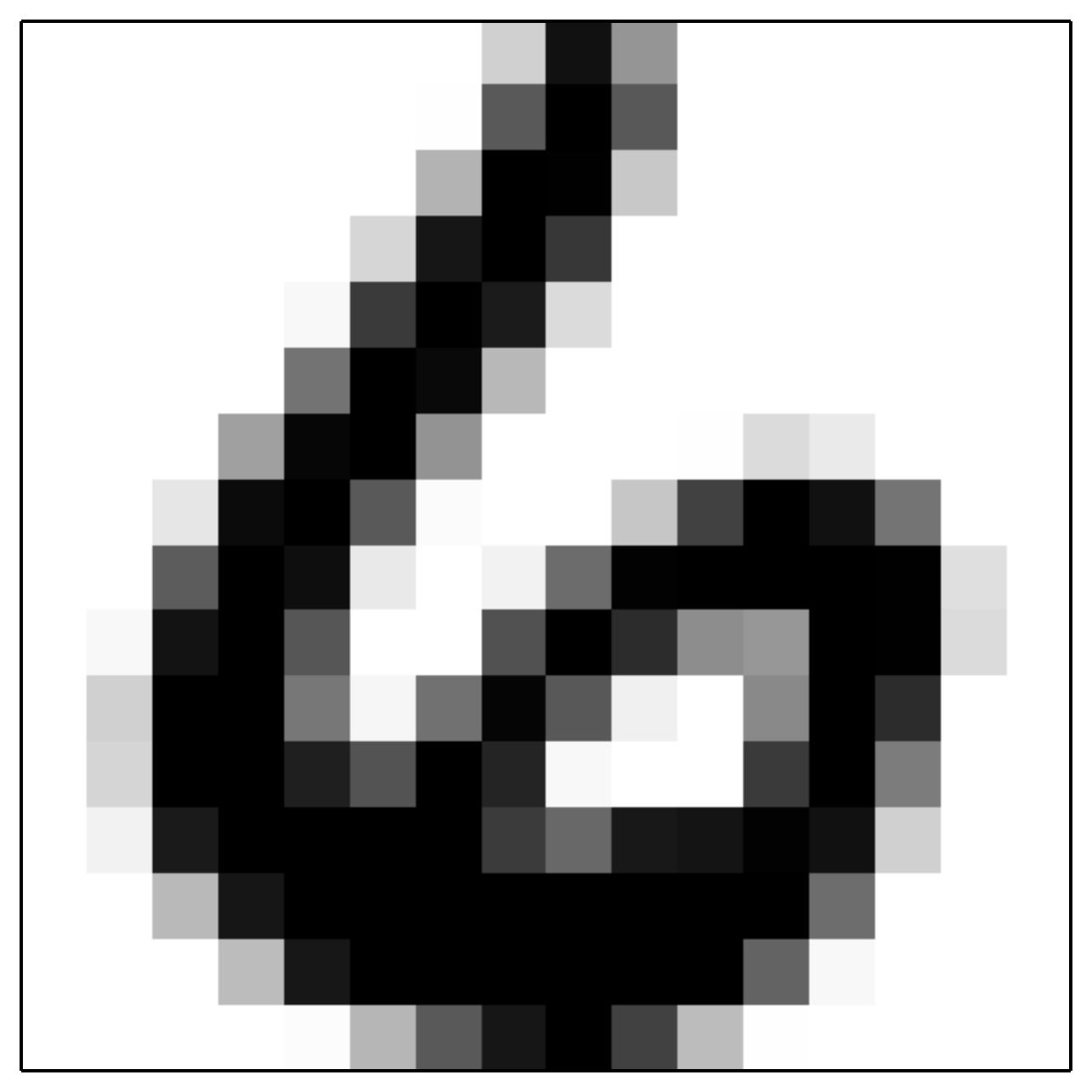} &
      \hspace{-.4cm}
      \includegraphics[width=.17\columnwidth]{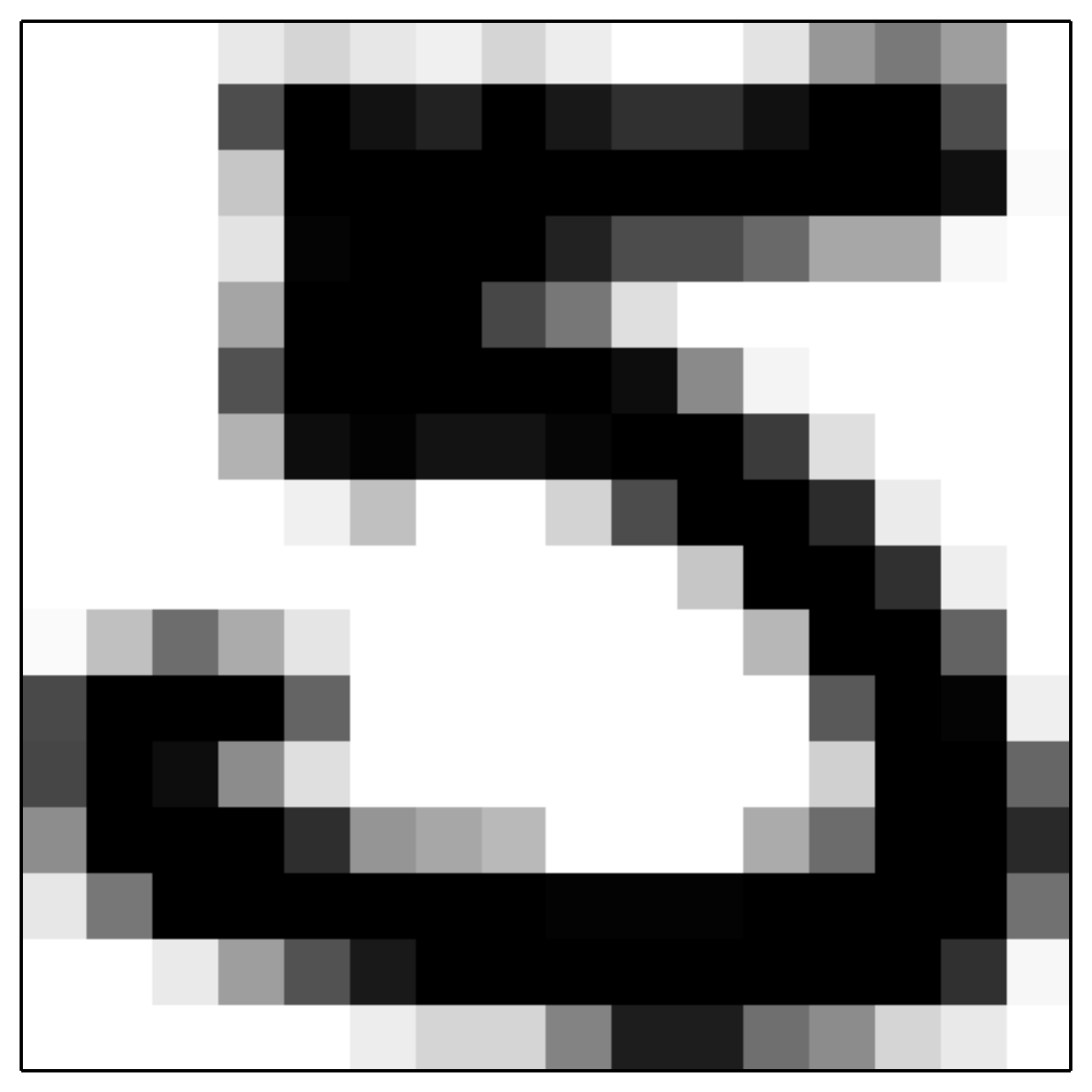} &
      \hspace{-.6cm}
      \includegraphics[width=.17\columnwidth]{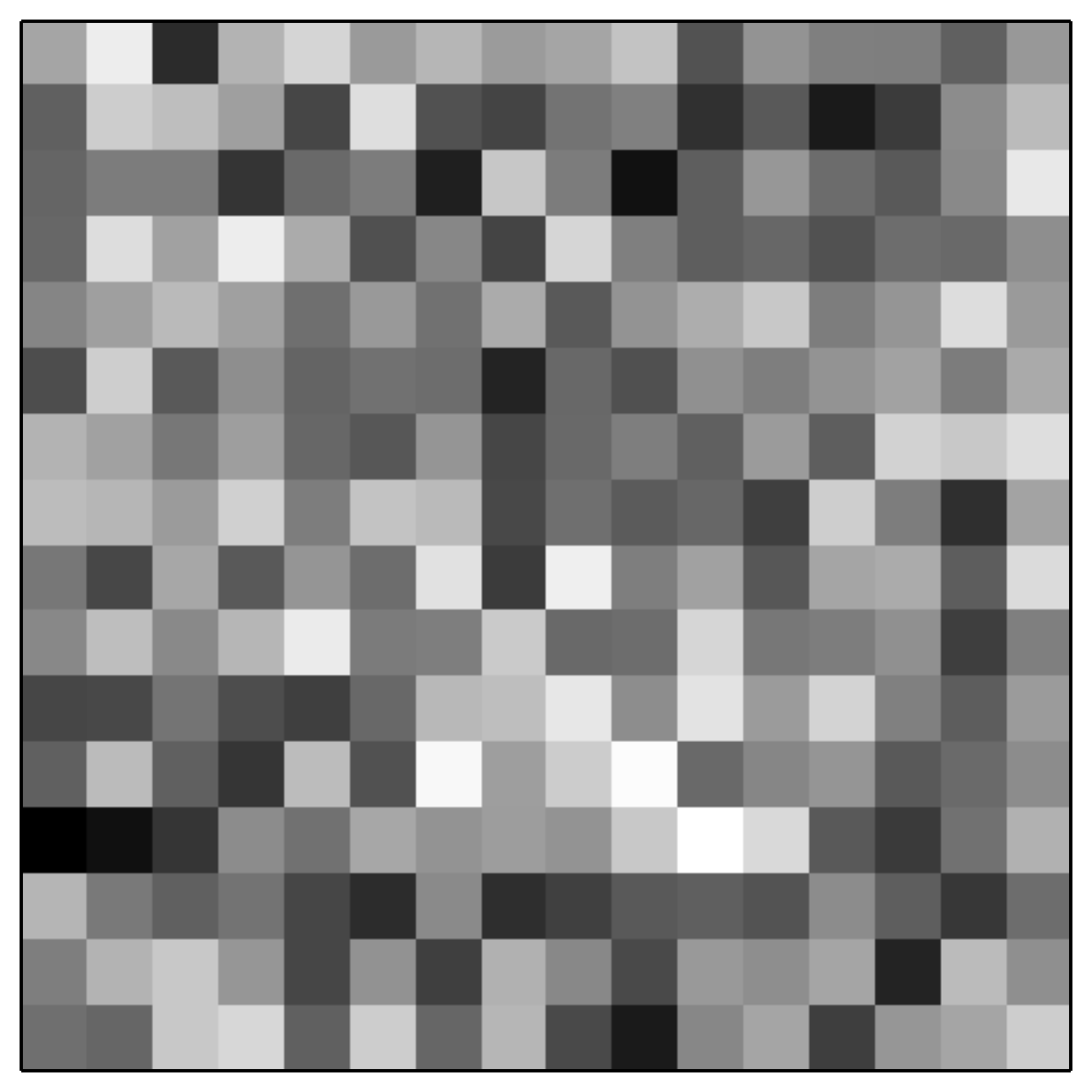} &
      \hspace{-.6cm}
      \includegraphics[width=.17\columnwidth]{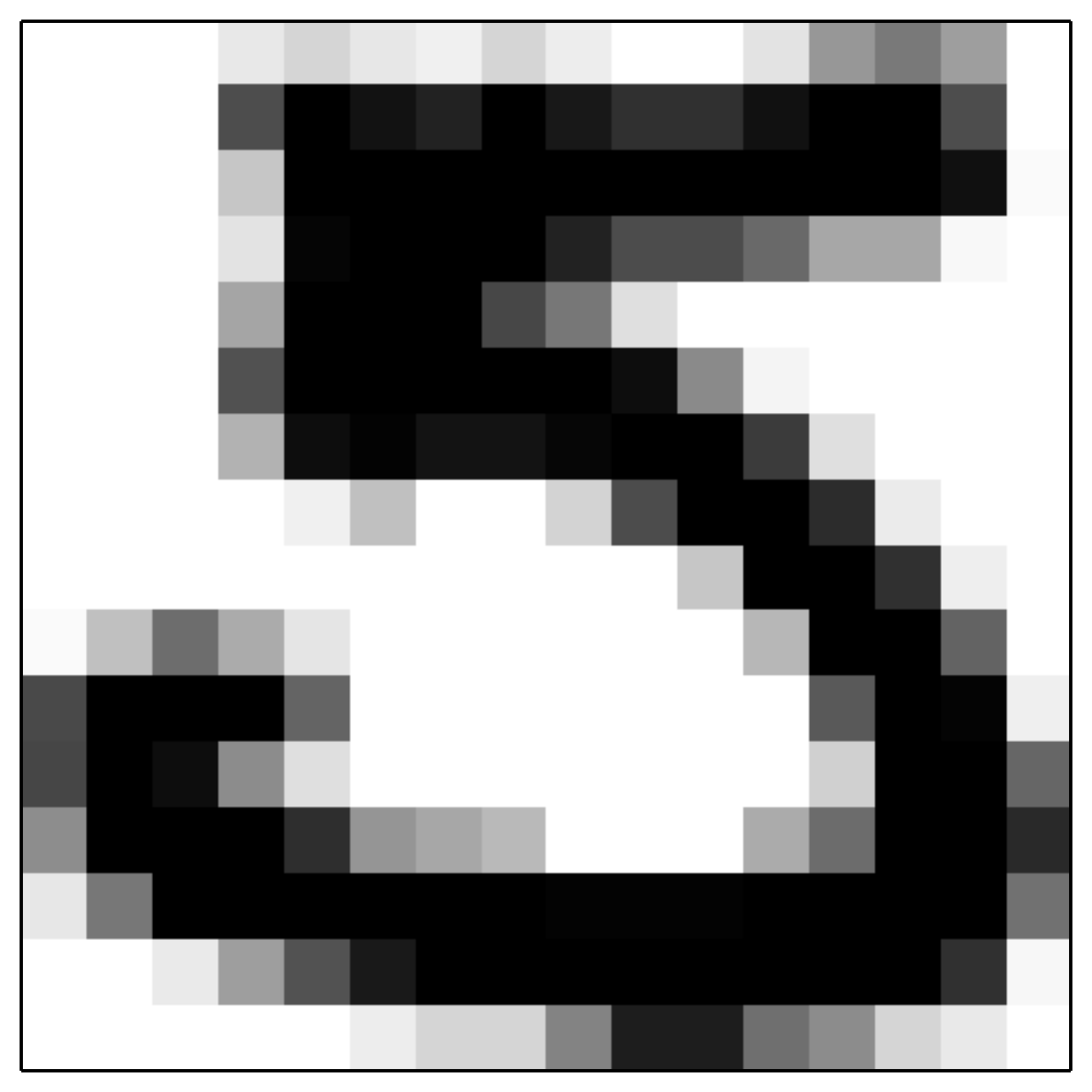}
    \end{tabular}
    \caption{\label{fig:digit-example} Example recoveries of digits.
      Left digit is true digit, middle is initialization,
      right is recovered image.}
  \end{center}
\end{figure}

\begin{figure}[t]
  \begin{center}
    \begin{overpic}[width=.7\columnwidth]
      {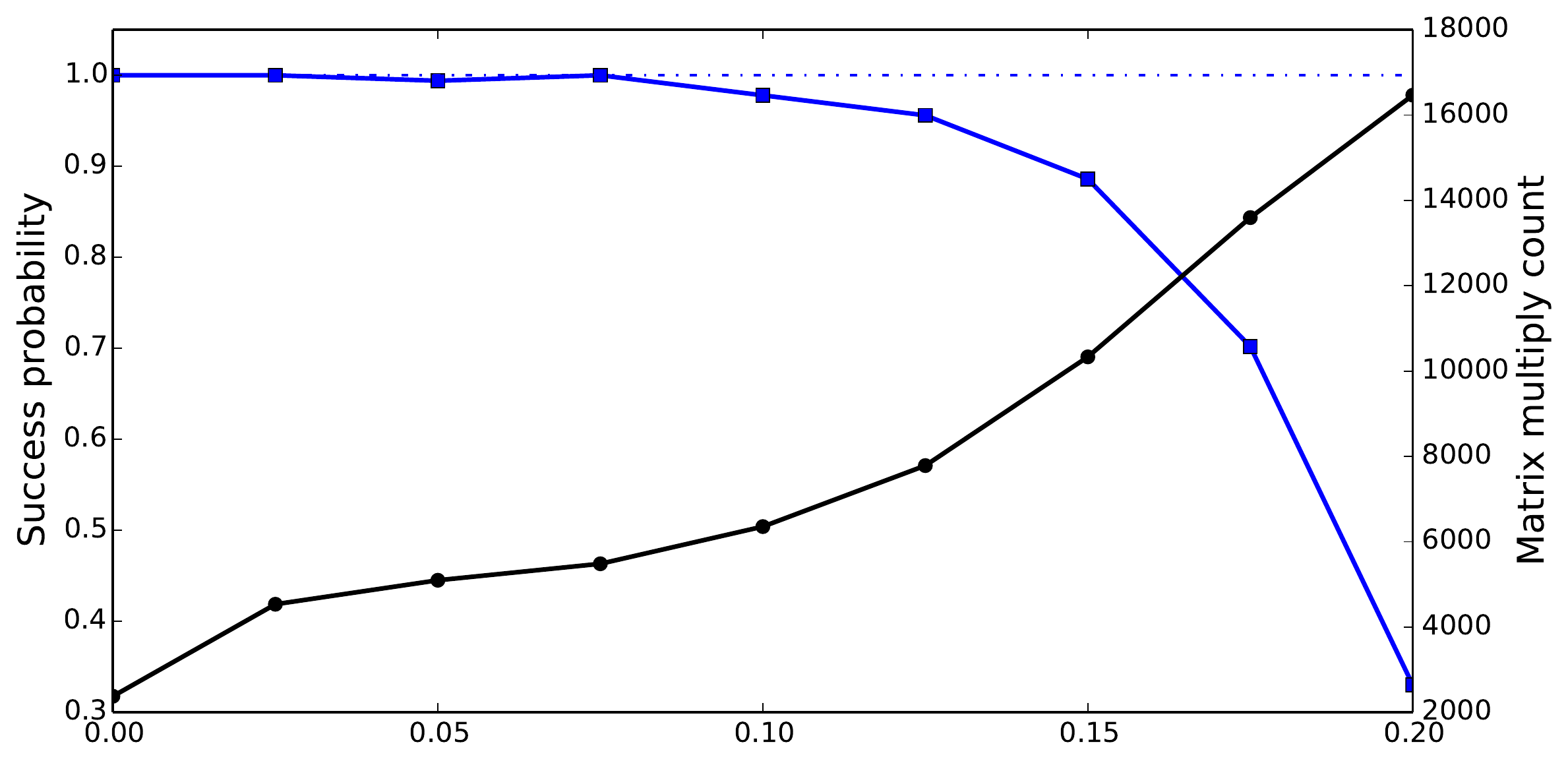}
      \put(55,1){$\pfail$}
    \end{overpic}
    \caption{\label{fig:digit-summary} Performance of composite optimization
      scheme using conjugate gradient method and proximal graph splitting to
      solve sub-problems (Sections~\ref{sec:pogs}--\ref{sec:cg-sub-prob}).}
  \end{center}
\end{figure}

\begin{figure}[ht!]
  \begin{center}
    \begin{tabular}{cc}
      \hspace{-.6cm}
      \includegraphics[width=.52\columnwidth]{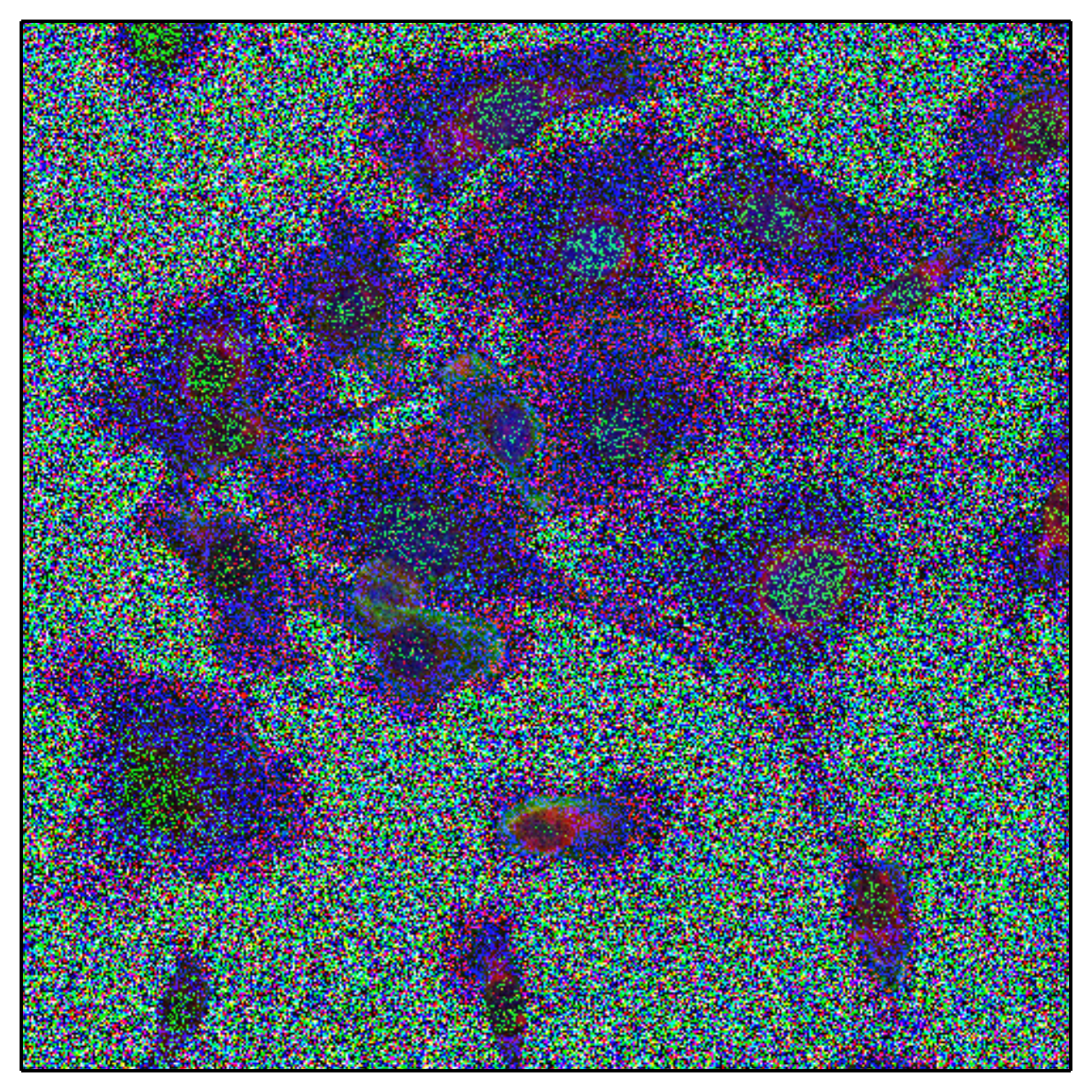} &
      \hspace{-.6cm}
      \includegraphics[width=.52\columnwidth]{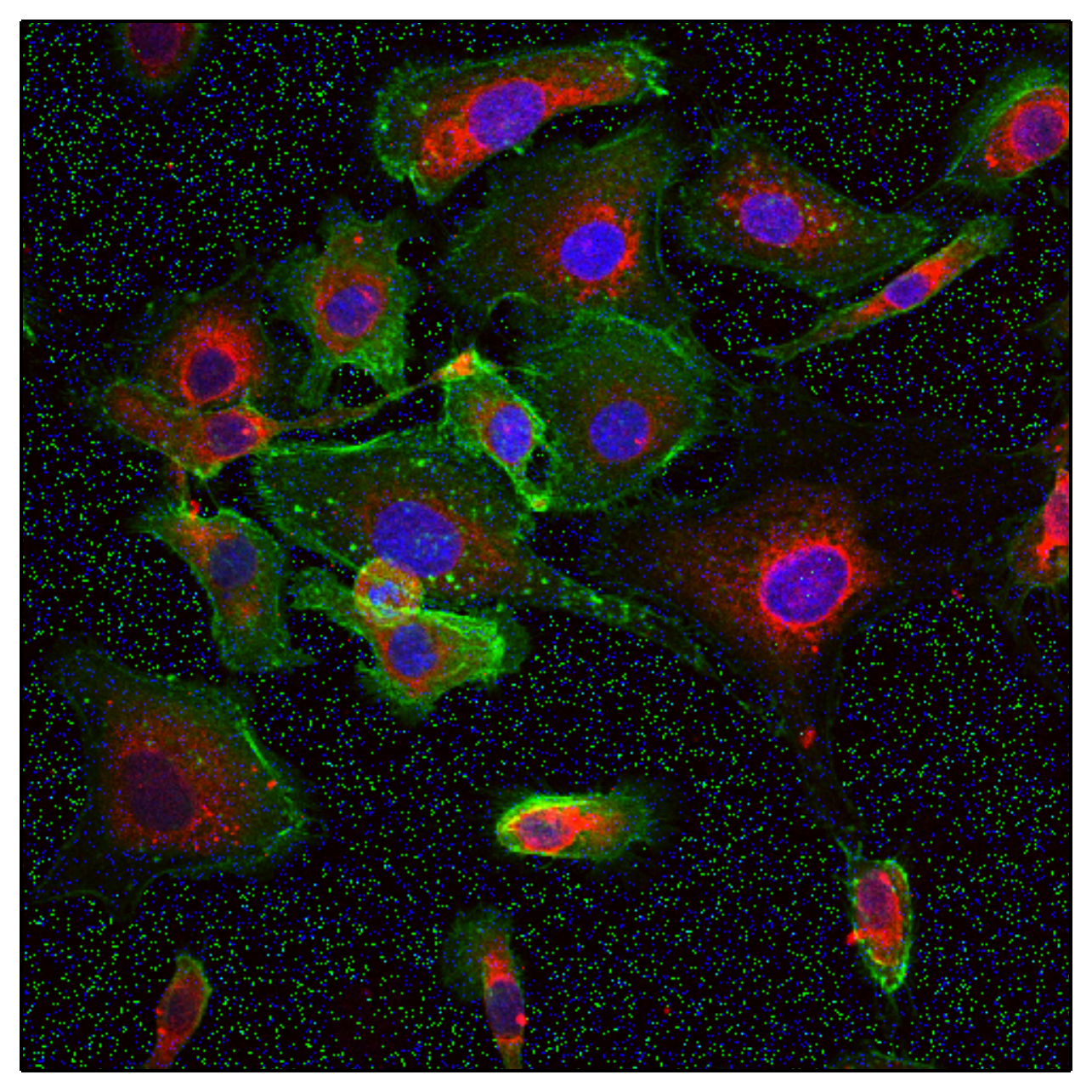} \\
      (a) & (b) \\
      \hspace{-.6cm}
      \includegraphics[width=.52\columnwidth]{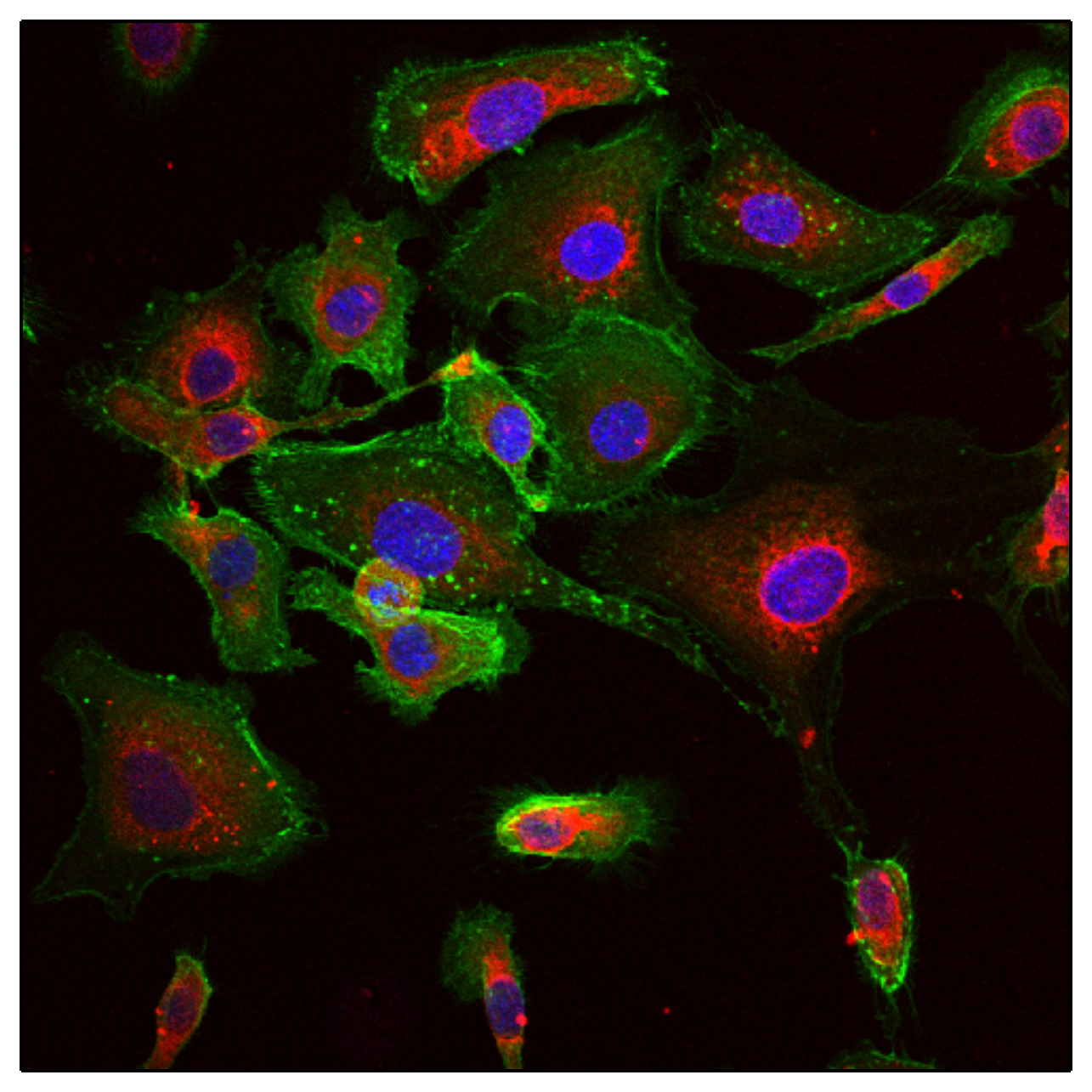} &
      \hspace{-.6cm}
      \includegraphics[width=.52\columnwidth]{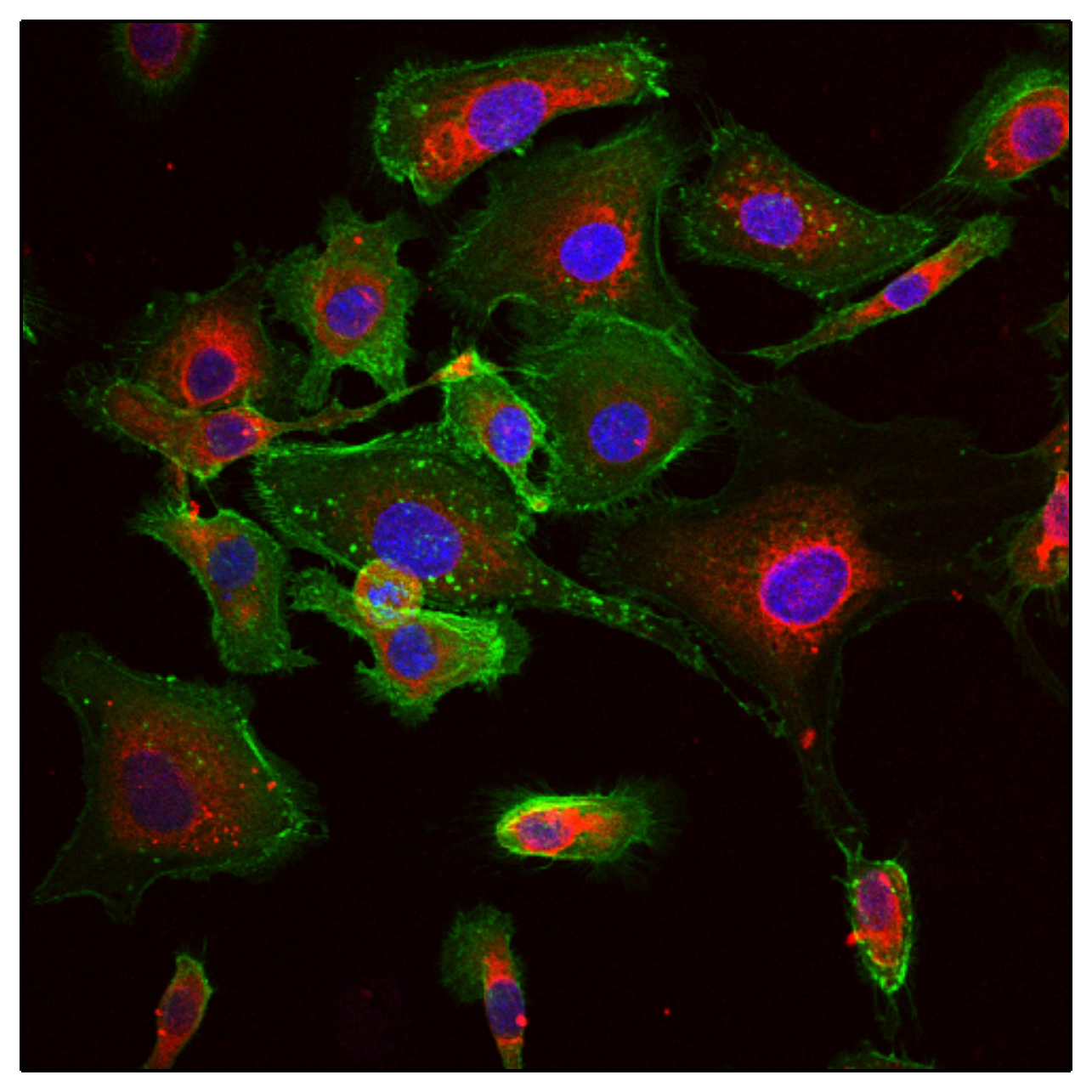} \\
      (c) & (d)
    \end{tabular}
    \caption{\label{fig:rna-recovery} Reconstruction of RNA image
      ($n = 2^{22}$) from $m = 3n$ measurements. (a) Initialization $x_0$ of
      prox-linear method.  (b) Result of 10 inaccurate ($\epsilon =
      10^{-3}$) solutions of prox-linear step~\eqref{eqn:prox-iteration}.  (c)
      Result of \emph{one} additional accurate ($\epsilon = 10^{-7}$)
      solution of prox-linear step~\eqref{eqn:prox-iteration}. (d)
      Original image.}
  \end{center}
\end{figure}

Our final collection of experiments investigates recovery of real-world
images using more specialized measurement matrices. In this case, we let
$H_n \in \{-1, 1\}^{n \times n} / \sqrt{n}$ denote a normalized Hadamard
matrix, where the multiplication $H_n v$ requires time $n \log n$, and $H_n$
satisfies $H_n = H_n^T$ and $H_n^2 = I_n$. For some $k \in \N$, we then take
$k$ i.i.d.\ diagonal sign matrices $S_1, \ldots, S_k \in \diag(\{-1,
1\}^n)$, uniformly at random, and define $A = [H_n S_1 ~ H_n S_2 ~ \cdots ~
  H_n S_k]^T \in \R^{kn \times n}$, as in
expression~\eqref{eqn:hadamard-sensing}.  We note that this matrix
explicitly \emph{does not} satisfy the stability conditions that we use for
our theoretical guarantees. Indeed, letting $e_1$ and $e_2$ be the first
standard basis vectors, we have $H_n S (e_1 + e_2) \perp H_n S (e_1 - e_2)$
no matter the sign matrix $S$; there are similarly pathological vectors for
FFT, discrete cosine, and other structured matrices. Nonetheless, we perform
experiments with this structured $A$ matrix as follows.\footnote{We do not
  compare to other methods designed for outliers because we could not set
  the constants the methods require in such a way as to yield good
  performance.}

Given an image $X$ represented as a matrix, we define $x\subopt = \mathop{\rm
  Vec}(X)$, the vectorized representation of $X$. (In the case of colored
images, where $X \in \R^{n_1 \times n_2 \times 3}$ because of the 3 RGB
channels, we vectorize the channels as well.) We then set $b = \phi(Ax)$,
where $\phi(\cdot)$ denotes elementwise squaring, and corrupt a fraction
$\pfail \in [0, .2]$ of the measurements $b_i$ by zeroing them.  We then
follow our standard experimental protocol, initializing $x_0$ by the
``small'' initialization scheme~\eqref{item:small}, with the slight twist
that now we use the proximal operator graph splitting method (POGS) with
conjugate gradient methods (Sec.~\ref{sec:cg-sub-prob}) to solve the graph
projection step~\eqref{eqn:tall-matrix-update}.

We first give results on a collection of 500 images of handwritten digits
(using $k = 3$),
available on the website for the book~\cite{HastieTiFr09}.  We provide
example results of the execution of our procedure in
Fig.~\ref{fig:digit-example}, which shows that while there is signal in the
initialization, there is substantial work that the prox-linear method must
perform to recover the images. For each of the 500 images, we vary $\pfail
\in \{0, .025, .05, \ldots, .175, .2\}$, then execute the prox-linear
method. We plot summary results in Fig.~\ref{fig:digit-summary}, which in
the blue curve with square markers gives the probability of successful
recovery of the digit (left axis) versus $\pfail$ (horizontal axis).
The right axis indexes the number of matrix multiplications the
method executes until completion (black line with circular marks).
We see that in spite of the demonstrated failure of the matrix $A$
to satisfy our stability assumptions, we have frequent recovery
of the images to accuracy $10^{-3}$ or better.

Finally, we perform experiments with eight real color images with sizes
up to $1024 \times 1024$ (yielding $n = 2^{22}$-dimensional problems), where
we use $k = 3$ random Hadamard sensing matrices. The prox-linear method
successfuly recovers each of the 8 images to relative accuracy at least
$10^{-4}$, and performs an average of $15100$ matrix-vector multiplications
(i.e.\ fast Hadamard transforms) over the eight experiments, with a standard
deviation of $2600$ multiplications. To give a sense of the importance of
different parts of our procedure, and the relative accuracies to which we
solve sub-problems~\eqref{eqn:prox-iteration}, we display one example in
Fig.~\ref{fig:rna-recovery}. In this example, we perform phase retrieval on
an image of RNA nanoparticles in cancer cells~\cite{PiZhGu16}. In
Fig.~\ref{fig:rna-recovery}(a), we show the result of initialization, which
displays non-trivial structure, though is clearly noisy. In
Fig.~\ref{fig:rna-recovery}(b), we show the result of 10 steps of the
prox-linear method, solving each step using POGS until the residual
errors~\eqref{eqn:residual-stopping} are below $\epsilon = 10^{-3}$. There
are clear artifacts in this image. We then perform one refinement step with
higher accuracy ($\epsilon = 10^{-7}$), which results in
Fig.~\ref{fig:rna-recovery}(c). This is indistinguishable, at least to our
eyes, from the original image (Fig.~\ref{fig:rna-recovery}(d)).



%% file: appendix-no-noise-complex.tex

\section{Proofs for noiseless phase retrieval}
\label{sec:no-noise-proofs}

In this appendix, we collect the proofs of the propositions and
results in Section~\ref{sec:no-noise}.

Because we use it multiple times in what follows, we state a standard
eigenvector perturbation result, a variant of the Davis-Kahane
$\sin$-$\Theta$ theorem.
\begin{lemma}[Stewart and Sun~\cite{StewartSu90}, Theorem 3.6]
  \label{lemma:simple-sin}
  For vectors $u, v \in \sphere^{n-1}$, define the angle $\theta(u, v) =
  \cos^{-1} |\<u, v\>|$.  Let $X \in \C^{n \times n}$ be symmetric, $\Delta$ a
  symmetric perturbation, and $Z = X + \Delta$, and define $\gap(X) =
  \lambda_1(X) - \lambda_2(X)$ to be the eigengap of $X$.  Let $v_1$ and
  $u_1$ be the first eigenvectors of $X$ and $Z$, respectively. Then
  \begin{equation*}
    \sqrt{1 - |\<u_1, v_1\>|^2} = |\sin \theta(u_1, v_1)|
    \le \frac{\opnorm{\Delta}}{\hinge{\gap(A) - \opnorm{\Delta}}}.
  \end{equation*}
\end{lemma}
\noindent
We will also have occasion to use the following one-sided variant
of Bernstein's inequality.

\begin{lemma}[One-sided Bernstein inequality]
  \label{lemma:one-sided-bernstein}
  Let $X_i$ be non-negative random variables with
  $\var(X_i) \le \sigma^2$. Then
  \begin{equation*}
    \P\left(\frac{1}{m} \sum_{i = 1}^m X_i \le
    \frac{1}{m} \sum_{i = 1}^m \E[X_i] - t \right)
    \le \exp\left(-\frac{m t^2}{2 \sigma^2}\right).
  \end{equation*}
\end{lemma}

\subsection{Proof of Proposition~\ref{proposition:small-ball}}
\label{sec:proof-small-ball}

Our proof uses Mendelson's ``small-ball'' techniques for
concentration~\cite{Mendelson14} along with control over a particular
VC-dimension condition. If we define
$h_i(u, v) \defeq \stabfunc^2 \indic{|\<a_i, u\>| \wedge |\<a_i, v\>|
  \ge \stabfunc}$,
then we certainly have
\begin{equation*}
  \frac{1}{m} \sum_{i=1}^m |\<a_i, u\>\<a_i, v\>| 
  \geq \frac{1}{m}
  \sum_{i=1}^m h_i(u, v)
  ~~ \mbox{for~all~} u, v \in \R^n.
\end{equation*}
We now control the class
$\mc{F} \defeq \{a \mapsto f(a) = |\<a, u\>| \wedge |\<a, v\>|   \mid
u, v \in \R^n\}$ by its VC-dimension:
\begin{lemma}
  \label{lemma:upper-bound-of-VC-dimension}
  There exists some numerical constant $C > 0$, such that, for any $c \geq
  0$, the VC-dimension of the collection of sets
  \begin{equation*}
    \mc{G} \defeq
    \left\{\left\{x \in \R^n ~\mbox{s.t.}~ |\<x, u\>| \wedge |\<x, v\>|
    \geq c\right\}
    \mid u \in \R^n, v\in \R^n\right\}
  \end{equation*}
  is upper bounded by $Cn$.
\end{lemma}
\begin{proof}
  The collection of half planes
  \begin{equation*}
    \mc{G}_{\rm plane} \defeq
    \left\{\left\{x \in \R^n ~\mbox{s.t.}~  \<u, x\> + b \geq 0\right\}
  	 \mid u \in \R^n, b \in \R\right\}
  \end{equation*}
  has VC-dimension at most $n+2$~\cite[Lemma 2.6.15]{VanDerVaartWe96}.
  We have the containment
  \begin{equation*}
    \mc{G} \subset
    \left\{(G_1 \cup G_2) \cap (G_3 \cup G_4)
    \mid G_1, G_2, G_3, G_4 \in \mc{G}\right\},
  \end{equation*}
  and so standard results on preservation of VC-dimension under set
  operations~\cite[Lemma 2.6.17]{VanDerVaartWe96}, imply that the
  VC-dimension of $\mc{G}$ is at most $C$ times the VC-dimension of $\mc{G}$
  for some numerical constant $C > 0$.
%
\end{proof}
\noindent
The associated thresholds $a \mapsto \indic{|\<a, u\>| \wedge |\<a, v\>| \ge
  \stabfunc}$ likewise have VC-dimension at most $Cn$ as $u, v$
vary. Applying standard VC-concentration
inequalities~\cite[Ch.~2.6]{BartlettMe02,VanDerVaartWe96} we immediately
obtain that
\begin{equation*}
  \P\left(\sup_{u, v \in \sphere^{n-1}}
  \frac{1}{\stabfunc} \left|\frac{1}{m} \sum_{i = 1}^m h_i(u, v)
  - \E[h_i(u, v)]\right|
  \ge c \sqrt{\frac{n}{m}} + t\right)
  \le 2 \exp\left(-\frac{m t^2}{2}\right)
\end{equation*}
for all $t \ge 0$, where $c < \infty$ is a numerical constant.
Substitute $t^2 \mapsto 2t$ to achieve the result.

\subsection{Proof of Proposition~\ref{proposition:complex-growth}}
\label{sec:proof-complex-growth}

We begin with a small lemma relating the Frobenius norm
of certain rank-two matrices to the distances between vectors.
\begin{lemma}
  \label{lemma:frob-rank-two}
  Let $X = xx\cg - yy\cg \in \C^{n \times n}$. Then
  $\lfro{X} \ge \inf_\theta \ltwos{x - e^{\imagunit\theta} y}
  \sup_\theta \ltwos{x - e^{\imagunit\theta} y}$.
\end{lemma}
\noindent
We defer the proof of the lemma to Section~\ref{sec:proof-frob-rank-two}.

To prove the proposition, we define the function $F : \C^{m \times n} \times
\C^{n \times n}$, for $Z = [z_1 ~ \cdots ~ z_m]\cg$, by
\begin{equation}
  \label{eqn:def-F-complex}
  F(Z, X)
  = \frac{1}{m} \sum_{i = 1}^m |\<z_i, X z_i\>|.
\end{equation}
Now, for a constant $c > 0$ to be chosen, define the truncated variables
\begin{equation}
  \label{eqn:def-z-complex}
  z_i \defeq \sqrt{n} a_i / \ltwo{a_i} \indic{\ltwo{a_i}
    \ge c \sqrt{n}}.
\end{equation}
Then we have that for any $A \in \C^{m \times n}$ that
\begin{equation*}
  F(A, xx\cg - yy\cg)
  \ge c^2 F(Z, xx\cg - yy\cg).
\end{equation*}
As the function $F$ is homogenous in its second argument,
based on Lemma~\ref{lemma:frob-rank-two}, the result of the theorem
will hold if we can show that (with suitably high probability)
\begin{equation}
  \label{eqn:F-to-show-complex}
  F(Z, X) \ge \frac{\absWconst^2}{4}
  ~~~ \mbox{for~all~rank~2~or~less}~ X \in \C^{n \times n}
  ~ \mbox{with}~ \lfro{X} = 1.
\end{equation}
It is, of course, no loss of generality to assume that $X$ is Hermitian
in inequality~\eqref{eqn:F-to-show-complex}.

With our desired inequality~\eqref{eqn:F-to-show-complex} in mind,
we present a covering-number-based argument to prove the theorem.
The first step in this direction is to lower bound the expectation
of $F(Z, X)$.
Let $X \in \mc{S}_r, X = \sum_{j = 1}^r \lambda_j u_j u_j\cg$. Then we have
that for $w_i = \sqrt{n} a_i / \ltwo{a_i}$ and our choice of $z_i$
that
\begin{align*}
  \E\left[|z_i\cg X z_i|\right]
  & = \E\left[|w_i\cg X w_i|\right]
  - \E\left[|w_i\cg X w_i| \indic{\ltwo{a_i} \le c \sqrt{n}}
    \right].
\end{align*}
Now, we have
\begin{equation*}
  \E[|w_i\cg X w_i| \indic{\ltwo{a_i} \le \epsilon \sqrt{n}}]
  \stackrel{(i)}{\le} \sqrt{\E[|w_i\cg X w_i|^2]}
  \sqrt{\smallfunc(c)}
  \stackrel{(ii)}{\le} \bigg(\sum_{i=1}^r \lambda_j^2\bigg)^\half
  \E\bigg[\sum_{i=1}^r |\<u_i, w_i\>|^4 \bigg]^\half
  \sqrt{\smallfunc(c)}
\end{equation*}
where the inequalities follow $(i)$ by Cauchy-Schwartz and
Assumption~\ref{assumption:rank-2-growth} (for the small-ball
probability bound $\smallfunc(c)$) and $(ii)$ by Cauchy-Schwartz,
respectively. As $w_i$ is a $\subg^2$-sub-Gaussian vector, standard results
on sub-Gaussians imply that
$\E[|\<u, w_i\>|^4] \le
(1 + e) \subg^4$ for any unit vector $u$, and thus
the final expectation has bound $\sum_{i=1}^r
\E[\<u_i, w_i\>^4] \le (1 + e) r \subg^4$,
and thus we obtain for any choice of $c$ in our
construction~\eqref{eqn:def-z-complex} of $z_i$ and any rank $r$ matrix
$X$ that
\begin{equation*}
  \E\left[|z_i\cg X z_i|\right]
  \ge \E\left[|w_i\cg X w_i|\right]
  - \sqrt{1 + e} \lfro{X} \subg^2 \sqrt{r \smallfunc(c)}.
\end{equation*}
In particular, as we consider only rank $r = 2$ matrices,
Assumption~\ref{assumption:rank-2-growth} allows us to choose $c$ in the
definition~\eqref{eqn:def-z-complex} small enough that
\begin{equation*}
  \smallfunc(c) \le \frac{1}{2(1 + e)} \cdot \frac{\absWconst^4}{\subg^4},
\end{equation*}
and then
\begin{equation}
  \label{eqn:truncated-a-good-growth-complex}
  \E\left[|z_i\cg X z_i|\right]
  \ge \E\left[|w_i\cg X w_i|\right]
  - \absWconst^2 \lfro{X} / 2
  \ge \frac{\absWconst^2}{2} \lfro{X}.
\end{equation}

The remainder of our argument now proceeds by a covering argument,
which we begin with a lemma on the continuity properties
of $F$.
\begin{lemma}
  \label{lemma:continuity-F-func-complex}
  Let Hermitian matrices $X, Y \in \C^{n \times n}$ have
  rank at most $r$ and
  eigen-decompositions $X = U \Lambda U\cg$ and $Y = V \Sigma V\cg$,
  and let $Z = [z_1 ~ \cdots ~ z_m]\cg \in \C^{m \times n}$.
  Then
  \begin{equation*}
    |F(Z, X) - F(Z, Y)|
    \le \frac{2 \sqrt{r}}{m} \opnorm{Z}^2 \left[
      \norm{U - V}_{1,2}
      + \lfro{\Lambda - \Sigma} \right].
  \end{equation*}
\end{lemma}
\noindent
We provide the proof of Lemma~\ref{lemma:continuity-F-func-complex}
in Section~\ref{sec:proof-continuity-F-func-complex}.

Based on Lemma~\ref{lemma:continuity-F-func-complex}, we develop a
particular covering set of the $n \times n$ Hermitian matrices of rank $r$,
following a construction due to \citet[Thm.~2.3]{CandesPl11}. Let
$\mc{S}_r \subset \{X \in \C^{n \times n}$ be the set of Hermitian rank $r$
matrices with $\lfro{X} = 1$, and let
$\mathsf{O}_{n,r} = \{U \in \C^{n \times r} \mid U\cg U = I_r\}$ denote the
set of $n\times r$ unitary matrices; for each $\epsilon > 0$ there exists an
$\epsilon$-cover $\wb{\mathsf{O}}_{n,r}$ of $\mathsf{O}_{n,r}$ in
$\norm{\cdot}_{1,2}$-norm of cardinality at most $(3 /
\epsilon)^{2nr}$. Similarly, there exists an $\epsilon$-cover $D \subset
\R^r$ of all vectors $v \in \R^r$ in $\ell_2$-norm of cardinality at most
$(3 / \epsilon)^{r}$. Now, let $\wb{\mc{S}}_r = \{U \diag(d) U\cg \mid U \in
\wb{\mathsf{O}}_{n,r}, d \in D\}$ be a subset of the rank-$r$ Hermitian
matrices. Then $\card(\wb{\mc{S}}_r) \le (3 / \epsilon)^{(2n + 1)r}$, and we
have the following immediate consequence of
Lemma~\ref{lemma:continuity-F-func-complex}.
\begin{lemma}
  \label{lemma:cover-F-complex}
  For any $\epsilon > 0$, there exists a set
  $\wb{\mc{S}}_r \subset \mc{S}_r$ of cardinality at most
  $\card(\wb{\mc{S}}_r) \le (3/\epsilon)^{(2 n + 1)r}$ such that
  \begin{equation*}
    \inf_{X \in \mc{S}_r}
    F(Z, X) \ge \min_{X \in \wb{\mc{S}}_r}
    F(Z, X)
    - \frac{4\sqrt{r}}{m} \opnorm{Z}^2 \epsilon.    
  \end{equation*}
\end{lemma}

For our final step, we
control the variance of $F(Z, Y)$, so that we can
apply the one-sided Bernstein inequality
(Lemma~\ref{lemma:one-sided-bernstein}) and then a covering argument.
For $X \in \mc{S}_r$, then, we have
\begin{equation*}
  \var(z_i\cg X z_i)
  \le \E[|z_i\cg X z_i|^2]
  = \E\left[\bigg|\sum_{j = 1}^r \lambda_j |\<z_i, u_j\>|^2\bigg|^2
    \right]
  \le \lfro{X}^2 \sum_{j = 1}^r \E[|\<z_i, u_j\>|^4]
  \le r (1 + e) \lfro{X}^2 \subg^2
\end{equation*}
by Cauchy-Schwartz and that $\E[|\<z_i, u_j\>|^4] \le (1 + e) \subg^4$
by sub-Gaussian moment bounds. Thus, Lemma~\ref{lemma:one-sided-bernstein}
implies that for any fixed $X \in \mc{S}_r$ we have
\begin{equation*}
  \P\left(F(Z, X) \le \frac{\absWconst^2}{2}
  - \sqrt{2(1 + e) r} \cdot \subg^2 t\right) \le e^{-mt^2}
\end{equation*}
for all $t \ge 0$.
Now, let $\mc{E}_t$ be the event that $\frac{1}{m} \opnorm{Z}^2 \le
\varepsilon_{n,m}(t) \defeq \subg^2 + 11 \subg^2 \max\{ \sqrt{4n / m + t},
4n/m + t\}$.  Then by Lemma~\ref{lemma:cover-F-complex} and
Lemma~\ref{lemma:matrix-concentration} with $r = 2$, we have for any $K$ and
$\epsilon$ that
\begin{align*}
  \P\left(\inf_{X \in \mc{S}_r} F(Z, X) \le K \right)
  & \le \P\left(\min_{X \in \wb{\mc{S}}_r} F(Z, X)
  - \frac{4}{m} \opnorm{Z}^2 \epsilon \le K\right) \\
  & \le \P\left(\min_{X \in \wb{\mc{S}}_r} F(Z, X)
  - \frac{4}{m} \opnorm{Z}^2 \epsilon \le K,
  \mc{E}_t\right) + \P(\mc{E}_t) \\
  & \le \P\left(\min_{X \in \wb{\mc{S}}_r} F(Z, X)
  - 4 \varepsilon_{n,m}(t) \epsilon \le K \right) + e^{-mt} \\
  & \le \left(\frac{3}{\epsilon}\right)^{(2n + 1) r}
  \sup_{X \in \mc{S}_r} \P\left(F(Z, X)
  - 4 \varepsilon_{n,m}(t) \epsilon \le K\right)
  + e^{-mt}.
\end{align*}
Letting $K = \frac{\absWconst^2}{2} - 2\sqrt{(1 + e)} \subg^2 t
- 4 \varepsilon_{n,m}(t) \epsilon$,
we thus see that
\begin{equation*}
  \P\left(\inf_{X \in \mc{S}_2} F(Z, X)
  \le \frac{\absWconst^2}{2} - 2 \sqrt{(1 + e)} \subg^2 t
  - 4 \varepsilon_{n,m}(t) \epsilon\right)
  \le \left(\frac{3}{\epsilon}\right)^{2 (2n + 1)}
  e^{-mt^2} + e^{-mt}.
\end{equation*}
Summarizing, we have the following lemma.
\begin{lemma}
  There exist numerical constants $C_1 \le 2 \sqrt{(1 + e)}$ and
  $C_2 \le 44$ such that
  \begin{align*}
    & \P\left(\inf_{X \in \mc{S}_2}
    F(Z, X)
    \le \frac{\absWconst^2}{2}
    - C_1 \subg^2 t
    - C_2 \subg^2
    \left(\max\left\{\sqrt{\frac{n}{m} + t},
    \frac{n}{m} + t\right\} \epsilon + \epsilon \right)
    \right) \\
    & \qquad\qquad\qquad\qquad\qquad\qquad\qquad\qquad
    \le
    \exp\left(-mt^2 + 2 (2n + 1)\log\frac{3}{\epsilon}\right)
    + e^{-mt}.
  \end{align*}
\end{lemma}

Rewriting this slightly by taking $\epsilon, t \lesssim
\frac{\absWconst^2}{\subg^2}$, we see that there exist numerical
constants $c_0 > 0$ and $c_1 < \infty$ such that
\begin{equation*}
  \P\left(\inf_{X \in \mc{S}_2}
  F(Z, X)
  \le \frac{\absWconst^2}{4}
  \right)
  \le \exp\left(- c_0 m \frac{\absWconst^4}{\subg^4}
  + c_1 n \log \frac{\subg^2}{\absWconst^2}\right).
\end{equation*}
Now, returning to inequality~\eqref{eqn:F-to-show-complex}, we see
we have that for $z_i = \sqrt{n}
a_i / \ltwo{a_i} \indics{\ltwo{a_i} \ge c \sqrt{n}}$ chosen as
in~\eqref{eqn:def-z-complex} and $c > 0$ chosen small enough so that
inequality~\eqref{eqn:truncated-a-good-growth-complex} holds,
we have
\begin{equation*}
  \inf_{X \in \mc{S}_2}
  F(Z, X)
  \ge
  \frac{\absWconst^2}{4}
\end{equation*}
with probability at least $1 - e^{-c_m
  \frac{\absWconst^4}{\subg^4} + c_1 n \log \frac{\subg^2}{\absWconst^2}}$.
Comparing with inequality~\eqref{eqn:F-to-show-complex} gives the theorem.

\subsubsection{Proof of Lemma~\ref{lemma:frob-rank-two}}
\label{sec:proof-frob-rank-two}

For $X = xx\cg - yy\cg$ that
\begin{align*}
  \lefteqn{\lfro{X}^2
    = \ltwo{x}^4 + \ltwo{y}^4 - 2 |\<x, y\>|^2} \\
  & = \ltwo{x}^4 + \ltwo{y}^4 - 2\left(
  \Real(\<x, y\>)^2 + \Imag(\<x, y\>)^2\right) \\
  & = \left(
  \sqrt{2 \norm{x}^4 + 2 \norm{y}^4}
  - 2 \sqrt{\Real(\<x, y\>)^2
    + \Imag(\<x, y\>)^2}\right)
  \left(
  \sqrt{2 \norm{x}^4 + 2 \norm{y}^4}
  + 2 \sqrt{\Real(\<x, y\>)^2
    + \Imag(\<x, y\>)^2}\right).
\end{align*}
By using that
\begin{equation*}
  \norms{x - e^{\imagunit \theta} y}^2
  = \norm{x}^2 + \norm{y}^2
  - 2\left(\cos \theta \cdot \Real(\<x, y\>)
  - \sin \theta \cdot \Imag(\<x, y\>)\right),
\end{equation*}
and that the concavity of $\sqrt{\cdot}$ implies
$\sqrt{2 a^4 + 2 b^4} \ge a^2 + b^2$,
the final expression above has lower bound
\begin{align*}
  & \left(
  \norm{x}^2 + \norm{y}^2
  - 2 \sqrt{\Real(\<x, y\>)^2
    + \Imag(\<x, y\>)^2}\right)
  \left(
  \norm{x}^2 + \norm{y}^2
  + 2 \sqrt{\Real(\<x, y\>)^2
    + \Imag(\<x, y\>)^2}\right)\\
  & ~~ = 
  \inf_\theta \norms{x - e^{\imagunit \theta} y}^2
  \cdot \sup_\theta \norms{x - e^{\imagunit \theta} y}^2.
\end{align*}

\subsubsection{Proof of Lemma~\ref{lemma:continuity-F-func-complex}}
\label{sec:proof-continuity-F-func-complex}

We have
\begin{align}
  \nonumber
  \lefteqn{m |F(Z, X) - F(Z, Y)|
    \le \sum_{i = 1}^m \left|\sum_{j = 1}^r
    \left(\lambda_j |\<u_j, z_i\>|^2 - \sigma_j |\<v_j, z_i\>|^2 \right)\right|}
  \\
  & \qquad ~ \le \sum_{i = 1}^m \left|\sum_{j = 1}^r
  \left(\lambda_j (|\<u_j, z_i\>|^2 - |\<v_j, z_i\>|^2) \right)\right|
  + \sum_{i = 1}^m \left|\sum_{j = 1}^r
  \left((\lambda_j - \sigma_j)|\<v_j, z_i\>|^2 \right)\right|.
  \label{eqn:bound-F-minus-F-complex}
\end{align}
We bound each of the terms in expression~\eqref{eqn:bound-F-minus-F-complex}
in term.
For the first, we note that for any complex $u, v, z \in \C^n$, we have
\begin{equation*}
  |\<u, z\>|^2
  - |\<v, z\>|^2
  \le \left|z\cg uu\cg z - z\cg vv\cg z
  + u\cg zz\cg v - v\cg zz\cg u\right|
  = \left|\<u - v, z\> \<z, u + v\>\right|,
\end{equation*}
as $u\cg zz\cg v - v\cg zz\cg u$ is purely imaginary.
As a consequence, we have
\begin{align*}
  \sum_{i = 1}^m \left|\sum_{j = 1}^r
  \lambda_j (|\<u_j, z_i\>|^2 - |\<v_j, z_i\>|^2) \right|
  & = \sum_{i = 1}^m \left|\sum_{j = 1}^r
  \lambda_j \<u_j u_j\cg - v_j v_j\cg, z_i z_i\cg\> \right| \\
  & \le 
  \sum_{j = 1}^r |\lambda_j|
  \sum_{i = 1}^m |\<z_i, u_j - v_j\>|
  |\<z_i, u_j + v_j\>| \\
  & \le \sum_{j = 1}^r |\lambda_j|
  \bigg(\sum_{i = 1}^m |\<z_i, u_j - v_j\>|^2
  \bigg)^\half
  \bigg(\sum_{i = 1}^m |\<z_i, u_j + v_j\>|^2
  \bigg) \\
  & \le 2 \sum_{j = 1}^r |\lambda_j|
  \opnorm{Z^T Z} \ltwo{u_j - v_j}
   \le 2 \lone{\lambda} \opnorm{Z^T Z}
  \norm{U - V}_{1,2},
\end{align*}
where we have used that $\ltwo{u_j} = \ltwo{v_j} = 1$ so
$\ltwo{u_j + v_j} \le 2$.
For the second term in inequality~\eqref{eqn:bound-F-minus-F-complex},
we have
\begin{align*}
  \sum_{i = 1}^m \left|\sum_{j = 1}^r
  \left((\lambda_j - \sigma_j) |\<v_j, z_i\>|^2 \right)\right|
  & \le \sum_{j = 1}^r |\lambda_j - \sigma_j| \sum_{i = 1}^m
  |\<v_j, z_i\>|^2
  \le \lone{\lambda - \sigma}
  \opnorm{Z^TZ}.
\end{align*}
Noting that for vectors $\lambda, \sigma \in \R^r$ we have
$\lone{\lambda} \le \sqrt{r} \ltwo{\lambda}$, we obtain the lemma.

\subsection{Proof of Proposition~\ref{proposition:non-noisy-initialization}:
  deterministic part}
\label{sec:proof-non-noisy-initialization-cond}

Our proof of Proposition~\ref{proposition:non-noisy-initialization}
eventually reduces to applying eigenvector perturbation results to the random
matrices $X\init$. To motivate our
approach, note that
\begin{equation}
  X\init \defeq \frac{1}{|\selected|}
  \sum_{i \in \selected} a_i a_i\cg 
  = I_n - \perploss(\epsilon) \cdot d\subopt d\subopt\cg 
  + \residual,
  \label{eqn:approximately-great-sum}
\end{equation}
where $\residual$ is a random error term that we will
show has small norm. From the
quantity~\eqref{eqn:approximately-great-sum} we can apply eigenvector
perturbation arguments to derive that the directional estimate $\what{d} =
\argmin_{d \in \sphere^{n-1}} d\cg  X\init d$ satisfies $\what{d} \approx
d\subopt$. This will hold so long as $\opnorm{\residual}$ is
small because there is substantial separation in the eigenvalues
of $I_n$ and $I_n - \perploss(\epsilon) d\subopt d\subopt\cg $.

With this goal in mind, we define two index sets that with high
probability surround $\selected$. Let
\begin{align*}
  \yesselected \defeq \left\{i \in [m]
  \mid |\<a_i, d\subopt\>|^2 < \frac{1-\epsilon}{2} \right\}
  ~~\text{and}~~
  \maybeselected \defeq \left\{i \in[m] \mid
  |\<a_i, d\subopt\>|^2 \leq \frac{1+\epsilon}{2} \right\}.
\end{align*}
We now define events $\event_1$ through $\event_5$, showing that conditional
on these five, the result of the proposition holds. These events roughly
guarantee ($\event_1$) that $\sum_{i = 1}^m a_i a_i\cg $ is well-behaved,
($\event_2$ and $\event_3$) that $\maybeselected \setminus \yesselected$ is
small, and ($\event_4$ and $\event_5$) that most of the vectors $a_i$ for
indices $i \in \selected$ are close enough to uniform on the subspace
perpendicular to $d\subopt$ that we have a good directional estimate. Now,
let $q \in [1, \infty]$ and let $1/p + 1/q = 1$, so that $p$ is its
conjugate. Recalling the definition of the error $\residual(\epsilon)$ in
Assumption~\ref{assumption:well-spread}.\eqref{item:conditional-id}, we
define
\begin{equation}
  \begin{split}
    \event_1 & \defeq \left\{
    \frac{1}{m} \opnorm{A\cg A} \in [1 - \epsilon, 1 + \epsilon] \right\},
    ~~
    \event_2 \defeq \left\{|\maybeselected| \le |\yesselected|
    + 2 \epsilon \smallprob m \right\}, \\
    \event_3 & \defeq \left\{|\yesselected| \ge \half m \probdir \right\},
    ~~
    \event_4 \defeq 
    \left\{ \opnormbigg{\frac{1}{m} \sum_{i \in \maybeselected \setminus \yesselected}
      a_i a_i\cg } \le 4 q \subgauss^2 (\smallprob \epsilon)^\frac{1}{p}
    \right\} \\
    \event_5 & \defeq
    \left\{\opnormbigg{
      \frac{1}{|\yesselected|}
      \sum_{i \in \yesselected} a_i a_i\cg 
      - \left(I_n - \perploss(\epsilon) d\subopt d\subopt\cg \right)}
    \le \opnorm{\residual(\epsilon)} + \epsilon \right\}.
  \end{split}
  \label{eqn:all-defined-events}
\end{equation}

We prove the result of the proposition when each $\event_i$ occurs.
Decompose the matrix $X\init$ into 
\begin{align*}
  X\init
  & = \underbrace{I_n - \perploss(\epsilon) d\subopt d\subopt\cg }_{\defeq Z_0}
  \\
  & ~~ +
  \underbrace{\left[\frac{1}{|\yesselected|} \sum_{i\in \yesselected} a_i a_i\cg - 
      (I_n - \perploss(\epsilon)d\subopt d\subopt\cg )\right]}_{\defeq Z_1}  
  + \underbrace{\frac{1}{|\yesselected|} \sum_{i \in \selected \setminus \yesselected} a_i a_i\cg }_{\defeq Z_2}
  - \underbrace{\left(\frac{1}{|\yesselected|} - \frac{1}{|\selected|}\right) 
    \sum_{i \in \selected} a_i a_i\cg }_{\defeq Z_3}.
\end{align*}
We bound the operator norms of $Z_1, Z_2, Z_3$ in turn.
On the event $\event_5$, 
we have
\begin{equation}
  \label{eqn:op-upper-bound-one}
  \opnorm{Z_1} \leq \opnorm{\Delta(\epsilon)} + \epsilon.
\end{equation}
We turn to the error matrix $Z_2$.
On the event $\event_1$, we evidently have
$\what{r} \in \ltwo{x\subopt} (1 \pm \epsilon)$ by definition
of $\what{r}^2 = \frac{1}{m} \ltwo{A x\subopt}^2$, so that
\begin{equation*}
  \yesselected \subset \selected \subset \maybeselected.
\end{equation*}
Using that the summands $a_i a_i\cg $ are all positive semidefinite,
we thus obtain the upper bound
\begin{equation}
  \label{eqn:op-upper-bound-two}
  \opnorm{Z_2}
  =  \frac{1}{|\yesselected|} \opnormbigg{
    \sum_{i\in \selected \setminus \yesselected} a_i a_i\cg }
  \leq \frac{1}{|\yesselected|} \opnormbigg{
    \sum_{i\in \maybeselected \setminus \yesselected} a_i a_i\cg }
  \stackrel{(i)}{\leq}
  \frac{1}{\probdir} \cdot 4 q \subgauss^2 (\smallprob \epsilon)^\frac{1}{p},
\end{equation}
where in inequality~$(i)$ we use that $2|\yesselected| \geq \probdir m$ on
$\event_3$ and $\opnorms{\sum_{i\in \maybeselected \setminus \yesselected} a_i a_i\cg } \leq 4 q
m\subgauss^2 (\smallprob \epsilon)^\frac{1}{p}$ on $\event_4$. Lastly, we
provide an upper bound on $\opnorm{Z_3}$.  Again using that $\yesselected \subset
\selected \subset \maybeselected$ on event $\event_1$, we have
\begin{equation*}
  \left|\frac{1}{|\yesselected|} - \frac{1}{|\selected|}\right|
  \leq \frac{1}{|\yesselected|} - \frac{1}{|\maybeselected|} 
  = \frac{|\maybeselected| - |\yesselected|}{|\yesselected||\maybeselected|}
  \leq \frac{2 \epsilon \smallprob }{\probdir^2 m}, 
\end{equation*}
where in the last inequality, we use that
$|\maybeselected| - |\yesselected| \leq 2 \epsilon \smallprob m$ 
on $\event_2$ and that $|\maybeselected| \geq |\yesselected| \geq \probdir m/2$ 
on $\event_3$. Thus, by the definition of $\event_1$, 
we have
\begin{equation}
  \label{eqn:op-upper-bound-three}
  \opnorm{Z_3} \leq
  \frac{2 \epsilon \smallprob}{\probdir^2}
  \opnorm{\frac{1}{m}\sum_{i=1}^m a_i a_i\cg } 
  \leq \frac{2 \epsilon(1+\epsilon) \smallprob}{\probdir^2}.
\end{equation}

Combining inequalities~\eqref{eqn:op-upper-bound-one},
\eqref{eqn:op-upper-bound-two} and~\eqref{eqn:op-upper-bound-three}
on the error matrices $Z_i$, the triangle inequality
gives
\begin{equation*}
  \opnorm{Z_1 + Z_2 + Z_3}
  \le \opnorm{\residual(\epsilon)} +
  \underbrace{\left[
      \epsilon
      + \frac{4 q \subgauss^2 (\smallprob \epsilon)^\frac{1}{p}}{
        \probdir}
      + \frac{2 \epsilon (1 + \epsilon) \smallprob}{
        \probdir^2}\right]}_{\defeq \errorterm(\epsilon)}. 
\end{equation*}
This implies equality~\eqref{eqn:approximately-great-sum} with error bound
$\opnorm{\residual} \le \opnorm{\residual(\epsilon)} +
\errorterm(\epsilon)$.  Recall the definition of $Z_0 = I_n -
\perploss(\epsilon) d\subopt d\subopt\cg $, which has smallest eigenvector $d\subopt$
and eigengap $\perploss(\epsilon)$.  Lastly, we simplify
$\errorterm(\epsilon)$ by a specific choice of $p$ and $q$ in the
definition~\eqref{eqn:all-defined-events} of $\event_4$. Without loss of
generality, we assume $\smallprob\epsilon < 1$ (recall
Assumption~\ref{assumption:well-spread} on $\smallprob$), and define $p = 1
+ \frac{1}{\log\frac{1}{\smallprob\epsilon}}$ and $q = 1 + \log \frac{1}{
  \smallprob\epsilon}$. Using
that for any $z < 0$ we have
$\exp(\frac{z}{1 - 1/z}) \le e^{z + 1}$,
we have
$(\smallprob \epsilon)^\frac{1}{p} \le e \smallprob \epsilon$,
allowing us to bound
$q (\smallprob \epsilon)^\frac{1}{p} \le
e (1 + \log\frac{1}{\smallprob \epsilon}) (\smallprob \epsilon)$
in the error term $\errorterm(\epsilon)$.


We now apply the eigenvector perturbation inequality
of Lemma~\ref{lemma:simple-sin}. Using that, for 
$\theta \in \R$,  $\ltwo{u- e^{\imagunit\theta} v}^2 \ge  2 - 
2|\<u, v\>| \ge 2 - 2 |\<u, v\>|^2$ for
$\ltwo{u} = \ltwo{v} = 1$, a minor rearrangement of 
Lemma~\ref{lemma:simple-sin} applied
to $X\init = Z_0 + \Delta$ for $\Delta
= Z_1 + Z_2 + Z_3$ yields
\begin{equation*}
  \dist \left(\hat{d}, \frac{1}{r} X\subopt\right) = 
  \inf_{\theta \in [0, 2\pi]} \ltwo{\hat{d} - e^{\imagunit\theta} d}
  \le
  \frac{\sqrt{2} (\opnorm{\residual(\epsilon)} + \errorterm(\epsilon))}{
    \hinge{\perploss(\epsilon)
      - (\opnorm{\residual(\epsilon)} + \errorterm(\epsilon))}}.
\end{equation*}
Finally, using that $x_0 = \what{r} \what{d}$ and defining
$r\subopt = \norm{x\subopt}$, we have by the triangle inequality
that
\begin{equation*}
  \dist(x_0, X\subopt)
  \le \what{r} \dist \left(\what{d}, \frac{1}{r}X\subopt\right) + |r\subopt - \what{r}| 
  \stackrel{(i)}{\le} \left[\sqrt{1 + \epsilon}
    \dist(\what{d}, d\subopt)
    + \epsilon\right] r\subopt, 
\end{equation*}
where inequality~$(i)$ uses $\event_1$, which we recall implies $(1 -
\epsilon) \ltwo{x\subopt}^2 \le \what{r}^2 \le (1 + \epsilon)
\ltwo{x\subopt}^2$, where $\epsilon \in [0, 1]$. This is the claimed
consequence~\eqref{eqn:no-noise-dist}
in Proposition~\ref{proposition:non-noisy-initialization}.

\subsection{Proof of Proposition~\ref{proposition:non-noisy-initialization}:
  high probability events}
\label{sec:proof-non-noisy-initialization-prob}

It remains to demonstrate that each of the events $\event_1, \ldots,
\event_5$ (recall definition~\eqref{eqn:all-defined-events}) holds with
high probability, to which we dedicate the remainder of this argument in
the next series of lemmas, each of which argues that one of the five
events occurs with high probability.
Before the statement of each lemma, we recall the corresponding
event whose high probability we wish to demonstrate.
Throughout, $c > 0$ and $C < \infty$ denote numerical constants
whose values may change.

We begin with $\event_1 \defeq \{
\frac{1}{m} \opnorms{A\cg A} \in [1 \pm \epsilon]\}$.
\begin{lemma}
  \label{lemma:non-noisy-initial-one}
  We have
  $\P(\event_1) \ge 1 - \exp(-c m \epsilon^2 / \subgauss^4)$
  for $m$ large enough that
  $m/n \ge \subgauss^4/(c\epsilon^2)$.
\end{lemma}
\begin{proof}
  Set $t = c \frac{\epsilon^2}{\subgauss^4}$ in
  Lemma~\ref{lemma:matrix-concentration}, noting that
  we must have $\subgauss^2 \ge 1$ because $\E[aa\cg ] = I_n$
  and $\opnorm{\E[aa\cg ]} \le \subgauss^2$ (recall the final
  part of Lemma~\ref{lemma:matrix-concentration}). Moreover,
  $\epsilon \in [0, 1]$ by assumption. Then
  taking $c$ small enough, once we have
  $\frac{n}{m} \le c \frac{\epsilon^2}{\subgauss^4}$
  we obtain the result.
\end{proof}

The event $\event_2 \defeq \{|\yesselected| \ge |\maybeselected| - 2 \smallprob \epsilon m\}$
likewise holds with high probability.
\begin{lemma}
  \label{lemma:non-noisy-initial-two}
  We have
  $\P(\event_2) \ge 1-\exp(-2 \epsilon^2 \smallprob^2 m)$.
\end{lemma}
\begin{proof}
  We always have that
  \begin{align*}
    \yesselected \subset \maybeselected ~~\text{and}~~\maybeselected \setminus \yesselected = \left\{i\in [m]: 
    \half (1-\epsilon) \leq |\<a_i, d\subopt\>|^2 \leq \half(1+\epsilon)\right\}.
  \end{align*}
  Therefore, the difference in cardinalities of
  $|\yesselected|$ and $|\maybeselected$ is
  \begin{equation*}
    \frac{1}{m} \left(|\maybeselected| - |\yesselected|\right) 
    = \frac{1}{m}\sum_{i=1}^m 
    \indic{1-\epsilon \leq 2|\<a_i, d\subopt\>|^2 \leq 1+\epsilon}.
  \end{equation*}
  The right hand side is an average of i.i.d.\ Bernoulli random variables
  with means bounded by $\smallprob \epsilon$ by
  Assumption~\ref{assumption:well-spread}\eqref{item:continuity-direction}.
  Hoeffding's inequality gives the result
  that $\P(|\maybeselected| > |\yesselected| + 2\smallprob m \epsilon)
  \le e^{-2\smallprob^2 \epsilon^2 m}$.
\end{proof}

\begin{lemma}
  \label{lemma:non-noisy-initial-three}
  The event 
  $\event_3 \defeq \{|\yesselected| \geq \half m \probdir\}$
  satisfies $\P(\event_3) \ge 1 - \exp(-\half m \probdir^2)$.
\end{lemma}
\begin{proof}
  As in Lemma~\ref{lemma:non-noisy-initial-two}, this result
  is immediate from Hoeffding's inequality.
  We have $|\yesselected| = \sum_{i=1}^m 
  \indics{|\<a_i, d\subopt\>|^2 \leq (1-\epsilon)/2}$,
  an i.i.d\ sum of Bernoullis with
  $\P(|\<a_i, d\subopt\>|^2 \le (1 - \epsilon)/2) \ge \probdir$ by
  Assumption~\ref{assumption:well-spread}\eqref{item:continuity-direction}.
  Hoeffding's inequality gives the result.
\end{proof}

Showing that events $\event_4$ and $\event_5$ in
Eq.~\eqref{eqn:all-defined-events} each happen with high probability
requires a little more work. We begin with $\event_4$, 
defined in terms of a conjugate pair $p, q \ge 1$ with
$1/p + 1/q = 1$, as
$\event_4 \defeq
\{\opnorms{\sum_{i \in \maybeselected \setminus \yesselected} a_i
  a_i\cg } \le 4 q \subgauss^2 (\smallprob \epsilon)^\frac{1}{p}\}$.
\begin{lemma}
  \label{lemma:non-noisy-initial-four}
  If
  $m/n > c^{-1} (\smallprob \epsilon)^{-\frac{2}{p}}$, then
  $\P(\event_4) \ge 1-\exp(-c m (\smallprob \epsilon)^\frac{2}{p})$.
\end{lemma}
\noindent
It is no loss of generality to assume that $\smallprob \epsilon
\le 1$ by Assumption~\ref{assumption:well-spread}, so
$(\smallprob \epsilon)^\frac{2}{p} \ge (\smallprob \epsilon)^2$.

\begin{proof}
  The $\{a_i\}_{i=1}^m$ are independent $\subgauss^2$-sub-Gaussian random
  vectors by Assumption~\ref{assumption:sub-gaussian-vector}, and for any
  random variable $B_i$ with $|B_i| \le 1$, which may depend on $a_i$, it is
  clear that the collection $\{B_i a_i\}_{i = 1}^m$ are mutually
  independent and still satisfy
  Definition~\ref{def:sub-gaussian-vector}. To that end,
  define the Bernoulli variables
  $B(a) = \indics{|\<a, d\subopt\>|^2 \in [\frac{1 - \epsilon}{2},
      \frac{1 + \epsilon}{2}]}$ (letting $B_i = B(a_i)
  = \indic{i \in \maybeselected \setminus \yesselected}$ for shorthand).
  Then
  Lemma~\ref{lemma:matrix-concentration}
  implies for a numerical constant $C < \infty$ that
  \begin{equation}
    \label{eqn:B-matrix-concentration}
    \P \left(\opnormbigg{\frac{1}{m}\sum_{i=1}^m a_i a_i\cg  B_i
      - \E[aa\cg  B(a)]} \ge
    C \subgauss^2
    \max\left\{\sqrt{\frac{n}{m} + t},
    \frac{n}{m} + t\right\}\right)
    \leq e^{-m t}.
  \end{equation}
  Now, note by H\"older's inequality that
  \begin{align*}
    \E[\<v, a\>^2 B(a)]
    & \le \E[\<v, a\>^{2q}]^\frac{1}{q}
    \P\left(|\<a, d\subopt\>|^2 \in \left[\frac{1 - \epsilon}{2},
      \frac{1 + \epsilon}{2}\right]\right)^\frac{1}{p} \\
    & \le (\subgauss^{2q} \Gamma(q + 1) e)^\frac{1}{q}
    (\smallprob \epsilon)^\frac{1}{p}
    \le q e^{1/q} \subgauss^2 (\smallprob \epsilon)^\frac{1}{p}
  \end{align*}
  where we have applied
  Assumption~\ref{assumption:well-spread}\eqref{item:continuity-direction}
  and Lemma~\ref{lemma:matrix-concentration} to bound
  $\E[\<a, d\subopt\>^{2q}]$.  Using the triangle inequality
  and substituting $t = \frac{1}{4C}(\smallprob \epsilon)^\frac{2}{p}$ into
  inequality~\eqref{eqn:B-matrix-concentration}, we find that
  for any $q \in (1, \infty)$ and $1/p + 1/q = 1$, if
  $\frac{n}{m} \le \frac{1}{4C}(\smallprob \epsilon)^{\frac{2}{p}}$ we have
  \begin{equation*}
    \P \left(\opnormbigg{\frac{1}{m}\sum_{i=1}^m a_i a_i\cg  B_i
      - \E[aa\cg  B(a)]} \ge
    q \subgauss^2
    (\smallprob \epsilon)^\frac{1}{p}\right)
    \leq \exp\left(-\frac{m(\smallprob \epsilon)^\frac{2}{p}}{4C}\right).
  \end{equation*}
  Applying the triangle inequality and that $1 + e^{1/q} < 4$ gives the result.
\end{proof}

The final high probability guarantee is the most complex, and
applies to the event $\event_5
\defeq \{\opnorms{\frac{1}{|\yesselected|}
  \sum_{i \in \yesselected} a_i a_i\cg 
  - (I_n - \perploss(\epsilon) d\subopt d\subopt\cg )}
\le \opnorms{\residual(\epsilon)} + \epsilon\}$.
\begin{lemma}
  \label{lemma:non-noisy-initial-five}
  Let $\event_3 = \{|\yesselected| \ge \half m \probdir\}$ as in
  Eq.~\eqref{eqn:all-defined-events}. Then
  \begin{equation*}
    c \frac{m}{n} \ge
    \frac{\subgauss^4\log^2 \probdir}{
      \probdir \epsilon^2}
  \end{equation*}
  implies $\P(\event_5 \mid \event_3 ) \geq 1 - \exp(-c
  \frac{m \probdir
    \epsilon^2}{\subgauss^4 \log^2 \probdir})$.
\end{lemma}
\begin{proof}
  For notational simplicity, define the following shorthand:
  \begin{equation*}
    E_{d\subopt}
    \defeq \E\left[aa\cg  \mid |\<a, d\subopt\>|^2 \le
      \frac{1 - \epsilon}{2}\right]
    = I_n - \perploss(\epsilon) d\subopt d\subopt\cg  + \residual(\epsilon),
  \end{equation*}
  where the equality uses
  Assumption~\ref{assumption:well-spread}\eqref{item:conditional-id}.
  The main idea of the proof is to show the following crucial fact: 
  define the new sub-Gaussian parameter
  $\tau^2 = \subgauss^2 \log \frac{e}{\probdir} \ge 1$.
  Then there exists a numerical constant $1 \le C < \infty$ such that for
  all $t \ge 0$,
  \begin{equation}
    \P\left(\opnormbigg{
      \frac{1}{|\yesselected|} \sum_{i \in \yesselected} a_i a_i\cg  - 
      E_{d\subopt}}
    \ge C \tau^2
    \max\left\{\sqrt{\frac{n}{|\yesselected|} + t},
    \frac{n}{|\yesselected|} + t\right\} \mid \yesselected\right)
    \leq \exp(- |\yesselected| t) .
    \label{eqn:cond-subgaussian}
  \end{equation}
  
  Suppose that the bound~\eqref{eqn:cond-subgaussian} holds.
  On the event $\event_3$, we have that
  $|\yesselected| \ge \half m \probdir$, and so choosing
  \begin{equation*}
    t = \frac{\epsilon^2}{4 C^2 \tau^4} < 1,
    ~~ \mbox{and letting} ~~
    \frac{m}{n}
    \ge \frac{2 C^2 \tau^4}{\probdir \epsilon^2}
  \end{equation*}
  yields that $\frac{n}{|\yesselected|} + t < 1$ on $\event_3$ and
  \begin{equation*}
    \P\left(\opnormbigg{
      \frac{1}{|\yesselected|} \sum_{i \in \yesselected} a_i a_i\cg  - 
      E_{d\subopt}}
    \ge \epsilon \mid \yesselected\right)
    \le 
    \exp\left(-\frac{m \probdir \epsilon^2}{8 C^2 \tau^4}\right).
  \end{equation*}
  By the definition
  of $E_{d\subopt}$ and the triangle inequality
  we have
  \begin{align*}
    & \P\left(\opnormbigg{
      \frac{1}{|\yesselected|} \sum_{i \in \yesselected} a_i a_i\cg  - (I_n - 
      \perploss\left(\epsilon\right) d\subopt d\subopt\cg )}
    \ge \opnorm{\residual(\epsilon)} + \epsilon \mid \yesselected \right) \\
    & \qquad\qquad\qquad\qquad\qquad\qquad\qquad\qquad~
    \le \P\left(\opnormbigg{
      \frac{1}{|\yesselected|} \sum_{i \in \yesselected} a_i a_i\cg  - 
      E_{d\subopt}}
    \ge \epsilon \mid \yesselected\right).
  \end{align*} 
  The lemma follows from the fact that $\event_3$ is measurable 
  with respect to the indices $\yesselected$.

  Now, we show the key inequality~\eqref{eqn:cond-subgaussian}. The main idea
  is to show that, conditioning on the set $\yesselected$, the distribution $\{a_i\}_{i
    \in \yesselected}$ is still conditionally independent and sub-Gaussian. To do so,
  we introduce a bit of (more or less standard) notation. For a random
  variable $X$, let $\mc{L}(X)$ denote the law of distribution of $X$. Using
  the independence of the vectors $a_i$, we have the fact that for any
  \emph{fixed} subset $\indset \subset [m]$, the collection $\{a_i\}_{i\in \indset}$
  is independent of $\{a_i\}_{i\not\in \indset}$. Therefore, using the definition
  that $\yesselected = \{i \in [m] : |\<a_i, d\subopt\>|^2 \le \frac{1 - \epsilon}{2}\}$, we
  have the key identity
  \begin{equation*}
    \mc{L}\bigg(\frac{1}{|\yesselected|} \sum_{i \in \yesselected} a_i a_i\cg  \mid \yesselected = \indset\bigg)
    \equivdist 
    \mc{L}\bigg(\frac{1}{|\indset|} \sum_{i\in \indset} a_i a_i\cg 
    \mid \max_{i\in \indset} |\<a_i, d\subopt\>|^2 \leq \frac{1-\epsilon}{2}\bigg).
  \end{equation*}
  This implies that, conditioning on $\yesselected = \indset$, the vectors $\{a_i\}_{i \in
    \indset}$ are still conditionally independent, and their conditional
  distribution is identical to the law $\mc{L}(a\mid |\<a, d\subopt\>|^2 \leq
  \frac{1-\epsilon}{2})$. The claimed inequality~\eqref{eqn:cond-subgaussian}
  will thus follow by the matrix concentration inequality in
  Lemma~\ref{lemma:matrix-concentration}, so long as
  we can demonstrate appropriate sub-Gaussianity of the
  conditional law
  $\mc{L}(a\mid |\<a, d\subopt\>|^2 \leq \frac{1 - \epsilon}{2})$.

  Indeed, let us temporarily assume that $a \mid |\<a, d\subopt\>|^2 \le \frac{1
    - \epsilon}{2}$ is $\tau^2$-sub-Gaussian, let $\mc{J}$ denote all
  subsets $\indset \subset [m]$ such that $|\indset| \ge \half m \probdir$,
  and define the shorthand $E_{d\subopt} = \E[aa\cg  \mid |\<a, d\subopt\>|^2 \le
    \frac{1 - \epsilon}{2}]$.  Then by summing over $\mc{J}$, we have on the
  event $\event_3$ that for a numerical constant $C < \infty$,
  \begin{align*}
    & \P\left(\opnormbigg{\frac{1}{|\yesselected|}
      \sum_{i \in \yesselected} a_i a_i\cg  - E_{d\subopt} }
    \ge C \tau^2 \max\left\{\sqrt{\frac{n}{|\yesselected|} + t},
    \frac{n}{|\yesselected|} + t\right\} \mid \event_3 \right) \\
    & = \sum_{\indset \in \mc{J}} \P\left(
    \opnormbigg{
      \frac{1}{|\indset|}
      \sum_{i \in \indset} a_i a_i\cg  - E_{d\subopt}}
    \ge C \tau^2 
    \max\left\{\sqrt{\frac{n}{|\indset|} + t},
    \frac{n}{|\indset|} + t\right\} \mid \yesselected = \indset \right)
    \P(\indset = \yesselected \mid \event_3) \\
    & = \sum_{\indset \in \mc{J}} \P\left(
    \opnormbigg{
      \frac{1}{|\indset|}
      \sum_{i \in \indset} a_i a_i\cg  - E_{d\subopt}}
    \ge C \tau^2 
    \max\left\{\sqrt{\frac{n}{|\indset|} + t},
    \frac{n}{|\indset|} + t\right\} \mid 
    \max_{i \in \indset} |\<a_i, d\subopt\>|^2
    \le \frac{1 - \epsilon}{2} \right)
    \P(\indset = \yesselected \mid \event_3) \\
    & \stackrel{(i)}{\le} \sum_{\indset \in \mc{J}} e^{-|\indset| t}
    \cdot \P(\indset = \yesselected \mid \event_3)
    \le e^{- \half m \probdir t},
  \end{align*}
  where inequality~$(i)$ is an application of
  Lemma~\ref{lemma:matrix-concentration}.
  This is evidently inequality~\eqref{eqn:cond-subgaussian} with
  appropriate choice of $\tau^2$.

  We thus show that $\mc{L}(a\mid |\<a, d\subopt\>|^2 \leq \frac{1-\epsilon}{2})$
  is subgaussian with parameter $\tau^2 = \subgauss^2
  \log\frac{e}{\probdir}$ by bounding the conditional moment generating
  function.
  Let $\lambda \in [1, \infty]$ and $\lambda'$ be conjugate,
  so that $1/ \lambda + 1/\lambda' = 1$.
  Then by H\"older's inequality, for any $v \in \sphere^{n-1}$ we have
  \begin{align*}
    \E \left[\exp\left(\frac{|\<a, v\>|^2}{\lambda \subgauss^2}\right) 
      \mid |\<a, v\>|^2 \leq \frac{1-\epsilon}{2}\right]
    &= \frac{\E \left[\exp\left(\frac{|\<a, v\>|^2}{\lambda \subgauss^2}\right) 
        \indic{|\<a, v\>|^2 \leq \frac{1-\epsilon}{2}}\right]}{
      \P\left(|\<a, v\>|^2 \leq \frac{1-\epsilon}{2}\right)}  \\
    & \leq \frac{(\E[\exp(\frac{|\<a, v\>|^2}{\subgauss^2})])^\frac{1}{\lambda}
      \P\left(|\<a, v\>|^2 \leq \frac{1-\epsilon}{2}\right)^\frac{1}{\lambda'}}{
      \P\left(|\<a, v\>|^2 \leq \frac{1-\epsilon}{2}\right)} \\
    & = \E\left[\exp\left(\frac{|\<a, v\>|^2}{\subgauss^2}\right)
      \right]^\frac{1}{\lambda} 
    \P\left(|\<a, v\>|^2 \leq \frac{1-\epsilon}{2}\right)^{-\frac{1}{\lambda}}
    \leq \left(\frac{e}{\probdir}\right)^\frac{1}{\lambda},
  \end{align*}
  where the final inequality uses the $\subgauss^2$-sub-Gaussianity
  of $a$.
  Set $\lambda = \log \frac{e}{\probdir}$ to see that
  conditional on $|\<a, d\subopt\>|^2 \le \frac{1 - \epsilon}{2}$,
  the vector $a$ is $\subgauss^2 \log \frac{e}{\probdir}$-sub-Gaussian,
  as desired.
\end{proof}

%% file: appendix-noise.tex

\section{Proofs for phase retrieval with outliers}
\label{sec:proofs-noisy}

In this section, we collect the proofs of the various results
in Section~\ref{sec:noise}.
Before providing the proofs, we state one inequality that we use frequently
that will be quite useful.  Let $W_i \in \{0, 1\}$ satisfy $W_i = 1$ if $i
\in \outliers$ and $W_i = 0$ otherwise, so that $W$ indexes the outlying
measurements.  Because $W_i$ are independent of the $a_i$ vectors and
$\sum_i W_i = \pfail m$, Lemma~\ref{lemma:matrix-concentration} implies
for all $t \ge 0$ that
\begin{equation}
  \P\left(\opnormbigg{\frac{1}{m} \sum_{i = 1}^m
    W_i a_i a_i^T
  - \pfail \E[a_i a_i^T]}
  \ge C \subgauss^2 \max\left\{\sqrt{\frac{n}{m} + t},
  \frac{n}{m} + t \right\}\right)
  \le e^{-m t}.
  \label{eqn:concentration-outlier-A}
\end{equation}

\subsection{Proof of Proposition~\ref{proposition:noise-stability}}
\label{sec:proof-noise-stability}

Recalling the set $\outliers$ of outlying indices, we evidently
have
\begin{align}
  f(x) - f(x\subopt)
  & = \frac{1}{m} \sum_{i = 1}^m
  |\<a_i, x\>^2 - \<a_i, x\subopt\>^2|
  + \frac{1}{m} \sum_{i \in \outliers}
  \left(|\<a_i, x\>^2 - \noise_i|
  - 
  |\<a_i, x\>^2 - \<a_i, x\subopt\>^2|\right)
  - f(x\subopt) \nonumber \\
  & = \frac{\lone{(Ax)^2 - (Ax\subopt)^2}}{m}
  + \frac{1}{m} \sum_{i \in \outliers}
  \left(|\<a_i, x\>^2 - \noise_i|
  - |\<a_i, x\>^2 - \<a_i, x\subopt\>^2|
  - |\<a_i, x\subopt\>^2 - \noise_i|\right) \nonumber \\
  & \ge \frac{\lone{(Ax)^2 - (Ax\subopt)^2}}{m}
  - \frac{2}{m} \sum_{i \in \outliers}
  |\<a_i, x\>^2 - \<a_i, x\subopt\>^2|.
  \label{eqn:outlier-stability-simple-lower}
\end{align}
Now, we note the trivial fact that
if we define $W_i = 1$ for $i \in \outliers$ and $W_i = 0$
for $i \in \inliers$, then
\begin{align*}
  \sum_{i \in \outliers}
  |\<a_i, x\>^2 - \<a_i, x\subopt\>^2|
  & = \sum_{i = 1}^m
  W_i |(x - x\subopt)^T a_i a_i^T (x + x\subopt)| \\
  & \le \opnorm{A^T \diag(W) A} \ltwo{x - x\subopt}
  \ltwo{x + x\subopt}.
\end{align*}
Inequality~\eqref{eqn:concentration-outlier-A}
shows that the matrix $\frac{1}{m}\sum_{i = 1}^m W_i a_i a_i^T$ is well
concentrated.

Now, let $\event$ denote the event that
$\frac{1}{m} \opnorms{A^T \diag(W) A} \le \pfail + C \subgauss^2
\sqrt{n/m + t}$, where $t$ is chosen so that $n/m + t \le 1$.
Returning to
inequality~\eqref{eqn:outlier-stability-simple-lower}, we
have
\begin{equation*}
  f(x) - f(x\subopt)
  \ge \frac{1}{m} \lone{(Ax)^2 - (Ax\subopt)^2}
  - 2 (\pfail + C \subgauss^2 \sqrt{n/m + t}) \ltwo{x - x\subopt}
  \ltwo{x + x\subopt}
\end{equation*}
for all $x \in \R^n$ with probability at least $1 - e^{-mt}$
by inequality~\eqref{eqn:concentration-outlier-A}.
We finish with the following lemma, which is a minor sharpening of
Theorem~2.4 of Eldar and Mendelson~\cite{EldarMe14} so that we have sharp concentration in all
dimensions $n$. We provide a proof for completeness in
Section~\ref{sec:proof-eldar-mendelson}.
\begin{lemma}
  \label{lemma:eldar-mendelson}
  Let $a_i$ be independent $\subgauss^2$-sub-Gaussian vectors, and define
  $\stabfunc(u, v) \defeq \E[|\<a, u\> \<a, v\>|]$ for $u, v \in \R^n$. Then
  there exist a numerical constants $c > 0$ and $C < \infty$ such that
  \begin{equation*}
    \frac{1}{m} \sum_{i = 1}^m |\<a_i, u\> \<a_i, v\>|
    \ge \stabfunc(u, v)
    - C \subgauss^2 \sqrt[3]{\frac{n}{m}} - \subgauss^2 t
    ~~ \mbox{for~all~} u, v \in \sphere^{n - 1}
  \end{equation*}
  with probability at least 
  $1 - e^{-c mt^2} - e^{-cm}$ when
  $m/n \ge C$.
\end{lemma}
\noindent
Noting that $|\<a_i, x\>^2 - \<a_i, x\subopt\>^2|
= |\<a_i, x - x\subopt\> \<a_i, x + x\subopt\>|$ and substituting
the result of the lemma into the preceding display, we have
\begin{equation*}
  f(x) - f(x\subopt)
  \ge \left(\stabfunc
  - 2 \pfail
  - C \subgauss^2 \sqrt[3]{\frac{n}{m}}
  - C \subgauss^2 t\right)
  \ltwo{x - x\subopt} \ltwo{x + x\subopt}
\end{equation*}
uniformly in $x$ with probability at least
$1 - 2 e^{-c m t^2} - e^{-c m}$.

\subsection{Proof of
  Proposition~\ref{proposition:good-direction-to-good-radius}}
\label{sec:proof-good-direction-to-good-radius}

We first state a lemma providing a deterministic bound on the errors of
the minimizing radius.
\begin{lemma}
  Let
  \begin{equation}
    \label{eqn:technical-condition-abstract}
    \delta = \frac{6 \ltwo{x\subopt}^2 \opnorm{A^T A}}{
      \sum_{i \in \inliers}
      \<a_i, x\subopt\>^2 - \sum_{i \in \outliers}
      \<a_i, x\subopt\>^2} \dist(\what{d}, \{\pm d\subopt\}).
  \end{equation}
  If $\delta \le 1$, then all minimizers of
  $G(\cdot)$ belong to the set $[1 \pm \delta] \ltwo{x\subopt}^2$.
\end{lemma}

Temporarily assuming the conclusions of the lemma, let us show that the
random quantities in the bound~\eqref{eqn:technical-condition-abstract}
are small with high probability. We apply
the matrix concentration inequality~\eqref{eqn:concentration-outlier-A}
to see that for a numerical constant $C < \infty$ and all
$t \in [0, 1 - \frac{n}{m}]$, we have
\begin{equation*}
  \frac{1}{m} \sum_{i \in \inliers}
  \<a_i, v\>^2
  \ge \left[(1 - \pfail)
    - C \subgauss^2 \sqrt{\frac{n}{m} + t}
    \right]
  ~~ \mbox{and} ~~
  \frac{1}{m} \sum_{i \in \outliers}
  \<a_i, v\>^2
  \le \left[\pfail + C \subgauss^2 \sqrt{\frac{n}{m} + t}
    \right] 
\end{equation*}
with probability at least $1 - e^{-mt}$ for all
vectors $v \in \sphere^{n - 1}$, and $\frac{1}{m} \opnorms{A^TA} \le
\subgauss^2(1 + C \sqrt{\frac{n}{m} + t})$ with the same probability. That
is, for $t \in [0, 1 - \frac{n}{m}]$ with probability at least $1 - 2
e^{-mt}$ we have that $\delta$ in
expression~\eqref{eqn:technical-condition-abstract} satisfies
\begin{equation*}
  \delta \le \frac{6 \subgauss^2}{
    1 - 2 \pfail - C \subgauss^2 \sqrt{\frac{n}{m} + t}}
  \dist(\what{d}, \{\pm d\subopt\}).
\end{equation*}
If we assume that $\frac{n}{m} \le c (1 - 2\pfail)^2 / \subgauss^4$
and replace $t$ with $c (1 - 2\pfail)^2 / \subgauss^4$ for small
enough constant $c$, we find
that
\begin{equation*}
  \delta \le \frac{C \subgauss^2}{1 - 2\pfail}
  \dist(\what{d}, \{\pm d\subopt\})
\end{equation*}
with probability at least $1 - 2 e^{-c m(1 - 2\pfail)^2 / \subgauss^4}$,
where $C$ is a numerical constant. This is our desired result.

\begin{proof}
  We define a few pieces of notation for shorthand.
  Let
  \begin{equation*}
    \what{\subgauss}^2
    \defeq \frac{1}{m} \opnorm{A^T A}
    ~~ \mbox{and} ~~
    \what{\subgauss}_{\outliers}^2
    \defeq \frac{1}{m} \sum_{i \in \inliers}
    \<a_i, x\subopt\>^2
    - \frac{1}{m} \sum_{i \in \outliers} \<a_i, x\subopt\>^2,
  \end{equation*}
  and define the functions
  $\what{g}(\delta) = G((1 + \delta) \ltwo{x\subopt}^2)$, 
  equivalently
  \begin{equation*}
    \what{g}(\delta) \defeq
    \frac{1}{m} \sum_{i = 1}^m
    |b_i - \ltwo{x\subopt}^2 \langle a_i, \what{d}\rangle (1 + \delta)|,
  \end{equation*}
  and a slightly more accurate counterpart
  \begin{equation*}
    g(\delta) \defeq
    \frac{1}{m} \sum_{i=1}^m \left|b_i - (1+\delta) \<a_i, x\subopt\>^2\right|= 
    \frac{1}{m} \sum_{i\not\in \inliers}
    \left|\<a_i, x\subopt\>^2 - (1+\delta) \<a_i, x\subopt\>^2\right|
    + \frac{1}{m} \sum_{i\in \outliers} 
    \left|b_i - (1+\delta) \<a_i, x\subopt\>^2 \right|.
  \end{equation*}
  Note that if $\delta$ minimizes $\what{g}(\delta)$, then
  $(1 + \delta) \ltwo{x\subopt}^2$ minimizes $G(r)$.
  By inspection we find that the subgradients
  of $g$ with respect to $\delta$ are
  \begin{equation*}
    \partial_\delta g(\delta)
    = \frac{1}{m}
    \sum_{i \in \inliers} \sgn(\delta) \<a_i, x\subopt\>^2
    - \frac{1}{m} \sum_{i \in \outliers}
    \sgn((1 + \delta) \<a_i, x\subopt\>^2 - b_i) \<a_i, x\subopt\>^2,
  \end{equation*}
  where $\sgn(t) = 1$ if $t > 0$, $\setminus 1$ if $t < 0$, and
  $\sgn(0) = [\setminus 1, 1]$. Evidently, for $\delta > 0$ we have
  $g'(\delta) \ge \what{\subgauss}_{\outliers}^2$
  and
  $g'(-\delta)
  \le -\what{\subgauss}_{\outliers}^2$,
  so that
  \begin{equation}
    \label{eqn:g-grows}
    g(\delta) \ge \what{\subgauss}_{\outliers}^2 |\delta| + g(0).
  \end{equation}

  Now, we consider the gaps between $\what{g}$ and $g$: for $\delta \in [\setminus 1,
    1]$, we have the gap
  \begin{align*}
    |g(\delta) - \what{g}(\delta)|
    & \le \frac{(1 + \delta) \ltwo{x\subopt}^2}{m}
    \sum_{i = 1}^m |\langle a_i, d\subopt\rangle^2 -
    \langle a_i, \what{d}\rangle^2| \\
    & \le \frac{(1 + \delta) \ltwo{x\subopt}^2}{m}
    \opnorm{A^TA} \ltwos{d\subopt - \what{d}}
    \ltwos{d\subopt + \what{d}}
    \le 4 \ltwo{x\subopt}^2 \what{\subgauss}^2
    \dist(d\subopt, \{\pm \what{d}\}),
  \end{align*}
  where we have used the triangle inequality
  and Cauchy-Schwarz.
  Thus we obtain
  \begin{equation*}
    \what{g}(\delta)
    - \what{g}(0)
    \ge g(\delta) - g(0)
    + \what{g}(\delta) - g(\delta)
    + g(0) - \what{g}(0)
    \ge \what{\subgauss}_{\outliers}^2 |\delta|
    - 6 \ltwo{x\subopt}^2 \what{\subgauss}^2 \dist(\what{d}, \{\pm d\subopt\}),
  \end{equation*}
  where we have applied inequality~\eqref{eqn:g-grows}. Rearranging,
  we have that if $\what{g}(\delta) \le \what{g}(0)$ we must have
  \begin{equation*}
    |\delta| \le \frac{6 \ltwo{x\subopt}^2 \what{\subgauss}^2}{
      \what{\subgauss}_{\outliers}^2} \dist(\what{d}, \{\pm d\subopt\}).
  \end{equation*}
  By convexity, any minimizer of $\what{g}$ must thus lie
  in the above region, which gives the result
  when we recall that minimizers $\delta$ of $\what{g}$
  are equivalent to minimizers $(1 + \delta) \ltwo{x\subopt}^2$ of $G$.
\end{proof}

\subsection{Proof of
  Proposition~\ref{proposition:noisy-good-initial-direction}}
\label{sec:proof-noisy-good-initial-direction}

We introduce a bit of notation before giving the proof proper.
Recall that $\outliers \subset [m]$ denotes the outliers, or failed
measurements, and $\inliers = [m] \setminus \outliers$ the set
of $i$ such that $b_i = \<a_i, x\subopt\>^2$ (the inliers).
Recalling the selected set of indices $\selected$, we define the shorthand
\begin{equation*}
  \inliersin \defeq \inliers \cap \selected
  ~~ \mbox{and} ~~
  \outliersin \defeq \outliers \cap \selected.
\end{equation*}
for the chosen inliers and outliers.



We decompose the matrix $X\init$ into four matrices, each of which we
control to guarantee that $X\init \approx I_n - c d\subopt {d\subopt}^T$ for some
constant $c$, thus guaranteeing $\what{d} \approx d\subopt$.  Let $\proj =
d\subopt {d\subopt}^T$ and $\projperp = I_n - d\subopt {d\subopt}^T$ be the projection
operator onto the span of $d\subopt$ and its orthogonal complement. Then
we may decompose the matrix $X\init$ into the four parts
\begin{equation}
  \label{eqn:x-init-outlier-expansion}
  \begin{split}
    X\init &= \underbrace{\frac{1}{m}
      \sum_{i=1}^m \proj a_i a_i^T \proj
      \indic{i \in \inliersin}}_{\defeq Z_0}
    + \underbrace{\frac{1}{m} \sum_{i=1}^m
      \left[\proj a_i a_i^T \projperp
        +  \projperp a_i a_i^T \proj\right]
      \indic{i \in \inliersin}}_{\defeq {Z_1}}\\
    & \qquad ~
    + \underbrace{\frac{1}{m} \sum_{i=1}^n \projperp a_i a_i^T \projperp
      \indic{i \in \inliersin}}_{\defeq Z_2}
    + \underbrace{\frac{1}{m} \sum_{i=1}^m a_i a_i^T
      \indic{i \in \outliersin}}_{\defeq Z_3}.
  \end{split}
\end{equation}
Let us briefly motivate this decomposition.  We expect that $Z_0$ should be
small because we choose indices $\selected$ by taking the smallest values of
$b_i$, which should be least correlated with $d\subopt$ (recall the $\proj =
d\subopt {d\subopt}^T$).  We expect $Z_1$ to be small because of the independence
of the vectors $\proj a_i$ and $\projperp a_i$ for Gaussian measurement
vectors, and $Z_3$ to be small because $\outliersin$ should be not too
large. This leaves $Z_2$, which (by Gaussianity) we expect to
be some multiple of $I_n - d\subopt {d\subopt}^T$, which will then allow
us to apply eigenvector perturbation guarantees using the eigengap
of the matrix $X\init$.

The rotational invariance of the Gaussian means that
it is no loss of generality to assume that $d\subopt = e_1$, the first
standard basis vector, so for the remainder of the argument we assume
this without further comment. This means that
we may decompose $a_i$ as $a_i = [a_{i,1} ~ a_{i, 2} ~ \cdots ~
a_{i, n}]^T = [a_{i,1} ~ a_{i, \setminus 1}^T]^T$, which we will do
without further comment for the remainder of the proof.

We now present four lemmas, each controlling one of the terms $Z_l$ in the
expansion~\eqref{eqn:x-init-outlier-expansion}.  We defer proofs of each
of the lemmas to the end of this argument.
We begin by considering the $Z_0$ term, which
(because $\proj$ is rank one) satisfies
\begin{equation*}
  Z_0 = \underbrace{\bigg(\frac{1}{m} \sum_{i = 1}^m \<a_i, d\subopt\>^2
    \indic{i \in \inliersin}\bigg)}_{\defeq z_0} d\subopt {d\subopt}^T.
\end{equation*}
Recalling the definition~\eqref{eqn:q-fail} of the
constants $\delta_q$ and $w_q$ in the statement of the proposition,
we have the following lemma.
\begin{lemma}
  \label{lemma:z0-bound}
  Define the random quantities $z_0 = \frac{1}{m} \sum_{i=1}^m \<a_i,
  d\subopt\>^2 \indic{i\in \inliersin}$ and $z_2 = \frac{1}{m}
  |\inliersin|$. Then for $t \in [0, 1]$,
  \begin{equation*}
    \P\left(z_2 \ge z_0 + (1 - 2\pfail)(1 - t)\delta_q \right) 
    \ge 1 -
    \exp\left(-\frac{1}{10} (1 - 2\pfail) m\right)
    - 2 \exp\left(-\frac{m}{4}\right)
    - \exp\left(-\frac{m t^2 \delta_q^2}{4 w_q^2}\right).
  \end{equation*}
\end{lemma}
\noindent
See Sec.~\ref{sec:proof-z0-bound} for a proof of the lemma.
We thus see that
it is likely that $z_0$ is substantially smaller than the
rough fraction of inlying indices selected.

We now argue that $Z_1$ is likely to be small because it is the
sum of products of independent vectors.
\begin{lemma}
  \label{lemma:z1-bound}
  For $t \ge 0$ we have
  \begin{equation*}
    \P\left(\opnorm{Z_1}
    \geq 2 \sqrt{\frac{n}{m}} + t
    \right)
    \leq \exp \left(-\frac{mt^2}{8}\right) + \exp\left(-\frac{m}{2}\right).
  \end{equation*}
\end{lemma}
\noindent
See Section~\ref{sec:proof-z1-bound} for a proof.
We can also show that $Z_2$ is well-behaved in the sense that it is
approximately a scaled multiple of $(I - d\subopt {d\subopt}^T)$.
\begin{lemma}
  \label{lemma:z2-bound}
  Let $z_2$ be the random quantity $z_2 \defeq
  \frac{1}{m}|\inliersin|$. There exists a numerical constant $C$ such that
  for $t \in [0, 1]$ we have
  \begin{equation*}
    \P \left(\opnorm{Z_2 - z_2 (I_n - d\subopt {d\subopt}^T)}
    \geq C\left(\sqrt{\frac{n}{m} + t}\right) \right)
    \le \exp(-mt).
  \end{equation*}
\end{lemma}
\noindent
See Section~\ref{sec:proof-z2-bound} for a proof of the lemma.

Finally, we control the size of the error matrix $Z_3$ in the
expansion~\eqref{eqn:x-init-outlier-expansion}, which
corresponds to the contribution of the outlying measurements
$a_i$ that are included in the initialization matrix $X\init$.
We provide two slightly different guarantees,
depending on the model (strength) of adversarial noise $\noise$ assumed.
\begin{lemma}
  \label{lemma:z3-bound}
  Define the random quantity
  $z_3 \defeq \frac{1}{m}|\outliers \cap \selected|
  \le \pfail$.
  There is a numerical constant $C$ such that the following hold.
  \begin{enumerate}[(i)]
  \item Under the independent noise model~\ref{model:full-independence}, 
    for all $t\in [0, 1]$,
    \begin{equation*}
      \P\left(\opnorm{Z_3 - z_3 I_n} \geq C\sqrt{\frac{n}{m} + t}\right)
      \leq \exp(-mt).
    \end{equation*}
  \item Under the adversarial noise model~\ref{model:partial-independence},
    for all $t\in [0, 1]$,
  \begin{equation*}
    \P\left(\opnorm{Z_{3}} \geq \pfail + C \sqrt{\frac{n}{m} + t}\right)
    \leq \exp(-mt).
  \end{equation*}
  \end{enumerate}
\end{lemma}
\noindent
See Section~\ref{sec:proof-z3-bound} for a proof.

With these four lemmas in hand, we can prove the result of the proposition
by applying the eigenvector perturbation
Lemma~\ref{lemma:simple-sin}. The
expansion~\eqref{eqn:x-init-outlier-expansion},
coupled with the four lemmas defining
the constants $z_0 = \frac{1}{m} \sum_{i \in \inliersin} \<a_i, d\subopt\>^2$,
$z_2 = \frac{1}{m} |\inliersin|$, and $z_3 = \frac{1}{m}
|\outliersin|$, guarantees that
\begin{align*}
  X\init = Z_0 + Z_1 + Z_2 + Z_3
  & = z_0 d\subopt {d\subopt}^T
  + z_2 (I_n - d\subopt {d\subopt}^T)
  + z_3 I_n
  + \Delta \\
  & = (z_2 + z_3) I_n - (z_2 - z_0) d\subopt {d\subopt}^T
  + \Delta,
\end{align*}
where the perturbation $\Delta \in \R^{n \times n}$ satisfies
\begin{equation*}
  \opnorm{\Delta}
  \le
  \opnorm{Z_1} + \opnorm{Z_2 - z_2(I_n - d\subopt {d\subopt}^T)}
  +
  \begin{cases} \opnorm{Z_3 - z_3 I_n} &
    \mbox{under model~\ref{model:full-independence}} \\
    \opnorm{Z_3} & \mbox{under model~\ref{model:partial-independence}}.
  \end{cases}
\end{equation*}
On the event that $z_2 > z_0$, the minimal eigenvector
of $(z_2 + z_3) I_n - (z_2 - z_0) d\subopt {d\subopt}^T$ is $d\subopt$
with eigengap $z_2 - z_0$.
Applying Lemma~\ref{lemma:simple-sin} gives that
$\what{d} = \argmin_{d \in \sphere^{n-1}} d^T X\init d$ satisfies
\begin{equation}
  \label{eqn:dist-d-hat-d-opt}
  2^{-\half} \dist(\what{d}, \{\pm d\subopt\})
  \le \sqrt{1 - \langle \what{d}, d\subopt \rangle^2}
  \le \frac{\opnorm{\Delta}}{
    \hinge{z_2 - z_0 - \opnorm{\Delta}}}.
\end{equation}

Applying Lemmas~\ref{lemma:z1-bound} and~\ref{lemma:z2-bound},
we have for some numerical constant $C < \infty$ that
\begin{equation*}
  \opnorm{Z_1} + \opnorm{Z_2 - z_2 (I_n - d\subopt {d\subopt}^T)}
  \le C \sqrt{\frac{n}{m} + t}
  ~~ \mbox{with~probability~} \ge 1 - e^{-mt} - e^{-m/2}
\end{equation*}
for any $t \ge 0$. We now consider the two noise models in turn.
Under Model~\ref{model:full-independence}, Lemma~\ref{lemma:z3-bound}
then implies that
$\opnorm{\Delta} \le C \sqrt{\frac{n}{m} + t}$ with probability
at least $1 - e^{-mt} - e^{-m/2}$.
Recalling Lemma~\ref{lemma:z0-bound}, we have for the constants
$w_q^2$ and $\delta_q > 0$ defined in the lemma that 
$z_2 \ge z_0 + \frac{1 - 2 \pfail}{2} \delta_q$ with probability at least
$1 - \exp(-c (1 - 2 \pfail) m) - \exp(-c m \delta_q^2 / w_q^2)$,
where $c > 0$ is a numerical constant. That is, the perturbatoin
inequality~\eqref{eqn:dist-d-hat-d-opt} implies that
under Model~\ref{model:full-independence}, we have with
probability at
least $1 - \exp(-m t) - \exp(-c(1 - 2 \pfail) m)
- \exp(-c m \delta_q^2 / w_q^2)$ that
\begin{equation*}
  2^{-\half} \dist(\what{d}, \{\pm d\subopt\})
  \le \frac{C \sqrt{n/m + t}}{
    \hinge{(1 - 2 \pfail) \delta_q - C \sqrt{n/m + t}}},
\end{equation*}
where $0 < c, C < \infty$ are numerical constants.
Under Model~\ref{model:partial-independence},
we can bound $\opnorm{Z_3}$ by
$\pfail + C \sqrt{n/m + t}$ with probability $1 - e^{-mt} - e^{-m/2}$
(recall Lemma~\ref{lemma:z3-bound}), so that with the same
probability as above, we have
\begin{equation*}
  2^{-\half}
  \dist(\what{d}, \{\pm d\subopt\})
  \le \frac{\pfail + C \sqrt{n/m + t}}{
    \hinge{(1 - 2\pfail) \delta_q - C \sqrt{n/m + t} - \pfail}}
  ~~ \mbox{under Model~\ref{model:partial-independence}}.
\end{equation*}
This is the proposition.

\subsubsection{Proof of Lemma~\ref{lemma:z0-bound}}
\label{sec:proof-z0-bound}

As noted earlier, it is no loss of generality to
assume that $d\subopt = e_1$, the first standard basis vector, so that using
our definitions of $z_0 = \frac{1}{m} \sum_{i \in \inliersin} \<a_i,
d\subopt\>^2$ and $z_2 = \frac{1}{m} |\inliersin|$, we have
\begin{equation}
  \label{eqn:clementines}
  z_2 - z_0 = \frac{1}{m}
  \sum_{i=1}^m (1 - a_{i, 1}^2) \indic{i\in \inliersin}.
\end{equation}
Given that we choose in indices $\selected$ by a median,
it is helpful to have the following median concentration result.
\begin{lemma}
  \label{lemma:quantile-concentration-gaussian}
  Let $\{W_i\}_{i=1}^m \simiid \normal(0, 1)$.
  Fix $p \in (0, 1)$ and choose $w_q \ge 0$ so that
  $q \defeq \P(W^2 \le w_q^2) = 2(1 - \Phi(w_q)) >
  p$. Then
  \begin{equation*}
    \P\left(\quantile_p \left(\{W_i^2\}_{i=1}^m\right) \geq w_q^2\right)
    \leq \exp\left(-\frac{m (q-p)^2}{2(q-p)/3 + 2q(1-q)}\right).
  \end{equation*}
\end{lemma}
\begin{proof}
  Note that
  $\quantile_p (\{W_i^2\}_{i=1}^m) \geq w_q^2$
  if and only if
  $\frac{1}{m} \sum_{i=1}^m \indic{W_i^2 \leq w_q^2} \leq p$.
  Using that $\var(\indic{W_i^2 \le w_q^2}) = q(1 - q)$,
  Bernstein's inequality applied to the i.i.d.\ sum
  $\sum_{i = 1}^m (\indics{W_i^2 \le w_q^2} - q)$ gives the result.
\end{proof}

We now control the median of the perturbed vector $b \in \R^m$.  Since we
have $|\outliers| \leq \pfail m$, we have deterministic result
\begin{equation*}
  \median\left(\{b_i\}_{i=1}^m\right) 
  \leq 
  \quantile_{\frac{1}{2(1-\pfail)}}
  \left(\{\<a_i, x\subopt\>^2\}_{i\in \inliers}\right),
\end{equation*}
so by upper bounding the right hand quantity we can upper
bound $\median(\{b_i\})$, which in turn allows us to control
$\selected$. 
By the definition of $w_q$ and $\qfail$ in
Eq.~\eqref{eqn:q-fail}, which satisfies $\delta = \qfail -
\frac{1}{2(1 - \pfail)} = \frac{1 - 2 \pfail}{4 (1 - \pfail)}$,
Lemma~\ref{lemma:quantile-concentration-gaussian} with $q = \qfail$ and the
fact that $|\inliers| = (1 - \pfail) m$ then implies
\begin{align}
  \lefteqn{\P\left(\median(\{b_i\}_{i=1}^m)
    \ge w_q^2 \ltwo{x\subopt}^2 \right)
    \le \P\left(\quant_{(2(1 - \pfail))^{-1}}(\{\<a_i, x\subopt\>^2\}_{
      i \in \inliers}) \ge w_q^2 \ltwo{x\subopt}^2\right)} \nonumber \\
  & \qquad ~ \le
  \exp\left(-\frac{m(1 - \pfail)\delta^2}{2\delta/3 + 2\delta(1-\delta)}\right)
  = \exp\left(-\frac{3(1 - 2\pfail) m}{4
    (2 + 6(1 - \delta))}\right)
  \le \exp\left(-\frac{3}{32}(1-2\pfail)m\right).
  \label{eqn:medians-usually-good}
\end{align}

\newcommand{\smallinliers}{\inliers_q}

We now consider the indices $i$ that are inliers for which
$\<a_i, d\subopt\>^2$ is small; again letting $w_q \ge 0$ be defined
as in the quantile~\eqref{eqn:q-fail}, we define
\begin{equation*}
  \smallinliers \defeq
  \left\{i \in \inliers \mid \<a_i, d\subopt\>^2 \le w_q^2\right\}
  = \left\{i \in \inliers \mid \<a_i, x\subopt\>^2 \le w_q^2 \ltwo{x\subopt}^2
  \right\}.
\end{equation*}
\begin{lemma}
  \label{lemma:small-inliers-good}
  Let the set of inliers $\smallinliers$ be defined
  as above, and let
  $\delta_q = 1 - \E[W^2 \mid W^2 \le w_q^2]$
  for $W \sim \normal(0, 1)$. Then for
  all $t \ge 0$ we have
  \begin{equation*}
    \P\left(\frac{1}{|\smallinliers|}
    \sum_{i \in \smallinliers} (1 - a_{i,1}^2)
    \le \delta_q - t
    \mid \smallinliers \right)
    \le \exp\left(-\frac{2 |\smallinliers| t^2}{
      w_q^2}\right).
  \end{equation*}
\end{lemma}
\noindent
We defer the proof of Lemma~\ref{lemma:small-inliers-good}, continuing on to
give the proof of Lemma~\ref{lemma:z0-bound}.

We now integrate out the conditioning in
Lemma~\ref{lemma:small-inliers-good}.
Recalling the definition~\eqref{eqn:q-fail}
of $w_q$ in terms of $\pfail$, we have that
$\qfail = \frac{3 - 2 \pfail}{4(1 - \pfail)}
> \half$ for $\pfail \in \openright{0}{1/2}$.
By Hoeffding's inequality we have
$\P(|\smallinliers| \le m/8) \le e^{-m/4}$ because
$|\inliers| \ge (1 - \pfail) m \ge m/2$, whence we obtain
\begin{align*}
  \P\left(\frac{1}{|\smallinliers|}
  \sum_{i \in \smallinliers} (1 - a_{i,1}^2) \le \delta_q
  - t \right)
  & \le \exp\left(-\frac{m t^2}{4 w_q^2}\right)
  \P(|\smallinliers| \ge m/8)
  + \P(|\smallinliers| \le m/8) \\
  & \le \exp\left(-\frac{m t^2}{4 w_q^2}\right)
  + \exp\left(-\frac{m}{4}\right).
\end{align*}
Using the notation $\delta_q$ in Lemma~\ref{lemma:small-inliers-good},
for $t \in [0, 1]$ we
define the event
\begin{equation*}
  \event \defeq \left\{\median(\{b_i\}_{i = 1}^m)
  \le w_q^2 \ltwo{x\subopt}^2,
  |\smallinliers| \ge \frac{m}{8},
  \frac{1}{|\smallinliers|}\sum_{i \in \smallinliers}
  (1 - a_{i,1}^2) \ge (1 - t) \delta_q
  \right\}.
\end{equation*}
We immediately find that
\begin{equation}
  \P(\event) \ge 1 - \exp\left(-\frac{1}{10} (1 - 2 \pfail) m\right)
  - 2 \exp\left(-\frac{m}{4}\right)
  - \exp\left(-\frac{m t^2 \delta_q^2}{4 w_q^2}\right)
\end{equation}
by the preceding display and inequality~\eqref{eqn:medians-usually-good}.
Recalling the set $\smallinliers = \{i\in \inliers: \<a_i, d\subopt\>^2 \leq
w_q^2\}$, we have on the event $\event$ that the selected inliers satisfy
$\inliersin \subset \smallinliers$ (because $\median(b) \le w_q^2
\ltwo{x\subopt}^2$).
Because of our selection mechanism
with $\inliersin$ as the smallest $b_i$ in the sample,
we have that
\begin{equation*}
  \frac{1}{|\inliersin|}
  \sum_{i \in \inliersin}
  (1 - a_{i,1}^2)
  \ge \frac{1}{|\smallinliers|} \sum_{i \in
    \smallinliers|} (1 - a_{i, 1}^2).
\end{equation*}
Moreover, on
the event $\event$ the rightmost sum is positive, and 
using that
$|\inliersin| \geq |\inliers| + |\selected| - m \geq (\half - \pfail)
m$, we obtain that on $\event$ we have
\begin{equation*}
  \frac{1}{m} \sum_{i \in \inliersin}
  (1- a_{i, 1}^2)
  \ge \frac{1- 2\pfail}{2 |\inliersin|}
  \sum_{i \in \inliersin} (1 - a_{i, 1}^2)
  \ge \frac{1- 2\pfail}{2 |\smallinliers|}
  \sum_{i \in \smallinliers} (1-a_{i, 1}^2)
  \ge \frac{1 - 2\pfail}{2} \delta_q.
\end{equation*}
Recalling expression~\eqref{eqn:clementines} thus gives
Lemma~\ref{lemma:z0-bound}.

\begin{proof-of-lemma}[\ref{lemma:small-inliers-good}]
  Fix any set of indices $\indset_0 \subset [m]$,
  and note that by Hoeffding's inequality for bounded random variables
  we have
  \begin{equation*}
    \P\left(\frac{1}{|\indset_0|}
    \sum_{i \in \indset_0}
    a_{i,1}^2 - \E[a_{i,1}^2 \mid a_{i,1}^2 \le w_q^2]
    \ge t
    \mid a_{i,1}^2 \le w_q^2
    ~ \mbox{for~} i \in \indset_0\right)
    \le \exp\left(-\frac{2 |\indset_0| t^2}{w_q^2}\right)
  \end{equation*}
  for $t \ge 0$.
  Recalling the definition
  $\delta_q \defeq 1 - \E[W^2 \mid W^2 \le w_q^2]$
  for $W \sim \normal(0, 1)$,
  we thus find that
  for any index set $\indset_0 \subset [m]$, we have
  \begin{align*}
    \lefteqn{\P\left(\frac{1}{|\smallinliers|}
      \sum_{i \in \smallinliers}
      (1 - a_{i,1}^2) \le \delta_q
      - t \mid \smallinliers = \indset_0\right)} \\
    & =
    \P\left(\frac{1}{|\indset_0|}
    \sum_{i \in \indset_0}
    (1 - a_{i,1}^2) \le \delta_q - t \mid
    a_{i,1}^2 \le w_q^2 ~ \mbox{for~} i \in \indset_0\right)
    \le \exp\left(-\frac{2|\indset_0| t^2}{w_q^2}\right).
  \end{align*}
  Noticing that the random vectors
  $\{a_i\}_{i \in \indset_0}$ are independent 
  of $\{a_i\}_{i \not \in \indset_0}$,
  we have for any measurable set $\mc{C} \subset \R$ that
  \begin{equation*}
    \P\left(\sum_{i \in \smallinliers} (1 - a_{i,1}^2)
    \in \mc{C} \mid \smallinliers = \indset_0\right)
    = \P\left(\sum_{i \in \indset_0}
    (1 - a_{i,1}^2) \in \mc{C}
    \mid a_{i,1}^2 \le w_q^2
    ~ \mbox{for~} i \in \indset_0\right).
  \end{equation*}
  Combining the preceding two displays
  yields Lemma~\ref{lemma:small-inliers-good}.
\end{proof-of-lemma}

\subsubsection{Proof of Lemma~\ref{lemma:z1-bound}}
\label{sec:proof-z1-bound}

By our assumption (w.l.o.g.) that $d\subopt = e_1$,
we have
\begin{equation}
  \label{eqn:z-1-to-ltwo}
  \opnorm{Z_1} = 
  \ltwo{\frac{1}{m} \sum_{i=1}^n a_{i, 1} a_{i, \setminus 1}
    \indic{i\in \inliersin}}
\end{equation}
Letting $\{a_i'\}_{i=1}^m \simiid \normal(0, I_n)$ be an independent
collection of vectors, the collections
$\{a_i\}_{i \in \inliers}$ and $\{\noise_i\}_{i\in \outliers}$
are independent, as are $a_{i,1}$ and $a_{i, \setminus 1}$
for each $i$. Thus, because $\inliersin \subset \inliers$, that for
any measurable set $\mc{C} \subset \R^{n-1}$ we have
\begin{align}
  \P \left(
  \sum_{i \in \inliersin} a_{i, 1} 
  a_{i, \setminus 1} \in \mc{C}
  \mid \{a_{i, 1}\}_{i\in \inliers},
  \{\noise_i\}_{i\in \outliers}\right)
  = \P \left( \sum_{i \in \inliersin} a_{i,1}
  a_{i, \setminus 1}' \in \mc{C}
  \mid \{a_{i, 1}\}_{i\in \inliers}, \inliersin\right). 
  \label{eqn:independent-a-copies}
\end{align}
Now, we use the standard result~\cite{Ledoux01} that if $f$ is an
$L$-Lipschitz function with respect to the $\ell_2$-norm, then for any
standard Gaussian vector $Z$ we have
$\P(f(Z) - \E[f(Z)] \ge t)
\le \exp(-\frac{t^2}{2 L^2})$ for all $t \ge 0$.
Thus, defining the random Lipschitz constant
$\what{L}^2 = \frac{1}{m} \sum_{i \in \inliers} a_{i,1}^2$ we obtain
\begin{equation*}
  \E\left[\ltwobigg{\frac{1}{m}
      \sum_{i = 1}^m a_{i,1} a_{i, \setminus 1}'
      \indic{i \in \inliersin} }^2 \mid \inliersin,
    \{a_{i,1}\}_{i \in \inliers}\right]
  = \frac{1}{m^2} \sum_{i \in \inliersin} a_{i,1}^2 (n - 1)
  \le \frac{n-1}{m} \what{L}^2,
\end{equation*}
and thus for $t \ge 0$
\begin{equation*}
  \P\left(\ltwobigg{\frac{1}{m}
    \sum_{i = 1}^m a_{i,1} a'_{i,\setminus 1} \indic{i \in \inliersin}
  } \ge
  \what{L} \sqrt{\frac{n-1}{m}} +
  t \mid \{a_{i,1}\}_{i \in \inliers}, \inliersin\right)
  \le \exp\left(-\frac{m t^2}{2\what{L}^2}\right).
\end{equation*}
Moreover, $\what{L}$ is $\{a_{i,1}\}_{i \in \inliers}$-measurable, and
$\P(\what{L} \ge 2) \le \exp(-m / 2)$, again by the Lipschitz
concentration of Gaussian random variables. We thus apply
Eq.~\eqref{eqn:independent-a-copies} and obtain
\begin{align*}
  \lefteqn{\P\left(
    \ltwobigg{\frac{1}{m} \sum_{i \in \inliersin}
      a_{i,1} a_{i, \setminus 1}} \ge
    2 \sqrt{\frac{n - 1}{m}} + t
    \right)} \\
  & \le \P \left(
  \ltwobigg{\frac{1}{m} \sum_{i \in \inliersin}
    a_{i,1} a_{i, \setminus 1}} \ge
  2 \sqrt{\frac{n - 1}{m}} + t
  \mid \what{L} \le 2 \right) \P(\what{L} \le 2)
  + \P(\what{L} \ge 2) \\
  & \le \exp\left(-\frac{mt^2}{8}\right) + \exp\left(-\frac{m}{2}\right).
\end{align*}
Recalling the equality~\eqref{eqn:independent-a-copies}
on $\opnorm{Z_1}$ shows that the previous display gives the lemma.

\subsubsection{Proof of Lemma~\ref{lemma:z2-bound}}
\label{sec:proof-z2-bound}

Our first observation is simply the definition
of the projection operator $\projperp$,
which gives
\begin{equation*}
  \opnorm{Z_{m, 2} - z_2 (I_n - d\subopt {d\subopt}^T)}  
  = \opnormbigg{\frac{1}{m}\sum_{i\in \inliersin}
    a_{i, \setminus 1} a_{i, \setminus 1}^T
    - z_2 I_{n - 1}}.
\end{equation*}
As in the proof of Lemma~\ref{lemma:z1-bound}, we let $\{a_i'\}_{i=1}^m
\simiid \normal(0, I_n)$ be an independent copy of the $a_i$.  The
independence of $\{a_i\}_{i\in \inliers}$ and $\{\noise_i\}_{i\in
  \outliers}$ imply that for any measurable set $\mc{C}$ we have
\begin{align*}
  \P \left(\sum_{i\in \inliersin} 
  (a_{i, \setminus 1} a_{i, \setminus 1}^T - z_2 I_{n - 1})
  \in \mc{C} \mid \{a_{i, 1}\}_{i\in \inliers}, 
  \{\noise_i\}_{i\in \outliers}\right)
  &= \P \left(\sum_{i \in \inliersin}
  \left(a'_{i, \setminus 1} {a'_{i, \setminus 1}}^T
  - I_{n - 1} \right) \in \mc{C}
  \mid \inliersin\right).
\end{align*}
As $\{a_{i, \setminus 1}'\}_{i=1}^m$ are i.i.d.\ Gaussian vectors,
Lemma~\ref{lemma:matrix-concentration} implies that for $t\in [0, 1]$:
\begin{equation*}
  \P \left(\opnorm{\frac{1}{m} \sum_{i=1}^m
    \left(a'_{i, \setminus 1} {a'_{i, \setminus 1}}^T - I_{n - 1}
    \right)\indic{i \in \inliersin}  } 
  \geq C \sqrt{\frac{n}{m} + t}
  \mid \inliersin\right) \leq \exp(-mt)
\end{equation*}
for some constant $C$. The claim of the lemma follows by integration
of the above inequality.

\subsubsection{Proof of Lemma~\ref{lemma:z3-bound}}
\label{sec:proof-z3-bound}

For notational convenience, let us denote the
selected outlying indices by $\outliersin = \outliers\cap
\selected$.

We begin our proof by considering Model~\ref{model:full-independence}, in
which case the proof is essentially identical to
Lemma~\ref{lemma:z2-bound} (see Sec.~\ref{sec:proof-z2-bound}).
Indeed, we have the identity
\begin{equation*}
  \opnorm{Z_3 - z_3 I_n}
  = \opnormbigg{\frac{1}{m}\sum_{i=1}^m 
    \left(a_i a_i^T - I_n\right)\indic{i\in \outliersin}
  },
\end{equation*}
and so following the completely parallel route of introducing
the independent sample $\{a_i'\}_{i = 1}^m \simiid \normal(0, I_n)$,
we have for any $t \ge 0$ that
\begin{align*}
  \P\left(\opnorm{Z_3 - z_3 I_n} \ge t\right)
  & = \E\left[\P\left(\opnorm{Z_3 - z_3 I_n} \ge t \mid
    \outliersin\right)\right] \\
  & = \E\left[\P\left(
    \opnormbigg{\frac{1}{m} \sum_{i = 1}^m
      (a_i' {a_i'}^T - I_n) \indic{i \in \outliersin}}
    \ge t \mid \outliersin
    \right)\right].
\end{align*}
Applying Lemma~\ref{lemma:matrix-concentration}
to the vectors $a_i' \indic{i \in \outliersin}$ gives the result.

In the case in which we assume Model~\ref{model:partial-independence},
we have the deterministic upper bound
\begin{equation*}
  \opnorm{Z_3}
  = \opnorm{\frac{1}{m} \sum_{i=1}^m a_i a_i^T \indic{i\in \outliersin}}
  \leq \opnorm{\frac{1}{m} \sum_{i=1}^m a_i a_i^T \indic{i\in \outliers}}.
\end{equation*}
As $a_i \indic{i \in \outliers}$ is an $O(1)$-sub-Gaussian vector,
Lemma~\ref{lemma:matrix-concentration}
implies
\begin{equation*}
  \P \left(\opnorm{\frac{1}{m} \sum_{i=1}^m 
    \left(a_{i} a_{i}^T - I_{n} \right)\indic{i \in \outliers}  } 
  \geq C\sqrt{\frac{n}{m} + t}
  \mid \outliersin \right) \leq \exp(-mt). 
\end{equation*}
Since $|\outliers|/m = \pfail$, the desired claim follows via the
triangle inequality.

%% file: appendix-generic.tex

\section{Technical results}

\subsection{Proof of Lemma~\ref{lemma:ez-square}}
\label{sec:proof-ez-square}

Let $p(z)$ denote the density of $z$ and
$p_c(z') = p(z') / \int_{-c}^c p(z) dz$ for $z' \in [-c, c]$.
Let $F_c(z) = \int_{-c}^{z} p_c(z') dz'$ be the CDF of $p_c$, and
define the mapping $T : [-c, c] \to [-c, c]$ by
\begin{equation*}
  T(u) = F_c^{-1}\left(\frac{u + c}{2c}\right)
  = z ~ \mbox{for~the~} z ~\mbox{s.t.}~
  F_c(z) = \frac{u + c}{2c}.
\end{equation*}
Evidently, for $Z \sim P_c$ we have
$T(Z) \sim \uniform[-c, c]$, and
by the symmetry and monotonicity properties of $p$ and $p_c$ we have
$|T(z)| \ge |z|$ for $z \in [-c, c]$. In particular,
letting $U \sim \uniform[-c ,c]$, we have
\begin{equation*}
  \var(U) = \E[U^2]
  = \E[T(Z)^2 \mid |Z| \le c]
  \ge \E[Z^2 \mid |Z| \le c].
\end{equation*}
Using that $\var(U) =
\frac{1}{2c} \int_{-c}^c u^2 du
= \frac{c^2}{3}$ gives the result.

\subsection{Proof of Lemma~\ref{lemma:eldar-mendelson}}
\label{sec:proof-eldar-mendelson}

The proof is an essentially standard concentration and covering number
argument, with a few minor wrinkles.
First, we note that
\begin{equation*}
  \var(|\<u, a_i\> \<v, a_i\>|)
  \le \E[|\<u, a_i\>\<v, a_i\>|^2]
  \stackrel{(i)}{\le} \sup_{u \in \sphere^{n-1}}
  \E[\<u, a_i\>^4]
  \le 2 e \subgauss^4,
\end{equation*}
where inequality~$(i)$ is Cauchy-Schwarz and the final inequality follows
Lemma~\ref{lemma:matrix-concentration}.  Thus, the lower tail Bernstein's
inequality (Lemma~\ref{lemma:one-sided-bernstein}) applied to the positive
random variables $|\<u, a_i\> \<v, a_i\>| \ge 0$ with variance bounded by $2
e \subgauss^4$ implies that
\begin{equation}
  \label{eqn:bernstein}
  \P\left(\frac{1}{m}
  \sum_{i = 1}^m |\<u, a_i\> \<v, a_i\>|
  - \stabfunc(u, v) \le -t\right)
  \le \exp\left(-\frac{m t^2}{4e \subgauss^4}\right).
\end{equation}
Define $Z_{u,v} \defeq \frac{1}{m} \sum_{i = 1}^m |\<u, a_i\> \<v, a_i\>|$
for shorthand.
Using that for vectors $w, x \in \R^n$ we have
$\sum_{i = 1}^m |\<a_i, w\> \<a_i, x\>|
\le \ltwo{A w} \ltwo{A x}$ by
Cauchy-Schwarz, we
see that for
any $u, v, u', v' \in \sphere^{n-1}$,
\begin{align*}
  |Z_{u,v} - Z_{u',v'}|
  & \le \frac{1}{m}
  \sum_{i = 1}^m |\<a_i, u\> \<a_i, v\>
  - \<a_i, u'\> \<a_i, v'\>| \\
  & \le \frac{1}{m}
  \sum_{i = 1}^m \left(|\<a_i, u - u'\> \<a_i, v\>|
  + |\<a_i, u\> \<a_i, v - v'\>|\right) \\
  & \le \frac{1}{m}
  \ltwo{A (u - u')} \ltwo{A v}
  + \frac{1}{m} \ltwo{A (v - v')} \ltwo{A u}
  \le \frac{\opnorm{A}^2}{m} \left(\ltwo{u - u'}
  + \ltwo{v - v'}\right).
\end{align*}
Now, let $\mc{N}$ be an $\epsilon$-cover of the sphere $\sphere^{n-1}$, which
we use to control $\inf_{u, v \in \sphere^{n-1}}
\{Z_{u, v} - \E[Z_{u, v}]\}$.
Using the previous display, we obtain by
the definition of an $\epsilon$-cover argument that
\begin{align*}
  \P\left(
  \inf_{u, v \in \sphere^{n-1}} \{Z_{u,v} - \stabfunc(u, v)\}
  \le -t \right)
  & \le
  \P\left(
  \min_{u, v \in \mc{N}}
  \{Z_{u, v} - \stabfunc(u, v)\}
  \le -t + \frac{2 \opnorm{A^TA}}{m} \epsilon \right).
\end{align*}
From Lemma~\ref{lemma:matrix-concentration}, we know that
$\opnorms{A^TA}$ is well-concentrated, and thus by
considering the event that $\opnorms{A^TA} \gtrsim m \subgauss^2$ that
for some numerical constant $C < \infty$ we have
\begin{align}
  \lefteqn{\P\left(
    \inf_{u, v \in \sphere^{n-1}} \{Z_{u,v} - \stabfunc(u, v)\}
    \le -t \right)} \nonumber \\
  & \le
  \P\left(
  \min_{u, v \in \mc{N}}
  \{Z_{u, v} - \stabfunc(u, v)\}
  \le -t + 2C\subgauss^2 \epsilon \right)
  + \P\left(\opnorm{A^TA} \ge C m \subgauss^2\right) \nonumber \\
  & \le
  \card(\mc{N})^2
  \exp\left(-\frac{c m \hinge{t - C \subgauss^2 \epsilon}^2}{\subgauss^4}
  \right)
  + \exp(-m),
  \label{eqn:covering-plus-bernstein}
\end{align}
where the first term comes from Bernstein's inequality~\eqref{eqn:bernstein}
and the second from Lemma~\ref{lemma:matrix-concentration}.

Now, let us assume that $\mc{N}$ is an $\epsilon$-cover of $\sphere^{n-1}$
with minimal cardinality, which by standard volume arguments~\cite[Lemma
  2]{Vershynin12} satisfies $N(\epsilon) \defeq \card(\mc{N}) \le (1 +
2/\epsilon)^n$ for $\epsilon > 0$.  Noting that $(1 + 2 / \epsilon)^{2n} \le
\exp(n / \epsilon)$ for $\epsilon \le \epsilon_0 \defeq .21398$, we may
replace inequality~\eqref{eqn:covering-plus-bernstein} with
\begin{equation*}
  \P\left(
  \inf_{u, v \in \sphere^{n-1}} \{Z_{u,v} - \stabfunc(u, v)\}
  \le -t \right)
  \le \exp\left(\frac{n}{\epsilon}
  - \frac{c m \hinge{t - C \subgauss^2 \epsilon}^2}{\subgauss^4}
  \right) + e^{-m}
\end{equation*}
valid for all $\epsilon \le \epsilon_0$.
Now, if we set
$\epsilon = \sqrt[3]{n/m}$ and
define $t = \subgauss^2 \wb{t} + (C + 1)\subgauss^2 \epsilon$, then
we find that if $m/n \ge \epsilon_0^{-3}$, we have
\begin{equation*}
  \P\left(
  \inf_{u, v \in \sphere^{n-1}} \{Z_{u,v} - \stabfunc(u, v)\}
  \le -\subgauss^2 \wb{t}
  - (C + 1) \subgauss^2 \sqrt[3]{\frac{n}{m}}
  \right)
  \le \exp\left(c m \wb{t}^2\right)
  + e^{-m}.
\end{equation*}
This is the desired result.

\subsection{Properties of Gaussian Random Variable}
\label{sec:properties-gaussians}

\begin{lemma}
  \label{lemma:conditional-expectation-truncated-gaussian}
  Let $W \sim \normal(0, 1)$. Then for all $c \in \R_+$
  \begin{equation*}
    \E\left[W^2 \mid W^2 \leq c \right] \le 1 - \half c \P(W^2 > c).
  \end{equation*}
\end{lemma}
\begin{proof}
  By the law of total probability, we have 
  \begin{equation*}
    \E[W^2] = \E[W^2 \mid W^2 \leq c] \P(W^2 \leq c)
    + \E[W^2 \mid W^2 > c] \P(W^2 > c).
  \end{equation*}
  Using that $\E[W^2] = 1$ yields
  \begin{align*}
    1 - \E[W^2 \mid W^2 \leq c]
    = \P(W^2 > c) \left(\E[W^2 \mid W^2 > c] -\E[W^2 \mid W^2 \leq c] \right).
  \end{align*}
  For $t \in [0, c]$, we have
  $\P(W^2 \in [c-t, c]) \leq \P(W^2 \in [0, t])$,
  so that for such $t$ we have
  %
  $\P(W^2 \geq c - t \mid W^2 \leq c) \leq \P(W^2 \leq t \mid W^2 \leq c)$, 
  or
  $\P(W^2 \geq t \mid W^2 \leq c) +  \P(W^2 \geq c-t \mid W^2 \leq c) \leq 1$. 
  Performing the standard change of variables to compute $\E[W^2 \mid W^2 \le c]$,
  we thus obtain
  \begin{align*}
    \E[W^2 \mid W^2 \leq c] &= \int_0^c \P(W^2 \geq t \mid W^2 \leq c)dt  \\
    &= \int_0^c \half\left[\P(W^2 \geq t \mid W^2 \leq c) +  \P(W^2 \geq c-t \mid W^2 \leq c)\right] dt  \\
    &\leq \int_0^c \half dt = \frac{c}{2}.
  \end{align*}
  Using that $\E[W^2 \mid W^2 > c] \geq c$ gives the lemma.
\end{proof}

\begin{lemma}
  \label{lemma:funny-rank-2-gaussian}
  Let $X \in \C^{n \times n}$ be rank-2 and Hermitian and
  $z = \frac{1}{\sqrt{2}} (\normal(0, I_n) + \imagunit
  \normal(0, I_n))$ be standard complex normal. Then
  $\E[|z\cg X z|] \ge \frac{2 \sqrt{2}}{\pi} \lfro{X}$.
\end{lemma}
\begin{proof}
  Without loss of generality, we assume that
  $\lfro{X} = 1$, so that $X = \lambda_1 u_1 u_1\cg +
  \lambda_2 u_2 u_2\cg$ for $\lambda_1^2 + \lambda_2^2 = 1$
  and $u_i$ orthonormal.
  Thus, if $Z_1, Z_2$ are standard normal, we have
  \begin{equation*}
    z\cg X z = \lambda_1 |\<z, u_1\>|^2
    + \lambda_2 |\<z, u_2\>|^2
    \stackrel{{\rm dist}}{=}
    \lambda_1 Z_1^2 + \lambda_2 Z_2^2.
  \end{equation*}
  Without loss of generality, we assume $\lambda_1 \ge 0$,
  so that by inspection we have for $\lambda = \lambda_1$ that
  $\E[|\lambda_1 Z_1^2 + \lambda_2 Z_2^2|]
  \ge \E[|\lambda Z_1^2 - \sqrt{1 - \lambda^2} Z_2^2|]$.
  Letting $f(\lambda) = \E[|\lambda Z_1^2 - \sqrt{1 - \lambda^2} Z_2^2|]$,
  we have that
  \begin{equation*}
    f'(\lambda) = \P\left(\frac{\lambda}{\sqrt{1 - \lambda^2}}
    Z_1^2 \ge Z_2^2\right)
    - \P\left(\frac{\lambda}{\sqrt{1 - \lambda^2}}
    Z_1^2 \le Z_2^2\right)
    = \begin{cases} > 0 & \mbox{if}~ \lambda > \frac{1}{\sqrt{2}} \\
      = 0 & \mbox{if~} \lambda = \frac{1}{\sqrt{2}} \\
      < 0 & \mbox{if~} \lambda < \frac{1}{\sqrt{2}}.
    \end{cases}
  \end{equation*}
  In particular, $f$ is quasi-convex and minimized at $\lambda =
  1 / \sqrt{2}$, and thus
  \begin{equation*}
    \frac{1}{\lfro{X}} \E[|z\cg X z|]
    \ge \frac{1}{\sqrt{2}} \E[|Z_1^2 - Z_2^2|]
    = \frac{2 \sqrt{2}}{\pi},
  \end{equation*}
  where the final equality is a calculation.
\end{proof}